\pdfoutput=1 

\documentclass[pdflatex,sn-mathphys-num]{sn-jnl-preprint}

\usepackage{graphicx}%
\usepackage{multirow}%
\usepackage{amsmath,amssymb,amsfonts}%
\usepackage{mathrsfs}%
\usepackage[title]{appendix}%
\usepackage[dvipsnames]{xcolor} 
\usepackage{textcomp}%
\usepackage{manyfoot}%
\usepackage{booktabs}%
\usepackage{algorithm}%
\usepackage{algorithmicx}%
\usepackage{algpseudocode}%
\usepackage{listings}%

\usepackage[nameinlink,capitalize]{cleveref}
\definecolor{darkgreen}{rgb}{0,0.6,0}
\usepackage{epstopdf}
\usepackage{amsopn}

\usepackage{caption}
\usepackage{subcaption}
\usepackage{mathtools} 
\usepackage{bm,bbm} 
\usepackage{stmaryrd} 
\usepackage{dsfont} 

\usepackage{array}
\newcolumntype{M}[1]{>{\centering\arraybackslash}m{#1}} 
\usepackage{tcolorbox}

\usepackage{tikz-cd} 
\usetikzlibrary{arrows.meta,calc,positioning,fit,backgrounds,shapes.geometric}
\usepackage{adjustbox}

\let\originalleft\left 
\let\originalright\right
\renewcommand{\left}{\mathopen{}\mathclose\bgroup\originalleft}
\renewcommand{\right}{\aftergroup\egroup\originalright}

\newcommand{\N}{\mathbb{N}}
\newcommand{\Z}{\mathbb{Z}}
\newcommand{\R}{\mathbb{R}}
\newcommand{\T}{\mathbb{T}}
\newcommand{\C}{\mathbb{C}}

\newcommand{\D}{\mathbb{D}}

\newcommand{\cL}{\mathcal{L}}
\newcommand{\cF}{\mathcal{F}}
\newcommand{\cH}{\mathcal{H}}
\newcommand{\cX}{\mathcal{X}}

\newcommand{\cC}{\mathcal{C}}
\newcommand{\cS}{\mathcal{S}}
\newcommand{\cG}{\mathcal{G}}

\newcommand{\cR}{\mathcal{R}}
\newcommand{\cE}{\mathcal{E}}

\newcommand{\cK}{\mathcal{K}}
\newcommand{\cQ}{\mathcal{Q}}
\newcommand{\cV}{\mathcal{V}}
\newcommand{\cU}{\mathcal{U}}

\newcommand{\cM}{\mathcal{M}}
\newcommand{\sL}{\mathscr{L}}

\newcommand{\sR}{\mathscr{R}}

\newcommand{\sP}{\mathscr{P}}

\newcommand{\sJ}{\mathscr{J}}

\newcommand{\sfA}{\mathsf{A}}

\newcommand{\sfD}{\mathsf{D}}
\newcommand{\sfE}{\mathsf{E}}
\newcommand{\sfF}{\mathsf{F}}
\newcommand{\sfG}{\mathsf{G}}

\newcommand{\sfL}{\mathsf{L}}
\newcommand{\sfM}{\mathsf{M}}

\newcommand{\sfR}{\mathsf{R}}

\newcommand{\sfd}{\mathsf{d}}

\newcommand{\pOmega}{{\partial\Omega}}
\newcommand{\dtn}{\Lambda_{\gamma}}
\newcommand{\dtnz}{\Lambda_{\gamma_0}}
\newcommand{\ntd}{\mathscr{R}_{\gamma}}
\newcommand{\ntdz}{\mathscr{R}_{\gamma_0}}

\newcommand{\tntd}{\widetilde{\mathscr{R}}_{\gamma}}
\newcommand{\tntdz}{\widetilde{\mathscr{R}}_{\gamma_0}}

\newcommand{\BV}{\mathsf{BV}}

\DeclarePairedDelimiterX{\iptemp}[2]{\langle}{\rangle}{#1, #2}
\newcommand{\ip}{\iptemp}
\DeclarePairedDelimiterX{\normtemp}[1]{\lVert}{\rVert}{#1}
\newcommand{\norm}{\normtemp}
\DeclarePairedDelimiterX{\abstemp}[1]{\lvert}{\rvert}{#1}
\newcommand{\abs}{\abstemp}
\DeclarePairedDelimiterX{\trtemp}[1]{(}{)}{#1}

\DeclarePairedDelimiterX{\SEtemp}[2]{(}{)}{#1, #2}

\DeclarePairedDelimiterX{\SGtemp}[1]{(}{)}{#1}

\DeclareRobustCommand{\rchi}{{\mathpalette\irchi\relax}}
\newcommand{\irchi}[2]{\raisebox{\depth}{$#1\chi$}} 

\newcommand{\set}[2]{{\left\{ #1 \,\middle|\, #2 \right\}}}
\newcommand{\defeq}{\coloneqq} 
\newcommand{\eqdef}{\eqqcolon} 

\DeclarePairedDelimiterX{\floor}[1]{\lfloor}{\rfloor}{#1} 
\DeclarePairedDelimiterX{\ceil}[1]{\lceil}{\rceil}{#1} 

\DeclareMathOperator{\Id}{Id} 
\newcommand{\HS}{\mathrm{HS}} 
\DeclareMathOperator{\image}{Ran} 
\DeclareMathOperator{\domain}{Dom} 
\DeclareMathOperator{\Int}{Int} 

\newcommand{\comp}{\textsf{c}} 
\newcommand{\Law}{\operatorname{Law}} 
\newcommand{\normal}{\mathcal{N}} 

\newcommand{\Unif}{\mathrm{Unif}}

\newcommand{\diff}{d}

\newcommand{\dd}[1]{\,\diff{#1}}

\newcommand{\lap}{\Delta} 
\newcommand{\onebm}{\mathds{1}} 

\DeclareMathOperator{\supp}{supp} 

\newcommand{\al}{\alpha}
\newcommand{\ep}{\varepsilon}

\def\qfa{\quad\text{for all}\quad}

\def\qa{\quad\text{and}\quad}

\def\qw{\quad\text{where}\quad}

\def\qon{\quad\text{on}\quad}

\def\qif{\quad\text{if}\quad}

\newcommand{\slot}{{\,\cdot\,}}
\newcommand{\sfit}[1]{\textsf{\small{#1}}} 

\usepackage{enumitem}
\makeatletter
\newcommand{\myitem}[1]{%
	\item[#1]\protected@edef\@currentlabel{#1}%
}
\makeatother

\usepackage{titlesec}
\newcommand{\mysecspace}{\hskip.5em}
\titlelabel{\thetitle.\mysecspace}

\colorlet{siaminlinkcolor}{green!50!black}
\colorlet{siamexlinkcolor}{red!50!black}
\colorlet{siamreviewcolor}{black!50}
\hypersetup{allcolors=siaminlinkcolor,urlcolor=siamexlinkcolor}

\theoremstyle{thmstyleone}%
\newtheorem{theorem}{Theorem}[section]
\newtheorem{proposition}[theorem]{Proposition}%
\newtheorem{corollary}[theorem]{Corollary}%
\newtheorem{lemma}[theorem]{Lemma}%

\theoremstyle{thmstyletwo}%
\newtheorem{remark}{Remark}%

\theoremstyle{thmstylethree}%
\newtheorem{definition}{Definition}%
\newtheorem{assumption}{Assumption}%
\newtheorem{condition}{Condition}%

\numberwithin{equation}{section}
\numberwithin{theorem}{section}

\crefname{hypothesis}{Hypothesis}{Hypotheses}
\crefname{assumption}{Assumption}{Assumptions}
\crefname{condition}{Condition}{Conditions}
\crefname{fact}{Fact}{Facts}
\crefname{notation}{Notation}{Notation}
\crefname{problem}{Problem}{Problems}

\ifpdf
\DeclareGraphicsExtensions{.eps,.pdf,.png,.jpg,.jpeg}
\else
\DeclareGraphicsExtensions{.eps}
\fi
\graphicspath{ {./figures/} }

\ifpdf
\hypersetup{
	pdftitle={Extension and neural operator approximation of the electrical impedance tomography inverse map},
	pdfauthor={M. V. de Hoop, N. B. Kovachki, M. Lassas, and N. H. Nelsen}
}
\fi

\begin{document}

\title[Extension and approximation of the EIT inverse map]{Extension and neural operator approximation of the electrical impedance tomography inverse map}

\author[1]{\fnm{Maarten V.} \spfx{de} \sur{Hoop}}\email{mdehoop@rice.edu}

\author[2]{\fnm{Nikola B.} \sur{Kovachki}}\email{nkovachki@nvidia.com}

\author[3]{\fnm{Matti} \sur{Lassas}}\email{matti.lassas@helsinki.fi}

\author*[4,5,6]{\fnm{Nicholas H.} \sur{Nelsen}}\email{nnelsen@oden.utexas.edu}

\affil[1]{\normalsize \orgdiv{Simons Chair in Computational and Applied Mathematics and Earth Science}, \orgname{Rice University}, \orgaddress{\city{Houston}, \state{TX} \postcode{77005}, \country{USA}}}

\affil[2]{\normalsize \orgdiv{NVIDIA AI}, \orgname{NVIDIA Corporation}, \orgaddress{\city{Santa Clara}, \state{CA} \postcode{95051}, \country{USA}}}

\affil[3]{\normalsize \orgdiv{Department of Mathematics and Statistics}, \orgname{University of Helsinki}, \orgaddress{\street{P.O. Box 68 FI-00014}, \city{Helsinki}, \country{Finland}}}

\affil[4]{\normalsize \orgdiv{Department of Mathematics}, \orgname{Massachusetts Institute of Technology}, \orgaddress{\city{Cambridge}, \state{MA} \postcode{02139}, \country{USA}}}

\affil[5]{\normalsize \orgdiv{Department of Mathematics}, \orgname{Cornell University}, \orgaddress{\city{Ithaca}, \state{NY} \postcode{14853}, \country{USA}}}

\affil[6]{\normalsize \orgdiv{Oden Institute for Computational Engineering and Sciences}, \orgname{The University of Texas at Austin}, \orgaddress{\city{Austin}, \state{TX} \postcode{78712}, \country{USA}}}

\abstract{This paper considers the problem of noise-robust neural operator approximation for the solution map of Calder\'on's inverse conductivity problem. In this continuum model of electrical impedance tomography (EIT), the boundary measurements are realized as a noisy perturbation of the Neumann-to-Dirichlet map's integral kernel. The theoretical analysis proceeds by extending the domain of the inversion operator to a Hilbert space of kernel functions. The resulting extension shares the same stability properties as the original inverse map from kernels to conductivities, but is now amenable to neural operator approximation. Numerical experiments demonstrate that Fourier neural operators excel at reconstructing infinite-dimensional piecewise constant and lognormal conductivities in noisy setups both within and beyond the theory’s assumptions. The methodology developed in this paper for EIT exemplifies a broader strategy for addressing nonlinear inverse problems with a noise-aware operator learning framework.
}



\maketitle

\setcounter{tocdepth}{3}
\tableofcontents

\clearpage

\section{Introduction}\label{sec:intro}
\emph{Electrical impedance tomography} (EIT) is a canonical nonlinear and severely ill-posed inverse boundary value problem. Let $ \Omega\subset \R^d $ be a bounded domain representing the medium of interest and $ \pOmega $ be its smooth boundary. In EIT, the physical model is the linear elliptic partial differential equation (PDE)
\begin{align}\label{eqn:elliptic}
\begin{split}
    -\nabla\cdot(\gamma \nabla u)&=0 \ \text{ in } \Omega\,, \\ 
    \gamma\dfrac{\partial u}{\partial \mathsf{n}}&=g \ \text{ on } \pOmega\,.
\end{split}
\end{align}
In the context of medical imaging, the solution $u$ is the electric potential, $\gamma$ is the unknown interior electrical conductivity, $g$ is the applied boundary current, and $\partial/\partial\mathsf{n}$ is the outward normal derivative. The inverse problem is to recover $ \gamma $ inside $\Omega$ given noisy measurements on $\pOmega$. Mathematically, these measurements are encoded by the \emph{Neumann-to-Dirichlet} (NtD) map $\ntd$ assigning to each applied boundary current $g$ the resulting boundary voltage $u|_{\pOmega}$. This leads to the so-called \emph{Calder\'on problem}.

Applications of EIT include lung monitoring, stroke detection, non-destructive testing, and geophysical exploration. Traditional direct reconstruction methods are grounded in PDE analysis and have well-understood stability properties. However, they can rely on strong regularization that blurs sharp inclusions and small scale structure in the reconstructions. Iterative methods are a local alternative based on optimization, but they can be too slow for real-time use. Moreover, both classes of methods are usually applied independently to each set of measurements and do not exploit historical information. Obtaining EIT reconstructions that are simultaneously accurate, noise-robust, and computationally efficient remains challenging, especially for highly heterogeneous or discontinuous conductivities.

\emph{Operator learning} offers a distinct, amortized viewpoint. Instead of repeatedly running the classic inverse solver for each new set of measurements from the map $\ntd$, one learns a single approximation of the inversion operator $\ntd\mapsto\gamma $ from training data using a suitable network architecture. Once trained, this map can be rapidly evaluated on new observations. It implicitly encodes rich prior information derived from historical datasets. The present paper develops and analyzes neural operator approximations for EIT in dimension $d\geq 2$. The theoretical analysis establishes an approximation theorem that is valid for noisy inputs. Numerically, the paper demonstrates that appropriately designed and trained neural operators can deliver highly accurate reconstructions of discontinuous conductivities given noisy measurements; see \cref{fig:compare_dbar} for a preview.

\Cref{sec:intro_contrib} presents the contributions of this paper and gives a high-level overview of the paper's main result and its proof. \Cref{sec:intro_lit} briefly reviews relevant literature. \Cref{sec:intro_outline} outlines the structure of the remainder of the paper.

\subsection{Contributions and proof overview}\label{sec:intro_contrib}
In the setting of Calder\'on's problem and EIT, this paper blends operator learning of inverse maps with the rich stability theory of nonlinear inverse problems. Our main contributions are as follows.

\begin{enumerate}[label=(C\arabic*),leftmargin=2.5\parindent,topsep=1.67ex,itemsep=0.5ex,partopsep=1ex,parsep=1ex]
    \item \label{contr:extend} \textbf{Extension of inverse map.}\indent
    We establish that the EIT inversion operator---usually defined between subsets of two Banach spaces---also maps between subsets of two Hilbert spaces, and moreover, can be extended to the whole input Hilbert space with a controlled modulus of continuity.

    \item \label{contr:approx} \textbf{Approximation of inverse map.}\indent
    Building on Contribution~\ref{contr:extend}, we prove that there exists a Fourier neural operator (FNO) \cite{li2020fourier} that uniformly approximates the EIT inverse map over a compact set of admissible conductivities and noisy NtD map data perturbations.

    \item \label{contr:numeric} \textbf{Numerical reconstructions.}\indent
    We numerically demonstrate that trained FNOs can rapidly and accurately reconstruct two and three phase discontinuous conductivities and lognormal conductivities from highly noisy NtD map integral kernel functions, in regimes both within and beyond the scope of the theoretical assumptions required by Contribution~\ref{contr:approx}.
\end{enumerate}

Along the way to Contributions~\ref{contr:extend}~and~\ref{contr:approx}, we develop several technical results that may be of independent interest. In dimension $d=2$, we work with an admissible set $\cX$ of conductivities bounded in $L^\infty$ and in total variation norm. This set contains certain piecewise constant functions, which makes our results widely applicable. We obtain quantitative logarithmic stability estimates on the set $\cX$, while accommodating boundary data not just in the form of NtD maps, but also in the form of integral kernel functions corresponding to these NtD maps. We also show that $\cX$ is compact in the $L^p(\Omega)$ topology. On the machine learning side, we generalize neural operators such as FNO to accommodate (\emph{i}) functions defined on compact manifolds instead of just bounded domains and (\emph{ii}) input and output domains of different spatial dimension.

\begin{figure}[tb]
    \centering
    \captionsetup{skip=10pt}
    \includegraphics[width=\textwidth]{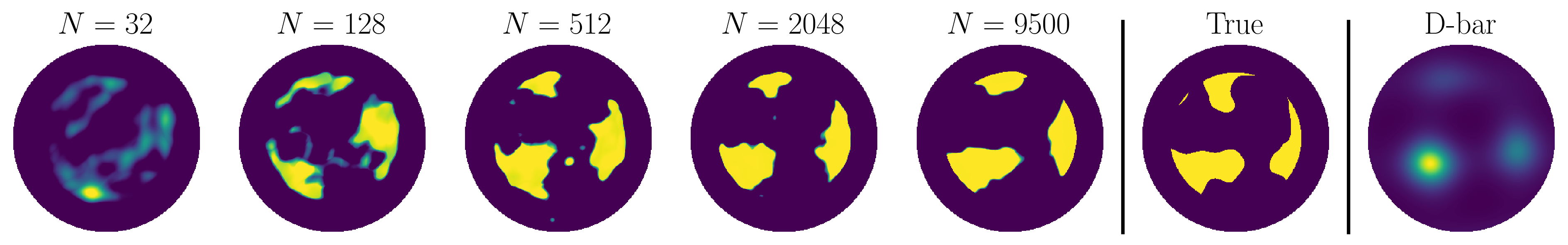}\\[-0.3\baselineskip]
    \includegraphics[width=\textwidth]{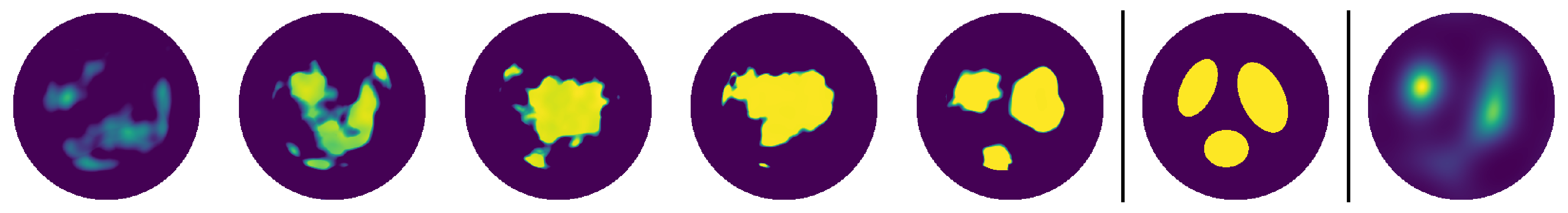}\\[-0.3\baselineskip]
    \includegraphics[width=\textwidth]{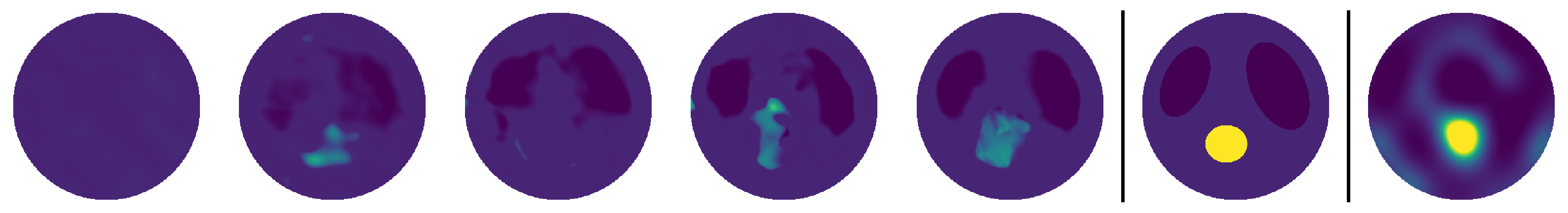}
    \caption{FNO reconstruction of three discontinuous conductivities as training dataset size $N$ increases. D-bar method \cite{knudsen2009regularized} reconstructions are shown in comparison. The noise level is $1\%$.}
    \label{fig:compare_dbar}
\end{figure}

The main result of this paper is \cref{thm:fno_approx_main}, which is the FNO approximation theorem for EIT as mentioned in Contribution~\ref{contr:approx}. The theorem holds for boundary data outside of the range of the EIT forward map. Data can leave this range for several reasons, including noisy measurements, discretization errors, or modeling errors. Our theory handles general worst-case perturbations and thus encompasses all three examples. 
The proof of \cref{thm:fno_approx_main} requires several steps, as illustrated in \cref{fig:flowchart}. We now give an overview of the argument.

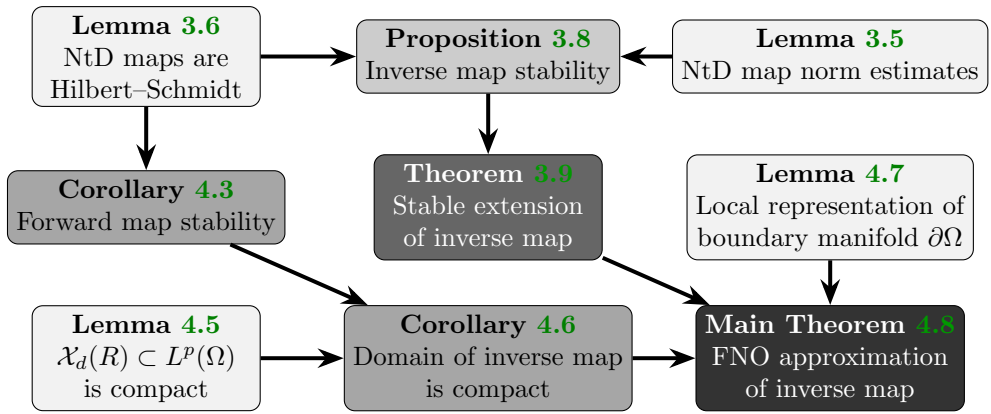
\begin{figure}[tb]
    \centering
    \captionsetup{skip=10pt}
    \begin{tikzpicture}[node distance=2cm, scale=1.0, every node/.style={transform shape}, every text node part/.style={align=center}]
    \tikzset{
        section/.style={rectangle, minimum width=3cm, minimum height=1cm, text centered, rounded corners, minimum height=2em, draw=black, fill=#1},
        arrow/.style={thick, ->, >=Stealth}
    }

    \node (1) [section=white!90!gray] {\textbf{Lemma~\ref{lem:ntd_hs_any_dim}}\\NtD maps are\\Hilbert--Schmidt};
    \node (2) [section=white!30!gray, below of=1] {\textbf{Corollary~\ref{cor:forward_stability_kernel}}\\ Forward map stability};
    \node (3) [section=white!90!gray, below of=2] {\textbf{Lemma~\ref{lem:compact_set}}\\$\mathcal{X}_d(R)\subset L^p(\Omega)$\\is compact};
    \node (4) [section=white!30!gray, right of=3, xshift=2.5cm] {\textbf{Corollary~\ref{cor:compact_set_kernel}}\\Domain of inverse map\\is compact};
    \node (5) [section=white!60!gray, right of=1, xshift=2.5cm] {\textbf{Proposition~\ref{prop:stability_kernel}}\\Inverse map stability};
    \node (6) [section=white!90!gray, right of=5, xshift=2.5cm] {\textbf{Lemma~\ref{lem:hs_to_op}}\\NtD map norm estimates};
    \node (7) [section=white!40!black, below of=5] {\color{white}\textbf{Theorem~{\hypersetup{linkcolor=darkgreen}\ref{thm:extension_eit_kernel}}}\\ \color{white}Stable extension\\ \color{white} of inverse map};
    \node (8) [section=white!20!black, right of=4, xshift=2.5cm] {\color{white}\textbf{Main Theorem~{\hypersetup{linkcolor=darkgreen}\ref{thm:fno_approx_main}}}\\ \color{white}FNO approximation\\ \color{white}of inverse map};
    \node (9) [section=white!90!gray, above of=8] {\textbf{Lemma~\ref{lem:left_inv_pou}}\\Local representation of\\boundary manifold $\pOmega$};
    
    \draw [arrow, line width=1.5pt] (1) to (2);
    \draw [arrow, line width=1.5pt] (2) to (4);
    \draw [arrow, line width=1.5pt] (3) to (4);
    \draw [arrow, line width=1.5pt] (1) to (5);
    \draw [arrow, line width=1.5pt] (6) to (5);
    \draw [arrow, line width=1.5pt] (5) to (7);
    \draw [arrow, line width=1.5pt] (7) to (8);
    \draw [arrow, line width=1.5pt] (9) to (8);
    \draw [arrow, line width=1.5pt] (4) to (8);
\end{tikzpicture}
    \caption{An overview of \cref{thm:fno_approx_main}'s proof structure.}
    \label{fig:flowchart}
\end{figure}

\subparagraph*{\emph{\textbf{Extension.}}}
Universal approximation theorems for neural operators such as FNO are by now well-established \cite{kovachki2021neural,kovachki2021universal,lanthaler2023nonlocality}. The first part of the proof of \cref{thm:fno_approx_main} requires us to identify the correct function-to-function operator to approximate with the FNO. Since we desire quantitative robustness to boundary data perturbations, we must extend the domain of the Calder\'on inversion operator to a larger space in which the perturbations belong. The FNO will approximate this extension. To extend the inverse map, we apply a result due to Benyamin~and~Lindenstrauss~\cite[Thm.~1.12, p.~18]{benyamini1998geometric} (see \cref{thm:hilbert_extension_ref}). This theorem extends a uniformly continuous map between subsets of Hilbert spaces to a globally-defined map with the same modulus of continuity. We identify appropriate Hilbert spaces of Hilbert--Schmidt operators and square integrable kernel functions with \cref{lem:ntd_hs_any_dim}. Together with some technical elliptic PDE estimates for NtD maps (\cref{lem:hs_to_op}), we show in \cref{prop:stability_kernel} that the inverse map from NtD integral kernels to conductivities in the admissible set $\cX$ is logarithmically stable between Hilbert spaces. Then invoking \cref{thm:hilbert_extension_ref}, we deduce the desired extension in \cref{thm:extension_eit_kernel}.
	
\subparagraph*{\emph{\textbf{Approximation.}}}
With the continuous function-to-function inverse map between Hilbert spaces now identified in \cref{thm:extension_eit_kernel}, it remains to apply the universal approximation theorem for FNOs and control errors due to boundary data perturbations. We apply a simple triangle inequality to decompose the reconstruction error into a stability part and an approximation error part \cite[Sec. 4.1.2, Eqn. (4.10), p. 27]{nelsen2025operator}. \Cref{thm:extension_eit_kernel} already handles the stability term. For the approximation error term, FNO universal approximation is uniform over a compact set of inputs. In our setting, this is a set of NtD integral kernel functions equal to the image of admissible conductivities $\cX$ under the EIT forward map. This image is indeed compact (\cref{cor:compact_set_kernel}) by the compactness of $\cX$ (\cref{lem:compact_set}) and the continuity of the forward map (\cref{cor:forward_stability_kernel}). A final technical issue is that the input kernels are bivariate functions on $\pOmega\times\pOmega$, which is a manifold that is incompatible with standard FNOs. Based on \cref{lem:left_inv_pou}, we instead use coordinate charts to define the FNO on local patches. \Cref{thm:fno_approx_main} follows.

\subsection{Literature review}\label{sec:intro_lit}
The Calder\'on problem and its application to EIT have been extensively studied from both analytic and computational perspectives \cite{Borcea,mueller2012linear,uhlmann2009electrical,hanke2017taste}.
The inverse boundary value problem for the conductivity equation
was originally proposed by Calder\'on \cite{calderon2006inverse} in 1980 and studied for analytic
conductivities in \cite{KohnVogelius1,KohnVogelius2}.
Sylvester and Uhlmann \cite{SylvesterUhlmann} proved the unique identifiability of the
conductivity in dimension $d\geq 3$ for isotropic conductivities which
are $C^\infty$-smooth. Analytic reconstruction methods based
on the complex geometrical optics solutions were given by Nachman 
\cite{Nachman1988} for the inverse conductivity problem and by Novikov \cite{Novikov1998} for the corresponding scattering problem. In three dimensions or higher, unique identifiability of the
conductivity is proven in \cite{HabermanTataru2013} for $C^1$ conductivities.
The problem has also been solved with measurements only on a part of the boundary \cite{KenigSjostrandUhlmann}.

In $d=2$ dimensions, the first global solution of the inverse conductivity problem is due to Nachman \cite{nachman1996global} for conductivities with two derivatives. In this seminal paper, the $\overline\partial$ technique was used for the first time in the study of Calder\'on's inverse problem.
Finally, Astala and P\"aiv\"arinta \cite{astala2006calderon} proved the uniqueness of the inverse 
problem in the form of its original formulation from \cite{calderon2006inverse}, i.e., for general isotropic conductivities in $L^\infty$ which are bounded from below and above by positive constants. For further developments on the uniqueness of the inverse problem, see
\cite{astala2006calderon,AstalaLassasPaivarinta2016,FerreiraKSU,faraco2018characterization,caro2016global,IdeUhlmannCPAM,
novikov2010global,uhlmann2009electrical}. 

Stability is much more delicate. Classical logarithmic-type stability estimates go back to \cite{alessandrini1988stable,barcelo2001stability,clop2010stability}, with further improvements under lower regularity assumptions such as $C^{1,\alpha}$ \cite{caro2013stability} or $W^{2,p}$ conductivities (with $p>d\geq 3$) \cite{choulli2023comments}. In parallel, a growing literature studies Lipschitz stability for finite-dimensional parametrizations and measurement models \cite{alberti2022infinite,harrach2019uniqueness,garde2024linearized,garde2025infinite}. This line of work provides a firm justification for reconstructing realistic conductivities from discretized data.

Regarding traditional numerical algorithms and regularization strategies for EIT, classical approaches include layer stripping \cite{somersalo1991layer}, D-bar and nonlinear Fourier methods \cite{knudsen2007d,knudsen2009regularized,mueller2003direct,astala2010numerical,astala2014nonlinear,mueller2020d,isaacson2021d,lytle2020nachman}, and variational formulations with total variation or level set regularization \cite{chung2005electrical,tanushev2007piecewise,jin2012analysis,rondi2008regularization}. Monotonicity and range-based methods provide yet another family of reconstruction strategies with guarantees \cite{garde2017convergence,sun2025learned}. These are especially natural in versions of the complete electrode model \cite{stuart2016bayesian,garde2021mimicking,garde2022series,harrach2019uniqueness}. Within this broad literature, most reconstruction results are performed under H\"older or Sobolev regularity assumptions on the conductivities, or for piecewise constant models with a fixed number of partitions \cite{faraco2018characterization,barcelo2001stability,clop2010stability}. In contrast, the current paper, at least in dimension two, studies more complicated sets of conductivities containing discontinuous functions of bounded variation.

Motivated by the promise of highly accurate and sharper resolution reconstructions, data-driven solution methods for inverse problems have gained traction in recent times \cite{arridge2019solving,ghattas2021learning,ongie2020deep}. In particular, there is substantial literature on combining deep learning with classical reconstruction pipelines in EIT and related imaging modalities. Examples include Deep~D-bar and Beltrami-Net methods, which learn post-processing or enhancement maps on top of D-bar reconstructions \cite{hamilton2018deep,hamilton2019beltrami,hamilton2016hybrid,siltanen2020electrical}, as well as end-to-end and hybrid approaches for EIT \cite{colibazzi2022learning,tanyu2023electrical,agnelli2020classification,beretta2025discrete,alsaker2024ct}; related ideas have been explored in full waveform inversion \cite{ding2022coupling}. These methods have demonstrated impressive empirical performance on a range of tasks, including inclusion detection, stroke classification, and hybrid CT/EIT imaging \cite{sun2025learned,agnelli2025stroke,tanyu2023electrical}. 

The present work builds on and contributes to a rapidly evolving theory of operator learning: the data-driven approximation of maps between infinite-dimensional spaces~\cite{kovachki2024operator,boulle2024mathematical,subedi2025operator,nelsen2024operator}. Once trained, such maps can be used as fast surrogate models in numerical analysis and scientific computing problems. Within this emerging subfield of scientific machine learning, a class of operator learning architectures known as neural operators are particularly effective for approximating function-to-function maps encoded by parametrized PDEs~\cite{kovachki2021neural}. Although the bulk of existing research centers on well-posed forward operator learning, the use of operator learning to accelerate or directly solve inverse problems is also gaining significant attention \cite{nelsen2025operator}, both in deterministic~\cite{kaltenbach2023semi,chen2023let,nguyen2024tnet,wang2024latent,yang2021seismic,zhang2024bilo,long2025invertible,guo2023transformer,lingsch2024fuse,jiao2024solving,massa2022approximation} and Bayesian settings~\cite{herrmann2020deep,baptista2024conditional,cao2023residual,cao2024lazydino,cao2025derivative}.

Concerning inverse problems for PDEs, neural operator methods have recently been proposed as flexible surrogates for mapping high-dimensional measurement data to unknown parameter fields. Neural Inverse Operators (NIO) learn a map from boundary measurement data---treated as an empirical measure of sensor locations and values---to unknown parameters, and have been successfully applied to solve various nonlinear inverse problems \cite{molinaro2023neural,guerra2025learning,nelsen2025operator}. Similarly, transformer-based architectures for inverse boundary value problems \cite{guo2023transformer} and direct sampling–inspired networks for EIT \cite{guo2021construct} process boundary data as unordered sets or point clouds, echoing the empirical measure viewpoint of NIO. However, these approaches are typically trained on noiseless or low-noise forward simulations and are primarily designed to invert data that lie in the range of the forward map.

In contrast, the current paper goes further by constructing and analyzing an extension of the EIT inverse map that is suitable for out-of-range or highly noisy measurement data. Indeed, quantitative control of the extension's modulus of continuity allows one to bound the influence of noise on the reconstruction error. Some works that are closest in spirit to the present setting also employ neural operators and extension-type approximation arguments~\cite{castro2024calderon,abhishek2024solving,pineda2023deep}. In comparison to these works, the current paper treats genuinely infinite-dimensional sets of conductivities, works with absolute instead of relative boundary maps, and provides a more detailed stability and approximation analysis for the approximate inverse map in the presence of noisy measurements. See \cref{sec:approx_discuss} for additional discussion.

\subsection{Outline}\label{sec:intro_outline}
The remainder of this paper is organized as follows. \Cref{sec:prelim} introduces essential background material. \Cref{sec:extend} makes Contribution~\ref{contr:extend} by stably extending the EIT inversion operator to a map between subsets of two Hilbert spaces. \Cref{sec:approx} contains the approximation theory and main result of the paper, making Contribution~\ref{contr:approx}. \Cref{sec:numerics} addresses Contribution~\ref{contr:numeric} by training and evaluating neural operators on three challenging EIT datasets. \Cref{sec:conclusion} provides concluding remarks and an outlook toward future developments for data-driven EIT and beyond. The appendix collects remaining proofs (\cref{app:proofs}), auxiliary lemmas (\cref{app:lemmas}), and numerical experiment details (\cref{app:numerics}).

\section{Preliminaries}\label{sec:prelim}
This section introduces notation (\cref{sec:prelim_notation}), mathematical preliminaries (\cref{sec:prelim_func}), the EIT problem setup (\cref{sec:prelim_eit}), and neural operators (\cref{sec:prelim_neural}).

\subsection{General notation}\label{sec:prelim_notation}
We use the term \emph{domain} to mean an open connected set $\Omega$ in $\R^d$ with sufficiently smooth boundary. We use $\mathsf{n}$ to denote the outward pointing unit vector normal to the boundary $\pOmega$ of domain $\Omega$ and $\overline{\Omega}$ for its closure. The Euclidean inner product on $\R^d$ between vectors $x$ and $y$ is denoted by $x\cdot y$ and the induced norm by $\abs{\slot}$. The constant function $x\mapsto 1$ is denoted by $\onebm$. We occasionally invoke the notation $a \lesssim b$ for nonnegative real numbers $a$ and $b$ if $a\leq Cb$ for an unimportant constant $C>0$ and similarly $a\gtrsim b$ if $b\lesssim a$. We write $a \asymp b$ if both $a\lesssim b$ and $b\lesssim a$. We let $\sP(X)$ denote the set of Borel probability measures supported on a set $X$. The identity map on $X$ is written as $\Id_X$. Define $\Z_{\geq 0}\defeq \{0,1,\ldots\}$ and $\R_{\geq 0}$ (respectively, $\R_{> 0}$) by $[0,\infty)$ (respectively, $(0,\infty)$). For $z_1+\mathsf{i}z_2\eqdef z\in\C$, we write $\overline{z}\defeq z_1-\mathsf{i}z_2$ for the complex conjugate. We use the convention that complex inner products on vector and function spaces are linear in the first argument and conjugate-linear in the second argument.

\subsection{Function spaces and linear operators}\label{sec:prelim_func}
For a domain $\Omega\subset \R^d$, we work with the standard Lebesgue spaces $L^p(\Omega;\C)$ for $p\in[1,\infty]$ and Sobolev Hilbert spaces $H^k(\Omega;\C)$ for $k\in\Z_{\geq 0}$ \cite{adams2003sobolev,lions2012non}. When $k\in\R$, we use interpolation $(k>0)$ and duality $(k<0)$ to define $H^k(\Omega;\C)$ \cite[Appendix~A]{abraham2019statistical}. We often omit the function space codomain if it is the real line $\R$ or complex plane $\C$, or when no confusion is caused by doing so. For example, we write $L^p(\Omega;\C)\equiv L^p(\Omega)$. However, we mostly work in the real-valued subsets of complex function spaces. Similarly, we may occasionally omit the domain of the function space under consideration if it is obvious from the context, e.g., writing $L^\infty$ instead of $L^\infty(\Omega)$.

The space of functions with \emph{bounded variation} (BV) \cite[Chp.~14.1]{leoni2017first} plays an important role in this paper. For an open set $U\subseteq\R^d$, the total variation of an element $v\in L^1_{\mathrm{loc}}(U;\R)$ is defined by
\begin{align}
	V(v;U)\defeq\sup\set{\int_{U}v(x)\nabla\cdot\phi(x)\dd{x}}{\phi\in C_c^{\infty}(U;\R^d) \,\ \text{and} \,\ \norm{\phi}_{L^\infty(U;\R^d)}\leq 1},
\end{align}
where $C_c^\infty(U;\R^d)$ is the set of infinitely continuously differentiable $\R^d$-valued functions with compact support in $U$. The set of real-valued BV functions is defined as
\begin{align}
	\mathsf{BV}(U)\defeq \set{v\in L^1(U)}{V(v;U)<\infty}\,.
\end{align}
We endow $\mathsf{BV}(U)$ with the norm
\begin{align}
	v\mapsto \norm{v}_{\BV}\defeq \max\Bigl(\norm{v}_{L^1(U)}, V(v;U)\Bigr)\,.
\end{align}
It follows that $(\BV(U),\norm{\slot}_\BV)$ is a Banach space \cite{leoni2017first}.

Next, for $\al\in(0,1]$, we recall the H\"older space $C^{1,\al}(\overline{\Omega})$ of real-valued continuously differentiable functions with $\al$-H\"older first derivatives on the closure $\overline{\Omega}$ of a domain $\Omega\subset\R^d$ \cite[Sec.~1.29]{adams2003sobolev}. It is a Banach space when endowed with the norm
\begin{align}
    v\mapsto \norm{v}_{C^{1,\al}}\defeq\max\Bigl(\norm{v}_{C^1}, \max_{\abs{\beta}=1}\abs{\partial^\beta v}_{C^{0,\al}}\Bigr)\,,
\end{align}
where the $C^1(\overline{\Omega})$ norm and $C^{0,\al}(\overline{\Omega})$ seminorm are defined for a given $v$ by
\begin{align}
    \norm{v}_{C^1}\defeq \max_{\abs{\beta}\leq 1}\sup_{x\in\Omega}\abs[\big]{(\partial^\beta v)(x)}
    \qa
    \abs{v}_{C^{0,\al}}\defeq \sup_{\substack{(x,x')\in\Omega\times\Omega\\x\neq x'}}\frac{\abs{v(x)-v(x')}}{\abs{x-x'}^\al}\,,
\end{align}
respectively. We have adopted standard multi-index notation in the preceding displays.

We now turn to linear operators. The Banach space of continuous linear operators from a Banach space $X$ to a Banach space $Y$ is denoted by $\sL(X;Y)$. It is equipped with the operator norm $T\mapsto \norm{T}_{\sL(X;Y)}\defeq \sup_{\norm{x}_X\leq 1}\norm{Tx}_Y$. The Hilbert space of Hilbert--Schmidt operators between Hilbert spaces $\cU$ and $\cV$ is denoted as $\HS(\cU;\cV)$. It has inner product $(S,T)\mapsto \ip{S}{T}_{\HS(\cU;\cV)}\defeq \sum_{j\in\Z_{\geq 0}}\ip{Se_j^{(\cU)}}{Te_j^{(\cU)}}_{\cV}$ for any orthonormal basis $\{e_j^{(\cU)}\}_{j\in\Z_{\geq 0}}$ of $\cU$. If $\cU=\cV$, we write $\HS(\cU)$ and similarly for $\sL(\cU)$. For $u\in\cU$ and $v\in\cV$, their tensor product is the linear operator $v\otimes_\cU u\colon w\mapsto\ip{w}{u}_\cU v$.

For $s\in\R$, we make frequent use of the Sobolev Hilbert spaces $H^s(\pOmega)$ of functions on the boundary $\pOmega$ of domain $\Omega$. To define them, we follow \cite[Appendix~A, pp.~192--193]{abraham2019statistical}. Let $\sigma$ be the surface measure on $\pOmega$. Since $\Omega$ is a domain, $\pOmega$ is a compact manifold. The (negative) Laplace--Beltrami operator $-\Delta_{\pOmega}$ on this manifold has a real discrete spectrum consisting of nonnegative eigenvalues $\{\lambda_j\}_{j\in\Z_{\geq 0}}$ that we sort in nondecreasing order. Let $\{\varphi_j\}_{j\in\Z_{\geq 0}}$ denote the corresponding $L^2(\pOmega)$-orthonormal real-valued eigenfunctions, i.e., $-\Delta_{\pOmega}\varphi_j=\lambda_j\varphi_j$. Define
\begin{align}\label{eqn:defn_boundary_sobolev_spectral}
    (h_1,h_2)\mapsto \ip{h_1}{h_2}_{H^s(\pOmega)}\defeq \sum_{j=0}^\infty (1+\lambda_j)^s\ip{h_1}{\varphi_j}_{L^2(\pOmega)}\overline{\ip{h_2}{\varphi_j}}_{L^2(\pOmega)}\,.
\end{align}
Then we write $H^s(\pOmega)\defeq \{h\in L^2(\pOmega)\colon \norm{h}_{H^s(\pOmega)}<\infty\}$ if $s\geq 0$. Otherwise, $H^s(\pOmega)$ is taken to be the completion of $L^2(\pOmega)$ under the norm induced by \eqref{eqn:defn_boundary_sobolev_spectral}. It is useful to define the zero-mean Hilbert subspace
\begin{align}
    H_\diamond^{s}(\pOmega)\defeq \set{h\in H^{s}(\pOmega)}{\int_\pOmega h(x) \, \sigma(dx) = 0}
\end{align}
and also the quotient Hilbert space $H^s(\pOmega)/\C$ of equivalence classes of functions that are equal up to the addition of complex scalars. Since $\varphi_0$ is proportional to $\onebm$ and $\lambda_0=0$, the quotient norm $\norm{\,[h]\,}_{H^s(\pOmega)/\C}\defeq \inf_{z\in\C}\norm{h-z}_{H^s(\pOmega)}=\norm{h_0}_{H^s(\pOmega)}$ for some zero-mean representative $h_0\in H_\diamond^s(\pOmega)$ of equivalence class $[h]\in H^s(\pOmega)/\C$ \cite[Remark~3.5, pp.~1931--1932]{garde2022series}. The quotient space $H^s(\Omega)/\C$ is defined in the same way.

The spectral definition \eqref{eqn:defn_boundary_sobolev_spectral} implies that $\{\varphi^{(s)}_j\defeq (1+\lambda_j)^{-s/2}\varphi_j\}_{j=0}^\infty$ is an orthonormal basis of $H^s(\pOmega)$ \cite[p.~193, Remark i]{abraham2019statistical}. Similarly, $\{\varphi^{(s)}_j\}_{j=1}^\infty$ is an orthonormal basis of both $H_\diamond^s(\pOmega)$ and $H^s(\pOmega)/\C$. Using these facts, for any real $s$ and $t$ and natural numbers $J$ and $K$, we define the orthogonal projection operator (tensor) $\Pi_{JK}\in \sL(\HS(H^s_\diamond(\pOmega);H^t(\pOmega)/\C))$ by
\begin{align}\label{eqn:projection_operator}
    T\mapsto \Pi_{JK}T\defeq \sum_{j=1}^{J}\sum_{k=1}^K\ip{T \varphi_k^{(s)}}{\varphi_j^{(t)}}_{H^t(\pOmega)}\,
    \varphi_j^{(t)}\otimes_{H^s(\pOmega)}\varphi_k^{(s)}\,.
\end{align}
The range of $\Pi_{JK}$ is a finite-dimensional subspace that is independent of $s$ and $t$ \cite[Appendix~B, p.~194]{abraham2019statistical}. When $J=K$, we define $P_J\defeq \Pi_{JJ}$.

\subsection{Electrical impedance tomography}\label{sec:prelim_eit}
In EIT imaging, real data is acquired by applying current patterns and measuring the resulting voltage at a finite number of electrodes attached to the boundary of an object \cite{cheney1999electrical}. The goal is to use such data to infer hidden structure about the object's interior. Calder\'on's problem \cite{calderon2006inverse} is a continuum model for EIT in which an uncountably infinite number of measurements is assumed to be available for inversion. We now describe this model. For a more comprehensive mathematical treatment, we refer the reader to several books and review articles \cite{feldman2025calderon,uhlmann2009electrical,uhlmann201230}.

In the rest of this paper, we fix $ \Omega\subset \R^d $ to be a bounded nonempty domain with $C^\infty$-smooth boundary $ \pOmega $. Consider the linear elliptic PDE \eqref{eqn:elliptic} with positive almost everywhere (a.e.) conductivity coefficient $ \gamma\in L^{\infty}(\Omega;\R_{>0}) $ and zero-mean Neumann boundary data $ g\in H_\diamond^{-1/2}(\pOmega) $. The (weak) solution $u=u_{\gamma,g}$ exists in $H^1(\Omega)/\C $ and is unique \cite[Chp.~2]{lions2012non}. Moreover, the NtD map $\ntd$ is defined (in the trace sense) by
\begin{equation}\label{eqn:ntd}
	g\mapsto \ntd g \defeq u_{\gamma,g}\big|_{\pOmega}
\end{equation}
and satisfies $ \ntd \in \sL(H_\diamond^{-1/2}(\pOmega);H_{\phantom{\diamond}}^{1/2}(\pOmega)/\C)$. A related object is the Dirichlet-to-Neumann (DtN) map $ \dtn \in \sL(H^{1/2}(\pOmega)/\C;H^{-1/2}_\diamond(\pOmega))$, which is the continuous inverse of $\ntd$. These linear operators represent all possible measurements that can be taken at the boundary under the model \eqref{eqn:elliptic}. Although defined between complex Hilbert spaces, both operators $\ntd$ and $\dtn$ map real-valued functions to real-valued functions \cite[Remark on p.~176]{abraham2019statistical}.

Calder\'on's inverse conductivity problem is to recover $ \gamma $ from knowledge of (a potentially noisy version of) the boundary operator $\dtn $. Notice that although $\dtn$ itself is linear, the mapping $\gamma\mapsto\dtn$ is nonlinear. Consequently, the Calder\'on problem is a canonical example of a severely ill-posed nonlinear inverse problem. Although most of the mathematical literature focuses on the DtN map $\dtn$, the present paper instead aims to accurately approximate a variant of the inversion operator $\ntd\mapsto \gamma$ defined for NtD maps. This NtD setup is more amenable to practical measurement systems, machine learning models, and finite-dimensional approximations.

The main results of this paper will hold for real-valued conductivities in an admissible set $\cX$. At a minimum, inverse problem solutions belonging to this set should be unique and (weakly) stable to perturbations of the boundary data. To define $\cX$, first fix $\Omega'\subset\R^d$ to be a domain that is compactly contained in $\Omega$, i.e.,
\begin{align}\label{eqn:compact_support_set}
    \overline{\Omega'}\subset \Omega\,.
\end{align}
Next, fix $m>0$ and $M\geq \max(1, m)$ to be strictly positive numbers. Define sets
\begin{align}\label{eqn:conductivity_set_prelim}
    \begin{split}
        \Gamma &\defeq \set{\gamma\in L^\infty(\Omega;\R)}{m\leq \gamma(x)\leq M \,\ \text{for a.e.} \,\ x\in\Omega}\qa\\
      \Gamma' &\defeq \set{\gamma\in\Gamma}{\gamma = \onebm \,\ \text{a.e. on} \,\ \Omega\setminus\Omega' }
        \,.
    \end{split}
\end{align}
Conductivities $\gamma\in \Gamma'$ have the property that $\gamma\equiv 1$ in a neighborhood of the boundary $\partial\Omega$. This fact will be crucial later in the paper when we exploit regularity theory for elliptic PDEs near the boundary to control norms of the NtD map. Additionally, it will be useful at times to view elements $\gamma\in\Gamma'$ as functions $\widetilde{\gamma}\colon\R^d\to\R$ on the whole of $\R^d$ by defining $\widetilde{\gamma}$ to be the extension of $\gamma$ with the property that $\widetilde{\gamma}\equiv 1$ on $\R^d\setminus\Omega$. This extension is continuous because $\gamma\equiv 1$ on $\Omega\setminus\Omega'$.

Now fix $R>0$ and $\al\in(0,1]$. We require conductivities in $\cX$ to satisfy \emph{a priori} uniform bounds that depend on the spatial dimension $d$. To this end, define
\begin{align}\label{eqn:conductivity_set}
    \begin{split}
        \cX_{\BV}(R) &\defeq \set{\gamma\in \Gamma' \cap \mathsf{BV}(\Omega)}{\norm{\gamma}_{\BV}\leq R}\,,\\
        \cX_{C^{1,\al}}(R) &\defeq \set{\gamma\in \Gamma' \cap C^{1,\al}(\overline{\Omega})}{\norm{\gamma}_{C^{1,\al}}\leq R}\,, \qa \\
        \cX_d(R) &\defeq 
        \begin{cases}
            \cX_{\BV}(R), & \qif d=2\,,\\
            \cX_{C^{1,\al}}(R), & \qif d\geq 3\,.\\
        \end{cases}
    \end{split}
\end{align}
We then take our admissible set to be $\cX\defeq \cX_d(R)$. To endow $\cX_d(R)$ with metric structure, fix $p\in[1,\infty)$. We equip the admissible set $\cX_d(R)$ with the metric $\sfd_p\colon \cX_d(R)\times \cX_d(R)\to\R_{\geq 0}$ induced by the $L^p(\Omega)$ norm. That is,
\begin{align}\label{eqn:metric_distance}
	\sfd_p(\gamma,\gamma_0)\defeq \norm{\gamma-\gamma_0}_{L^p(\Omega)}
\end{align}
for any $\gamma$ and $\gamma_0$ in $\cX_d(R)$. Due to the uniform $L^\infty$ bounds in the definition of $\Gamma$, on $\cX_d(R)$ all $L^p$ distances $\sfd_p$ for $p\in [1,\infty)$ are topologically equivalent. However, the $L^1$ norm is better suited than the usual $L^\infty$ norm to detect closeness of discontinuous or nearly discontinuous conductivities. Thus, we use $p=1$ in our numerical experiments. In the theoretical analysis, we work with the $L^p$ distance $\sfd_p$ for general $p$, but with a primary emphasis on the $p=2$ case.

\subsection{Neural operators}\label{sec:prelim_neural}
We use \emph{neural operators} to approximate the nonlinear mapping from boundary measurements to conductivities arising in the Calder\'on problem. Neural operators mimic the feedforward, compositional structure of finite-dimensional neural networks in a function space-consistent way \cite{kovachki2021neural}. They take the form
\begin{align}\label{eqn:fno_torus}
    \Psi^{(\mathrm{NO})}\defeq \cQ\circ\cL_L\circ\cL_{L-1}\circ\cdots\circ\cL_2\circ\cL_1\circ\cS\,.
\end{align}
The local Nemytskii operators $\cS$ and $\cQ$ in \eqref{eqn:fno_torus} act pointwise on input functions. They are usually chosen to be shallow neural networks, e.g., $\cQ(h(x))=A_1\varsigma(A_0 h(x) + a_0) +a_1$ for a function $h$, point $x$, pointwise nonlinearity $\varsigma$, and weights and biases $\{A_0,a_0,A_1,a_1\}$. In contrast, the $L$ hidden layer operators $\{\cL_\ell\}_{\ell=1}^L$ act \emph{nonlocally}. Each hidden layer $\cL_\ell$ maps vector-valued functions to vector-valued functions. The vector dimension $d_\mathrm{c}$ is typically fixed for all $\ell$ and is called the channel width of the neural operator. Suppose that the underlying domain is a bounded subset $\sfD\subset\R^n$.
The basic hidden layer parametrization is
\begin{equation}\label{eqn:no_layers}
    \bigl(\cL_{\ell}(h)\bigr)(x) = \varsigma_\ell\bigl(W_\ell h(x) + (\cK_\ell h)(x) + b_\ell(x)\bigr)
\end{equation}
for an input function $h\colon \sfD\to\R^{d_\mathrm{c}} $ and evaluation point $x\in\sfD$. The function $\varsigma_\ell\colon\R\to\R$ is a pointwise nonlinear activation; usually the last hidden layer activation $\varsigma_L$ is set to the identity. The local weight matrix $W_\ell \in \R^{d_\mathrm{c}\times d_\mathrm{c}}$ acts pointwise, while $b_\ell\colon\sfD\to\R^{d_\mathrm{c}}$ is a vector-valued bias function. The main innovation of neural operators is the linear kernel integral operator $\cK_\ell$ given by
\begin{align}\label{eqn:no_kernel}
    (\cK_\ell h)(x)\defeq\int_\sfD \kappa_\ell(x,y)h(y)\dd{y}\,.
\end{align}
The matrix-valued kernel $\kappa_\ell\colon\sfD\times\sfD\to\R^{d_\mathrm{c}\times d_\mathrm{c}}$ determines the type of the neural operator architecture~\cite[Sec.~4]{kovachki2021neural}. This paper primarily considers the FNO~\cite{li2020fourier}. The FNO is a neural operator of the form \eqref{eqn:fno_torus} with kernel integral operator \eqref{eqn:no_kernel} parametrized as the Fourier series
\begin{align}\label{eqn:fno_layer}
    (\cK_\ell h)(x) = \left\{\sum_{k \in \Z^n}\left(\sum_{j=1}^{d_\mathrm{c}} \bigl(P^{(k)}_\ell\bigr)_{l j}\ip[\big]{e^{2\pi \mathsf{i} \ip{k}{\slot}_{\R^n}}}{h_j}_{L^2(\mathbb{T}^n;\mathbb{C})} \right) \, e^{2\pi \mathsf{i} \ip{k}{x}_{\R^n}}\right\}_{l=1}^{d_\mathrm{c}}.
\end{align}
The Fourier series coefficients $P^{(k)}_\ell \in \mathbb{C}^{d_{c} \times d_{c}}$ for each $k\in\Z^n$ are trainable. By choosing $\sfD\defeq \T^n \simeq [0,1]^n_{\mathrm{per}}$ to be the unit hypertorus, the parametrization \eqref{eqn:fno_layer} induces a convolutional structure that admits efficient numerical implementation via the Fast Fourier Transform (FFT) \cite{lanthaler2024discretization,li2020fourier}. More sophisticated versions maintain computational efficiency without assuming periodicity or well-behaved geometries~\cite{li2023geometry,li2023fourier,li2025geometric,huang2025operator,lanthaler2023nonlocality}.

\section{Extension of the inverse map}\label{sec:extend}
This section carries out the inverse problem extension procedure that forms the foundation of our neural operator approximation. The main result of the section relies on the following abstract extension theorem due to Benyamin and Lindenstrauss~\cite[Thm.~1.12, p.~18]{benyamini1998geometric} that pertains to uniformly continuous maps between Hilbert spaces.

\begin{theorem}[Benyamin--Lindenstrauss extension]\label{thm:hilbert_extension_ref}
    Let $\cH_1$ and $\cH_2$ be Hilbert spaces. Suppose that $U$ is a subset of $\cH_1$ and that $\omega\colon \R_{>0}\to \R_{\geq 0}$ is a nonnegative concave function with the property that $\lim_{t\downarrow 0}\omega(t)=0$. If $G\colon U\to \cH_2$ is a uniformly continuous map with modulus of continuity bounded above pointwise by $\omega$, then there exists a uniformly continuous map $\sfG\colon \cH_1\to \cH_2$ defined on the whole of $\cH_1$ such that $\sfG|_{U}=G$. The modulus of continuity of the extension $\sfG$ is also bounded above by $\omega$.
\end{theorem}

The aim of this section is to apply \cref{thm:hilbert_extension_ref} with $G$ representing the EIT solution operator mapping the admissible boundary data to the conductivity. Once found, such an extension allows us to control the inverse problem reconstruction error of a neural operator method by the noise level in the boundary measurements and the network approximation error itself \cite[Sec.~4.1.2, Eqn.~(4.10), p.~27]{nelsen2025operator}. In particular, this approach does not require quantitative stability of the neural operator, unlike in \cite{pineda2023deep}.

The rest of \cref{sec:extend} is primarily devoted to verifying the hypotheses of \cref{thm:hilbert_extension_ref}. \Cref{sec:extend_stability} develops a stability result for Calder\'on's problem (with DtN maps) that is applicable to the admissible set $\cX_d(R)$ of conductivities considered in this paper. This is a first step toward identifying $\omega$ and $\cH_2$. Next, \cref{sec:extend_bounds} establishes various bounds on NtD maps and their differences. These estimates expose the smoothing properties of such maps, enable the interpolation of different norms on spaces of linear operators, and facilitate stability estimates written in terms of NtD maps. \Cref{sec:extend_kernel} identifies NtD maps with square integrable kernel functions, which in turn determines the sets $U$ and $\cH_1$. For machine learning architectures, such kernel functions are a more convenient form of input data than are linear operators. Last, \cref{sec:extend_main} deploys the technical results from the previous two subsections alongside \cref{thm:hilbert_extension_ref} to deliver a stable extension of the EIT inverse map.

\subsection{Stability of the inverse map}\label{sec:extend_stability}
This subsection produces a stability theorem for Calder\'on's problem with conductivities in the set $\cX_d(R)$ from~\eqref{eqn:conductivity_set}. The case $d\geq 3$ is more standard in the literature and typically delivers logarithmic stability in $L^\infty(\Omega)$. In contrast, the $d=2$ case for the BV conductivities in $\cX_2(R)$ is more challenging. We establish a $L^2(\Omega)$ continuity result when $d=2$ that is a consequence of \cite[Thm.~1.6, pp.~5661--5662]{faraco2018characterization}. It relies on the notion of integral modulus of continuity. To this end, for any $p\in [1,\infty)$, the integral $p$-modulus of continuity of any $v\colon\R^d\to\R$ is the function $\omega^{(p)}(v)\colon \R_{\geq 0}\to\R_{\geq 0}$ defined for $t\geq 0$ by
\begin{align}\label{eqn:modulus_sup}
    \omega^{(p)}(v)(t)\defeq \sup_{\abs{y}\leq t}\norm{\tau_y v - v}_{L^p(\R^d)}\,.
\end{align}
In \eqref{eqn:modulus_sup}, $\tau_y$ is the translation-by-$y$ operator $(\tau_yv)(x)\defeq v(x+y)$ for fixed $y\in\R^d\setminus\{0\}$ and for any $x\in\R^d$. We often invoke \eqref{eqn:modulus_sup} with $v$ equal to the continuous extension $\widetilde{\gamma}\colon \R^d\to\R_{\geq 0}$ of $\gamma\in\cX_d(R)$ defined by $\widetilde{\gamma} = \gamma$ on $\Omega$ and $\widetilde{\gamma}\equiv 1$ on $\R^d\setminus\Omega$.

We are now in a position to state and prove the $L^2(\Omega)$ stability theorem.
\begin{theorem}[$L^2$ stability of the Calder\'on problem]\label{thm:stability}
    Let $d\geq 2$. Suppose that $\al\in(0,1)$ and $R>0$ are as in \eqref{eqn:conductivity_set}. Then there exist constants $C>0$, $\rho>0$, and $t_0\in (0,1)$ depending on $d$, $\al$, $M$, $m$, $R$, $\Omega$, and $\Omega'$ such that for any $\gamma\in\cX_d(R)$ and any $\gamma_0\in\cX_d(R)$ with $\cX_d(R)$ as in \eqref{eqn:conductivity_set}, the DtN maps $\dtn$ and $\dtnz$ satisfy
    \begin{align}\label{eqn:stability_log}
        \norm{\gamma-\gamma_0}_{L^2(\Omega)} \leq C \omega\Bigl(\norm{\Lambda_{\gamma} - \Lambda_{\gamma_0}}_{\sL\left(H^{1/2}(\pOmega)/\C;H^{-1/2}(\pOmega)\right)}\Bigr)\,,
    \end{align}
    where the modulus of continuity $\omega\colon \R_{>0}\to\R_{>0}$ is given by
    \begin{align}\label{eqn:modulus_log}
        t\mapsto \omega(t)\defeq \log\left(\frac{1}{\min(t,t_0)}\right)^{-\rho}\,. 
    \end{align}
\end{theorem}
\begin{proof}
    First, $\dtn\onebm =\dtnz\onebm =0$ for any $(\gamma,\gamma_0)\in\Gamma\times\Gamma$ implies that
    \begin{align*}
        \norm{\Lambda_{\gamma} - \Lambda_{\gamma_0}}_{\sL\left(H^{1/2}(\pOmega)/\C;H^{-1/2}(\pOmega)\right)} =  \norm{\Lambda_{\gamma} - \Lambda_{\gamma_0}}_{\sL\left(H^{1/2}(\pOmega);H^{-1/2}(\pOmega)\right)}\,.
    \end{align*}
    Thus, it suffices to develop the stability results for the rightmost operator norm. We denote this norm by $\norm{\slot}$ in the remainder of the proof. 
    
    We begin with the $d=2$ case. Let $B \defeq \max\{1, M,m^{-1}\}\geq 1$. Notice that $\gamma\in\cX_2(R)$ implies $B^{-1}\leq\widetilde{\gamma}(x)\leq B$ for a.e. $x\in\R^2$. Fix $p=p_B > 2B\geq 2$ and $\gamma\in\cX_2(R)$. 
    The asserted result will follow by \cite[Thm.~1.6, pp.~5661--5662]{faraco2018characterization} if there exists a so-called uniform modulus of continuity $\omega_0$ such that $\sup_{\gamma\in\cX_2(R)} \omega^{(p)}(\widetilde{\gamma})(t)\leq \omega_0(t)$ for every $t\geq 0$. To show this, we invoke properties of BV functions. For any $y\in\R^2$, by H\"older's inequality and the triangle inequality we compute
    \begin{align*}
        \norm{\tau_y \widetilde{\gamma} - \widetilde{\gamma}}_{L^{p}(\R^2)}&=\biggl(\int_{\R^2}\abs{\widetilde{\gamma}(x+y)-\widetilde{\gamma}(x)}^{p-1} \abs{\widetilde{\gamma}(x+y)-\widetilde{\gamma}(x)} \dd{x}\biggr)^{1/p}\\
        &\leq \norm[\big]{\tau_y \widetilde{\gamma} - \widetilde{\gamma}}_{L^{\infty}(\R^2)}^{(p-1)/p}\norm[\big]{\tau_y \widetilde{\gamma} - \widetilde{\gamma}}_{L^{1}(\R^2)}^{1/p}\\
        &\leq (2B)^{(p-1)/p}\norm[\big]{\tau_y \widetilde{\gamma} - \widetilde{\gamma}}_{L^{1}(\R^2)}^{1/p}\,.
    \end{align*}
    By \cref{lem:tv_identity} and the definitions of $\cX_2(R)$, BV, and total variation, it holds that $V(\widetilde{\gamma};\R^2)=V(\gamma;\Omega)<\infty$ and $\widetilde{\gamma}\in L^1_{\mathrm{loc}}(\R^2)$ for any $\gamma\in\cX_2(R)$. Application of \cite[Exercise~14.3(ii), p. 460]{leoni2017first}, \cite[Exercise~14.3(i), p. 460]{leoni2017first}, and \cite[Lemma~14.37, p.~483]{leoni2017first}, in that order, delivers the estimate
    \begin{align*}
        \norm{\tau_y \widetilde{\gamma} - \widetilde{\gamma}}_{L^{1}(\R^2)}\leq  V(\widetilde{\gamma};\R^2)\abs{y}= V(\gamma;\Omega) \abs{y}\leq R \abs{y}
    \end{align*}
    on $\cX_2(R)$. For every $t\geq 0$, we deduce that
    \begin{align*}
        \sup_{\gamma\in\cX_2(R)} \omega^{(p)}(\widetilde{\gamma})(t)=\sup_{\gamma\in\cX_2(R)} \sup_{\abs{y}\leq t}\norm{\tau_y \widetilde{\gamma} - \widetilde{\gamma}}_{L^{p}(\R^2)}\leq \bigl((2B)^{(p-1)/p}R^{1/p}\bigr) t^{1/p}\,.
    \end{align*}
    Thus, with $C_{B,R}$ the rightmost constant in the preceding display, 
    the map $\omega_0\colon t\mapsto C_{B,R}\, t^{1/p_B}$ is a valid uniform modulus of continuity. In particular, $\omega_0$ is continuous on $\R_{>0}$ and hence upper semi-continuous. Then \cite[p.~5662]{faraco2018characterization} implies that \eqref{eqn:stability_log} is valid with modulus of continuity $\omega$ in \eqref{eqn:modulus_log} for some $\rho=\rho_2$ and $t_0=t_0^{(2)}$.

    Now let $d\geq 3$. We exploit the stronger \emph{a priori} bounds on the conductivities in the set $\cX_d(R)$. Define $b=\max\{1, m^{-1}, R\}$. For $\gamma\in\cX_d(R)$, we have $\gamma(x)\geq m^{-1}\geq b^{-1}$ for all $x\in \Omega$ and $\norm{\gamma}_{C^{1,\al}}\leq R\leq b$. Then \cite[Thm.~1.1, p.~470]{caro2013stability} delivers the stability bound $\norm{\gamma-\gamma_0}_{L^\infty(\Omega)} \leq C \omega'(\norm{\dtn - \dtnz})$ on $\cX_d(R)$, where $\omega'(t)=(\log(\min(t, t_0^{(d)})^{-1}))^{-\rho_d}$ for some $\rho_d\in(0,1)$ depending only on $d$ and $\al$. Here, the value of $t_0^{(d)}\in(0,1)$ is implied by \cite[pp.~489--490]{caro2013stability}. To complete the proof, we use the continuous embedding $L^\infty\hookrightarrow L^2$, enlarge $C$, and choose $\rho\defeq \min(\rho_2,\rho_d)$ and $t_0\defeq \max(t_0^{(2)}, t_0^{(d)})$ to ensure that \eqref{eqn:stability_log} and \eqref{eqn:modulus_log} are valid for all $d\geq 2$.
\end{proof}

The stability estimate \eqref{eqn:stability_log} from \cref{thm:stability} remains valid with the $L^p(\Omega)$ norm replacing the $L^2(\Omega)$ norm on the left-hand side if $p\leq 2$. This follows by continuous embedding. By the uniform $L^\infty$ bounds on the set $\Gamma$, interpolating the $L^\infty$ and $L^2$ norms shows that $L^p(\Omega)$ stability holds for $p>2$ if $\omega$ is replaced by $\omega^{2/p}$.
The following remark comments on other stronger norms in the $d\geq 3$ case.
\begin{remark}[stronger metrics on the set of conductivities]
    It is possible to obtain a result similar to \cref{thm:stability} in the $d\geq 3$ setting by applying \cite[Thm.~1.2, p.~4]{choulli2023comments}. This theorem proves logarithmic stability with the $L^2(\Omega)$ norm in \eqref{eqn:stability_log} replaced by the $H^1(\Omega)$ norm whenever the conductivities are uniformly bounded in the Sobolev space $W^{2,p}(\Omega)$ with $p>d$.  Although the universal approximation theorems we later invoke remain valid for Hilbert Sobolev spaces such as $H^1(\Omega)$, we choose to work in the weaker setting of \cref{thm:stability} in the $d\geq 3$ case to provide a unified presentation of the main results in all dimensions $d\geq 2$.
\end{remark}

\subsection{Norm estimates for Neumann-to-Dirichlet maps}\label{sec:extend_bounds}
The previous stability result in \cref{thm:stability} is formulated in terms of DtN maps. However, in realistic EIT current and voltage hardware setups, it is more convenient to work with NtD maps. This subsection develops useful technical machinery for NtD maps. Our first lemma controls differences of DtN maps by differences of NtD maps.
\begin{lemma}[NtD to DtN]\label{lem:ntd_to_dtn}
    There exists a constant $C>0$ depending on $d$, $\Omega$, $m$, and $M$ such that for any $\gamma\in\Gamma$ and $\gamma_0\in \Gamma$ with $\Gamma$ as in \eqref{eqn:conductivity_set_prelim}, it holds that
    \begin{align}\label{eqn:ntd_to_dtn}
         \norm[\big]{\dtn - \dtnz}_{\sL(H^{1/2}(\pOmega)/\C;H^{-1/2}(\pOmega))} \leq C \norm[\big]{\ntd-\ntdz}_{\sL(H^{-1/2}_\diamond(\pOmega);H_{\phantom{\diamond}}^{1/2}(\pOmega)/\C)}\,.
    \end{align}
\end{lemma}
\begin{proof}
    We use the fact that $\dtn\colon H^{1/2}(\pOmega)/\C\to H_{\diamond}^{-1/2}(\pOmega)$ and $\ntd\colon H_{\diamond}^{-1/2}(\pOmega)\to H^{1/2}(\pOmega)/\C$ are continuous inverses of each other; see, e.g., \cite[Sec.~12.3.2, pp.~164--165]{mueller2012linear} or \cite[Lemma~20, p.~197]{abraham2019statistical}.
    For any $\gamma\in\Gamma$, $\gamma_0\in\Gamma$, and distinguished representative $f$ of the $\C$-equivalence class $[f] \in H^{1/2}(\pOmega)/\C$, observe that
    \begin{align*}
        \dtn(\ntdz - \ntd)\dtnz f &= \dtn (\ntdz \dtnz f - \ntd\dtnz f)\\
        &=\dtn f - (\dtn\ntd)\dtnz f\\
        &=(\dtn - \dtnz)f\,.
    \end{align*}
    This implies the operator norm estimates
    \begin{align*}
        &\norm{\Lambda_{\gamma} - \Lambda_{\gamma_0}}_{\sL(H^{1/2}/\C;H^{-1/2})} = \norm{\dtn(\ntdz - \ntd)\dtnz}_{\sL(H^{1/2}/\C;H^{-1/2})}\\
        &\qquad\quad\quad\leq \norm{\dtn(\ntdz - \ntd)}_{\sL({H}_\diamond^{-1/2};H^{-1/2})}\norm{\dtnz}_{\sL(H^{1/2}/\C;H^{-1/2})}\\
        &\qquad\quad\quad\leq \norm{\dtn}_{\sL(H^{1/2}/\C;H^{-1/2})}\norm{\dtnz}_{\sL(H^{1/2}/\C;H^{-1/2})}\norm{\ntd - \ntdz}_{\sL({H}_\diamond^{-1/2};H^{1/2}/\C)}.
    \end{align*}
    By \Cref{lem:dtn_uniform}, there exists $c>0$ such that $\sup_{\gamma\in\Gamma} \norm{\dtn}_{\sL(H^{1/2}/\C;H^{-1/2})}\leq c$. Taking $C\defeq c^2$ completes the proof.
\end{proof}

The next lemma bounds the operator norms of NtD maps acting between boundary Sobolev spaces of arbitrarily small or large smoothness exponents, uniformly over the set $\Gamma'$ from \eqref{eqn:conductivity_set_prelim}. Such estimates are enabled by harmonic regularity in a neighborhood of the boundary for conductivities $\gamma\in\Gamma'$ because $\gamma\equiv 1$ on such a neighborhood. 
\begin{lemma}[uniform bounds on NtD maps]\label{lem:ntd_unif_op}
    For any $s\in\R$ and any $t\in\R$, there exist $C_s>0$ and $C_{s,t}'>0$ depending on $s$ and $t$ such that
    \begin{align}\label{eqn:ntd_unif_op}
        \begin{split}
            \sup_{\gamma\in\Gamma'}\norm{\ntd}_{\sL({H}_\diamond^s(\pOmega);H^{s+1}(\pOmega)/\C)}&\leq C_s \qa\\
            \sup_{\gamma\in\Gamma'}\norm{\ntd - \mathscr{R}_{\onebm}}_{\sL({H}_\diamond^s(\pOmega);H^{t}(\pOmega)/\C)}&\leq C_{s,t}'\,.
        \end{split}
    \end{align}
\end{lemma}

\Cref{app:proofs} provides the proof of \cref{lem:ntd_unif_op}. The arbitrarily high order smoothing properties \eqref{eqn:ntd_unif_op} of shifted NtD maps plays an important role in our quantitative extension theory. In particular, it helps us obtain the following H\"older-type continuity result for a particular embedding operator acting between spaces of linear operators.
\begin{lemma}[NtD inverse inequality: Hilbert--Schmidt norm to operator norm]\label{lem:hs_to_op}
    For any real numbers $s$ and $t$, there exists $C>0$ depending on $d$, $s$, $t$, $\Omega$, $\Omega'$, $m$, and $M$ such that for all $\gamma$ and $\gamma_0$ belonging to $\Gamma'$, it holds that
    \begin{align}\label{eqn:hs_to_op}
        \norm[\big]{\ntd-\sR_{\gamma_0}}_{\sL(H^{-1/2}_\diamond(\pOmega);H_{\phantom{\diamond}}^{1/2}(\pOmega)/\C)}\leq C\norm[\big]{\ntd-\sR_{\gamma_0}}^{1/2}_{\HS(H^{s}_\diamond(\pOmega);H^{t}(\pOmega)/\C)}\,.
    \end{align}
\end{lemma}
\begin{proof}
    The following argument mimics the proof of \cite[Lemma~5, p.~177]{abraham2019statistical} for a similar result applicable to DtN maps instead of NtD maps. Let $\gamma\in\Gamma'$ and $\gamma_0\in\Gamma'$ be arbitrary. We make use of low-rank approximations and the extreme smoothing properties of the shifted NtD maps $\tntd\defeq \ntd-\sR_{\onebm}$ and $\tntdz\defeq \ntdz - \sR_{\onebm}$.
    Write $\norm{\slot}$ for the $\HS(H_\diamond^{-1/2}(\pOmega);H^{1/2}(\pOmega)/\C)$ operator norm. To begin, we project onto finite-dimensional subspaces via $P_J=\Pi_{JJ}$ from \eqref{eqn:projection_operator}. For any $J$, this leads to
    \begin{align*}
        \norm{\ntd-\sR_{\gamma_0}}_{\sL(H_\diamond^{-1/2};H_{\phantom{\diamond}}^{1/2}/\C)}&\leq \norm{\ntd-{\sR}_{\gamma_0}}\\
        &=\norm{{\tntd} - {\tntdz}}\\
        &\leq \norm{{\tntd} - P_{J}\tntd} + \norm{P_{J}\tntd - P_{J}\tntdz}  + \norm{P_{J}\tntdz - \tntdz}\\
        &\leq 2\sup_{\gamma'\in\Gamma'}\norm{\widetilde{\sR}_{\gamma'} - P_{J}\widetilde{\sR}_{\gamma'}}  + \norm{P_{J}\tntd - P_{J}\tntdz}\,.
    \end{align*}
    
    Next, we control the truncation error represented by the first term in the last line of the preceding display. For any $\varrho \geq 0$, \cref{lem:op_to_hs_rn} (applied with $K=J$, $r=-1/2$, $s=1/2$, and $q=-p=\varrho(d-1) + 1/2$) delivers the bound
    \begin{align*}
        \norm{{\tntd} - P_{J}\tntd}\leq C \norm{\tntd}_{\sL(H^{-(\varrho+1)(d-1)- 1/2}_\diamond;H_{\phantom{\diamond}}^{\varrho(d-1) + 1/2}/\C)} J^{-\varrho}\leq C_{\varrho,d} J^{-\varrho}\,.
    \end{align*}
    The final inequality in the last display is due to \cref{lem:ntd_unif_op} (applied with $s=-(\varrho+1)(d-1)-1/2$ and $t=\varrho(d-1) + 1/2$). The constant $C_{\varrho,d}$ depends on $\varrho$ and $d$ but does not depend on $\gamma$. Thus, $C_{\varrho,d} J^{-\varrho}$ also upper bounds the supremum.
    
    Now fix real numbers $s$ and $t$. To control the second term in the error decomposition, we apply \cref{lem:ntd_norm_equiv} with $K=J$, $r=-1/2$, $s=1/2$, $p=s$, and $q=t$. This finite-dimensional norm equivalence result gives
    \begin{align*}
        \norm{P_J\tntd - P_J\tntdz}&\leq C J^{\frac{p(s,t)}{d-1}}\norm{\tntd -\tntdz}_{\HS(H^{s}_\diamond;H^{t}/\C)}\,,
    \end{align*}
    where $p(s,t)\defeq (s+1/2)_+ + (1/2-t)_+\geq 0$ and $a\mapsto a_+\defeq\max(0,a)$. 
    
    Finally, we balance the preceding two upper bounds with a judicious choice of truncation level $J$. For another constant $c_{\varrho,d}>0$, it holds that
    \begin{align*}
        x \defeq \bigl(c_{\varrho,d} \norm{\tntd -\tntdz}_{\HS(H^{s}_\diamond;H^{t}/\C)}\bigr)^{-\frac{d-1}{\varrho(d-1) + p(s,t)}}
    \end{align*}
    solves the equation $C_{\varrho,d}x^{-\varrho} = C x^{p(s,t)/(d-1)}\norm{\tntd -\tntdz}_{\HS(H^{s}_\diamond;H^{t}/\C)}$. It further holds that $x\geq {c_\varrho,d}'$ for a constant $c_{\varrho,d}'>0$. That is, $x$ is bounded away from zero uniformly in $\gamma$ and $\gamma_0$ because $\sup_{(\gamma,\gamma_0)\in\Gamma'\times\Gamma'}\norm{\tntd -\tntdz}_{\HS(H^{s}_\diamond;H^{t}/\C)}\leq C$ by the triangle inequality and \cref{lem:op_to_hs_rn,lem:ntd_unif_op}. Since $J$ was arbitrary, choosing the truncation level $J\defeq \ceil{x} \in \N$ implies that
    \begin{align*}
        \norm{\ntd-\sR_{\gamma_0}}_{\sL(H_\diamond^{-1/2};H_{\phantom{\diamond}}^{1/2}/\C)}&\leq C \norm[\big]{\tntd -\tntdz}_{\HS(H^{s}_\diamond;H^{t}/\C)}^{\frac{\varrho(d-1)}{\varrho(d-1) + p(s,t)}}\\
        &= C \norm[\big]{\ntd -\ntdz}_{\HS(H^{s}_\diamond;H^{t}/\C)}^{\frac{\varrho(d-1)}{\varrho(d-1) + p(s,t)}}\,.
    \end{align*}
    Setting $\varrho\defeq p(s,t)/(d-1)$ if $p(s,t)>0$ and $\varrho>0$ otherwise completes the proof.
\end{proof}
\Cref{lem:hs_to_op} is a kind of ``inverse inequality'' in our application of interest. We will be interested in smoothness exponents $t=0$ and $s$ close to zero. However, for general operators, the inequality only holds in the \emph{opposite direction} (without the square root) for certain values of non-zero $s$ and $t$. We are able to obtain \eqref{eqn:hs_to_op} for all values of $s$ and $t$ because shifted NtD maps are arbitrarily smoothing, as shown in \eqref{eqn:ntd_unif_op}.

\subsection{From Neumann-to-Dirichlet maps to kernel functions}\label{sec:extend_kernel}
Another advantage of working with NtD maps $\ntd$ instead of DtN maps $\dtn$ is that, viewed as mappings from $L^2_\diamond(\pOmega)$ into $L^2_\diamond(\pOmega)$, NtD maps are self-adjoint \emph{compact} linear operators. Thus, they are amenable to numerical discretization, e.g., finite element methods or finite spectral measurements. In two dimensions in particular, the NtD map $\ntd\in \sL(L^2_\diamond(\pOmega))$ is actually a Hilbert--Schmidt operator if $\gamma\in\Gamma$ \cite[Lemma~A.1, p.~1948]{garde2022series}. This fact is useful because every Hilbert--Schmidt operator on an $L^2$ space has an integral kernel \cite[Prop.~2.8.6, p.~74]{arveson2002short}. Thus, we may identify the linear operator $\ntd$ with its kernel function. Such an identification is natural for the purpose of learning the solution map of Calder\'on's problem because neural operators (in general) map functions to functions instead of operators to functions \cite{kovachki2021neural}.

However, if $d>2$, then $\ntd$ is no longer a Hilbert--Schmidt operator on $L^2$ \cite[p.~1948]{garde2022series}. Although there is work concerning neural operators that directly receive linear operators like $\ntd$ as input, these require specialized approximation architectures that are not yet fully understood \cite{castro2024calderon,de2019deep,abhishek2024solving}.
In the present paper, we nonetheless seek an identification between NtD maps and kernel functions in the higher-dimensional setting. To do so, we compose the NtD map with a $d$-dependent but otherwise fixed compact operator chosen such that their composition is Hilbert--Schmidt on $L^2$; similar ideas are carried out in \cite[Sec.~4.3.2]{kratsios2023transfer}. In the rest of this section, for $t\in\R$, we view $\ntd$ as a mapping into $H^t_\diamond(\pOmega;\C)$ instead of into the quotient space $H^{t}(\pOmega)/\C$. We do this by working with the mean-zero representative $f\in H_\diamond^{t}(\pOmega;\C)$ of any equivalence class $[f]\in H^{t}(\pOmega)/\C$ because their respective norms agree \cite[Appendix~A, Remark (iii)]{abraham2019statistical}. Thus, operator norms and Hilbert--Schmidt norms of the NtD maps viewed as acting on mean-zero spaces and on quotient spaces, respectively, also agree.

Recall the nondecreasing sequence of Laplace--Beltrami eigenvalues $\{\lambda_j\}_{j\in\Z_{\geq 0}}\subset\R_{\geq 0}$ and corresponding $L^2(\pOmega)$-orthonormal eigenfunctions $\{\varphi_j\}_{j\in\Z_{\geq 0}}$ from \cref{sec:prelim_func}. Define compact self-adjoint operator $\cC\colon L^2_\diamond(\pOmega;\C)\to L^2_\diamond(\pOmega;\C)$ by
\begin{align}\label{eqn:inverse_shifted_neumann_lap}
    \cC\defeq \sum_{j=1}^\infty(1+\lambda_j)^{-1}\varphi_j\otimes_{L^2(\pOmega;\C)}\varphi_j\,.
\end{align}
Since $\{(\lambda_j,\varphi_j)\}$ are real, $\cC$ maps real-valued functions to real-valued functions.
For $s>0$, we define powers $\cC^s$  of $\cC$ by functional calculus. The next lemma shows that the composition $\ntd \cC^s$ is Hilbert--Schmidt for an appropriate choice of $s$.
\begin{lemma}[NtD maps and Hilbert--Schmidt integral operators]\label{lem:ntd_hs_any_dim}
    Let $\tau>0$ and $d\geq 2$. Define $r_d\defeq \frac{d-1}{2}-1+\tau$. If $\gamma\in \Gamma'$, then $\ntd \cC^{r_d/2}\in\HS(L^2_\diamond(\pOmega))$.
\end{lemma}
\begin{proof}
    The squared $\HS(L^2_\diamond(\pOmega))$ norm of $\ntd \cC^{r_d/2}$ satisfies
    \begin{align*}
        \sum_{j=1}^{\infty}\norm[\big]{\ntd \cC^{r_d/2}\varphi_j}^2_{L^2}&\leq \norm{\ntd}_{\sL(H^{-1}_\diamond;L^2)}^2\sum_{j=1}^{\infty}\norm[\big]{\cC^{r_d/2}\varphi_j}^2_{H^{-1}}\\
        &=\norm{\ntd}_{\sL(H^{-1}_\diamond;L^2/\C)}^2\sum_{j=1}^{\infty}\norm[\bigg]{\frac{(1+\lambda_j)^{-r_d/2}}{\sqrt{1+\lambda_j}} \sqrt{1+\lambda_j}\varphi_j}^2_{H^{-1}}\\
          &=\norm{\ntd}_{\sL(H^{-1}_\diamond;L^2/\C)}^2\sum_{j=1}^{\infty}(1+\lambda_j)^{-1-r_d}\norm[\big]{\varphi_j^{(-1)}}^2_{H^{-1}}\,.
    \end{align*}
    The $\sL(H^{-1}_\diamond(\pOmega);L^2(\pOmega)/\C)$ operator norm of $\ntd$ is finite for any $d\geq 2$ by \cref{lem:ntd_unif_op}. Since $\norm{\varphi_j^{(-1)}}^2_{H^{-1}}=1$ for all $j$, Weyl's law \cite[Corollary~16, p.~194]{abraham2019statistical} shows that the series in the last line of the preceding display is bounded above and below by $\sum_{j=1}^{\infty}(1+j^{\frac{2}{d-1}})^{-1-r_d}$. This sum is finite if and only if $r_d>-1+(d-1)/2$. The assertion follows.
\end{proof}

\begin{remark}[weighted Hilbert--Schmidt spaces]
    It is known that the $\HS(L^2_\diamond(\pOmega))$ norm of $\ntd \cC^{r_d/2}$ is equal to the $\HS(H^{r_d}_\diamond(\pOmega);L^2_\diamond(\pOmega))$ norm of $\ntd$, which is a weighted Hilbert--Schmidt norm with information-theoretic relevance in Bayesian formulations of EIT~\cite{abraham2019statistical}. Additionally, these norms also equal the $L^2_\mu(L^2_\diamond(\pOmega))$ norm of $\ntd$ for any centered probability measure $\mu\in\sP(L^2_\diamond(\pOmega))$ with Mat\'ern-like covariance operator $\cC^{r_d}$ \cite[Sec.~2.2.1]{de2021convergence}. This viewpoint is especially relevant when the NtD map is regressed from paired data under squared loss, as is common in operator learning \cite{kovachki2024operator,boulle2024mathematical}.
\end{remark}

Since every Hilbert--Schmidt operator on $L^2$ is a kernel integral operator \cite[Prop.~2.8.6, p.~74]{arveson2002short}, \cref{lem:ntd_hs_any_dim} allows us to make the following definition.

\begin{definition}[integral kernel of smoothed NtD map]\label{def:kernel_integral}
    Let $\tau>0$ and $r_d\defeq \frac{d-1}{2}-1+\tau$. For any $\gamma\in\Gamma'$ and NtD map $\ntd$, let $\kappa_\gamma\in L^2(\pOmega\times\pOmega;\R)$ be the unique, real, integral (Schwartz) kernel of $\ntd \cC^{r_d/2}\in\HS(L^2_\diamond(\pOmega))$. That is, $\kappa_\gamma\colon\pOmega\times\pOmega\to\R $ is such that
    \begin{align}\label{eqn:ntd_action_as_integral}
        \ntd \cC^{r_d/2} g =\int_{\pOmega} \kappa_\gamma(\slot, y)g(y)\sigma(dy)
    \end{align}
    for all $g\in L^2_\diamond(\pOmega)$.
\end{definition}

The fact that $\kappa_\gamma$ is real-valued follows because both $\ntd$ and $\cC$ map real-valued functions to real-valued functions. Two physically relevant settings are when $d=2$ or $d=3$. When $d=2$, we can take $r_2=0$ (with the choice $\tau=1/2$) because $\ntd$ itself is a Hilbert--Schmidt operator with symmetric kernel $\kappa_\gamma$ in this case. If $d=3$, then we can take $r_3=\tau>0$ arbitrarily small. However, the kernel $\kappa_\gamma$ is not necessarily symmetric. A more convenient choice of $r_3$ for numerical acquisition of boundary data is $r_3=\tau=2$, especially when the boundary $\pOmega$ or eigendecomposition of $\lap_{\pOmega}$ are not known precisely. Specifically, let $A\defeq -\Delta_{\pOmega} + \Id_{L^2(\pOmega)}$ be the unbounded self-adjoint operator on $L^2_\diamond(\pOmega)$ with domain $\domain(A)=\{g\in L^2_\diamond(\pOmega)\,|\, \sum_{j=1}^\infty (1+\lambda_j)^2\ip{\varphi_j}{g}_{L^2}^2<\infty\}$. We can view $\cC=A^{-1}$. So, application of $\cC^{r_3/2}=\cC$ to a function $h$ requires solving the linear second order elliptic PDE $Ag=h$ for $g=\cC h$. Then $g$ becomes the boundary current input in a practical EIT imaging system.

\begin{remark}[integral kernels of shifted NtD maps]
    Theoretically, we could have worked with shifted NtD maps $\ntd-\sR_{\onebm}$. These shifted operators belong to $\HS(L^2_\diamond(\pOmega))$ in any dimension $d\geq 2$ by \cref{lem:ntd_unif_op} and the proof of \cref{lem:ntd_hs_any_dim}. Although mathematically convenient, it is known that this relative continuum measurement model is not clinically relevant in practice. Indeed, having access to the background NtD map $\sR_{\onebm}$ is not a realistic assumption \cite[Chp.~12.8.5, p.~182]{mueller2012linear}, and there is work that attempts to circumvent relative or difference EIT measurements \cite{garde2021mimicking}. We emphasize that the availability of relative NtD boundary data \emph{is not assumed} in our work.
\end{remark}

\subsection{Extending beyond the range of the forward map}\label{sec:extend_main}
To complete the extension argument using \cref{thm:hilbert_extension_ref}, it remains to specify Hilbert spaces $\cH_1$ and $\cH_2$ such that the inversion operator has a controlled concave modulus of continuity when viewed as a map between subsets of these two spaces. By virtue of \cref{thm:stability}, it is natural to take $\cH_2\equiv L^2(\Omega)$. However, the operator norm on the right-hand side of \eqref{eqn:stability_log} is non-Hilbertian. Thus, the choice of $\cH_1$ is not yet clear. The results in \cref{sec:extend_bounds} remedy this issue by bounding the operator norm by fractional powers of certain Hilbert--Schmidt norms, which are Hilbertian. Although the Hilbert space of Hilbert--Schmidt operators is sufficient from the perspective of verifying the hypotheses of \cref{thm:hilbert_extension_ref}, this space is still unsatisfactory from the machine learning perspective because linear operators are a nonstandard data type for most deep learning architectures (the exceptions being those in \cite{castro2024calderon,de2019deep,abhishek2024solving}; see also \cite[Sec.~4.2--4.3]{nelsen2025operator}). Similar to \cite{fan2020solving}, we apply the ideas from \cref{sec:extend_kernel} to identify NtD operators with certain square-integrable kernel functions. Consequently, we choose $\cH_1\equiv L^2(\pOmega\times\pOmega)$.

With the real Hilbert spaces $\cH_1$ and $\cH_2$ now fixed, we define the nonlinear EIT inversion operator with domain
\begin{align}\label{eqn:domain_of_inverse}
    \domain(\Psi^\star)\defeq \set{\kappa_\gamma}{\gamma\in\cX_d(R)} \subseteq L^2(\pOmega\times\pOmega)
\end{align}
by the expression
\begin{align}\label{eqn:inverse_map_final_def}
\begin{split}
    \Psi^\star\colon \domain(\Psi^\star)&\to L^2(\Omega)\,,\\ 
    \kappa_\gamma&\mapsto\gamma\,,
\end{split}
\end{align}
where $\kappa_\gamma$ is the integral kernel corresponding to a smoothed NtD map from \cref{def:kernel_integral}.
The domain $\domain(\Psi^\star)$ \eqref{eqn:domain_of_inverse} of the inverse map \eqref{eqn:inverse_map_final_def} is taken to be equal to the range of the forward map $\gamma\mapsto\kappa_\gamma$. This complicated set plays the role of $U$ in \cref{thm:hilbert_extension_ref}.
To complete the procedure outlined in the preceding discussion, we must bound the modulus of continuity of $\Psi^\star$. This requires an additional result.

\begin{lemma}[norm of operator composition]\label{lem:change_of_norm}
    Let $\cC$ be as in \eqref{eqn:inverse_shifted_neumann_lap}.
    For any real numbers $s$ and $t$, any $T\in\HS(H^s_\diamond(\pOmega);H^t_\diamond(\pOmega))$, and any $r\geq 0$, it holds that
    \begin{align}
        \norm[\big]{T}_{\HS(H^s_\diamond(\pOmega);H^t_\diamond(\pOmega))} = \norm[\big]{T\cC^{r/2}}_{\HS(H^{s-r}_\diamond(\pOmega);H^t_\diamond(\pOmega))}\,.
    \end{align}
\end{lemma}
\begin{proof}
    By the spectral decomposition of $\cC$ from \eqref{eqn:inverse_shifted_neumann_lap}, it holds that
    \begin{align*}
        \norm[\big]{T}^2_{\HS(H^s_\diamond;H^t_\diamond)}=\sum_{j=1}^\infty \norm[\big]{T\varphi_j^{(s)}}_{H^t}^2&=\sum_{j=1}^{\infty} (1+\lambda_j)^r(1+\lambda_j)^{-r}\norm[\big]{T(1+\lambda_j)^{-s/2}\varphi_j}_{H^t}^2\\
        &=\sum_{j=1}^{\infty} \norm[\big]{T\cC^{-r/2}(1+\lambda_j)^{-(s-r)/2}\varphi_j}_{H^t}^2\\
        &=\sum_{j=1}^{\infty}\norm[\big]{T\cC^{-r/2}\varphi_j^{(s-r)}}_{H^t}^2\,.
    \end{align*}
    This is the asserted result because $\{\varphi^{(s-r)}_j\}_{j=1}^\infty$ is an orthonormal basis of $H^{s-r}_\diamond$.
\end{proof}

Application of \cref{lem:change_of_norm} and the supporting results from \cref{sec:extend_stability,sec:extend_bounds,sec:extend_kernel} lead to the following logarithmic stability estimate for $\Psi^\star$.
\begin{proposition}[$L^2$ stability of $\Psi^\star$]\label{prop:stability_kernel}
    If $d\geq 2$ and $R>0$, then there exists a constant $C^\star>0$ such that for any $\gamma\in\cX_d(R)$ and any $\gamma_0\in\cX_d(R)$, it holds that
    \begin{align}\label{eqn:stability_kernel_in}
        \norm{\Psi^\star(\kappa_\gamma)-\Psi^\star(\kappa_{\gamma_0})}_{L^2(\Omega)}\leq C^\star \omega^\star\bigl(\norm{\kappa_{\gamma}-\kappa_{\gamma_0}}_{L^2(\partial\Omega\times\pOmega)}\bigr)\,,
    \end{align}
    where $t\mapsto \omega^\star(t)\defeq \omega(\sqrt{t})$ and $\omega$ is as in \eqref{eqn:modulus_log}.
\end{proposition}
\begin{proof}
    In the following, we let $C>0$ be a constant that is allowed to change from line to line. Let $\gamma$ and $\gamma_0$ belong to $\cX_d(R)$. Let $r_d$ be as in \cref{lem:ntd_hs_any_dim}. Since $\Psi^\star(\kappa_\gamma)=\gamma$ by definition \eqref{eqn:inverse_map_final_def} and moduli of continuity are increasing functions, successive application of \cref{thm:stability}, \cref{lem:ntd_to_dtn}, and \cref{lem:hs_to_op} (with $s=r_d$ and $t=0$) yield the inequalities
    \begin{align*}
        \norm{\Psi^\star(\kappa_\gamma)-\Psi^\star(\kappa_{\gamma_0})}_{L^2(\Omega)}&=\norm{\gamma-\gamma_0}_{L^2(\Omega)}\\
        &\leq C\omega\bigl(\norm{\dtn-\dtnz}_{\sL(H^{1/2}(\pOmega)/\C;H^{-1/2}(\pOmega))}\bigr)\\
        &\leq C\omega\bigl(\norm{\ntd-\ntdz}_{\sL(H^{-1/2}_\diamond(\pOmega);H_{\phantom{\diamond}}^{1/2}(\pOmega)/\C)}\bigr)\\
        &\leq C \omega\left(\norm{\ntd-\ntdz}^{1/2}_{\HS(H^{r_d}_\diamond(\pOmega);L^2(\pOmega)/\C)}\right)\,.
    \end{align*}
    Next, \cref{lem:change_of_norm}, \cref{lem:ntd_hs_any_dim}, and \cref{def:kernel_integral} imply that
    \begin{align*}
        \norm[\big]{\ntd-\ntdz}_{\HS(H^{r_d}_\diamond(\pOmega);L^2(\pOmega)/\C)}&=\norm[\big]{\ntd\cC^{r_d/2}-\ntdz\cC^{r_d/2}}_{\HS(L^2_\diamond(\pOmega))}\\
        &=\norm{\kappa_\gamma-\kappa_{\gamma_0}}_{L^2(\pOmega\times\pOmega)}\,.
    \end{align*}
    To achieve the last equality in the previous display, we used linearity and the fact that the mapping from a $L^2(\partial\Omega\times \partial\Omega)$ kernel function to the $\HS(L^2_\diamond(\partial\Omega)) $ integral operator with that same kernel function is an isometric isomorphism \cite[Prop.~2.8.6, p.~74]{arveson2002short}. Putting together the pieces, we deduce that
    \begin{align*}
        \norm{\Psi^\star(\kappa_\gamma)-\Psi^\star(\kappa_{\gamma_0})}_{L^2(\Omega)}\leq C\omega\bigl(\norm{\kappa_\gamma-\kappa_{\gamma_0}}_{L^2(\pOmega\times\pOmega)}^{1/2}\bigr)
    \end{align*}
    as asserted.
\end{proof}

With this stability result between Hilbert function spaces in hand, we are now in a position to apply the Benyamin--Lindenstrauss extension from \cref{thm:hilbert_extension_ref} to deliver the promised result. 

\begin{theorem}[extending outside of the range]\label{thm:extension_eit_kernel}
    Let $\Psi^\star\colon\domain(\Psi^\star)\subseteq L^2(\pOmega\times\pOmega)\to L^2(\Omega)$ be as in \eqref{eqn:domain_of_inverse} and \eqref{eqn:inverse_map_final_def}. Let $C^\star$ and $\omega^\star$ be as in \cref{prop:stability_kernel}. There exists $\widetilde{\Psi}^\star\colon L^2(\pOmega\times\pOmega)\to L^2(\Omega)$ such that $\widetilde{\Psi}^\star(\kappa)=\Psi^\star(\kappa)$ for all $\kappa\in\domain(\Psi^\star)$, and
    \begin{align}\label{eqn:stability_kernel_out}
        \norm{\widetilde{\Psi}^\star(\kappa)-\widetilde{\Psi}^\star(\kappa_0)}_{L^2(\Omega)}\leq C^\star \omega^\star\bigl(\norm{\kappa-\kappa_0}_{L^2(\partial\Omega\times\pOmega)}\bigr)
    \end{align}
    for all $\kappa\in L^2(\pOmega\times\pOmega)$ and all $\kappa_0\in L^2(\pOmega\times\pOmega)$.
\end{theorem}
\begin{proof}
    The assertion follows by applying \cref{thm:hilbert_extension_ref} with $\cH_1=L^2(\pOmega\times\pOmega)$, $\cH_2=L^2(\Omega)$, $U=\domain(\Psi^\star)$, $G=\Psi^\star$, and $\omega=\omega^\star$. That $\omega^\star$ is a modulus of continuity for $\Psi^\star$ is due to \cref{prop:stability_kernel}. Moreover, $\omega^\star$ is clearly nonnegative. It is also concave if the threshold $t_0$ from \eqref{eqn:modulus_log} is constructed to be sufficiently small; examination of the proof of \cref{thm:stability} and the results used therein show that this is always possible.
\end{proof}

Even though $\widetilde{\Psi}^\star$ from \cref{thm:extension_eit_kernel} extends $\Psi^\star$, in general $\widetilde{\Psi}^\star$ only maps into $L^2(\Omega)$ instead of into the admissible set $\cX_d(R)$. The next remark shows that the use of a thresholding operation can upgrade the range of the extension to $\Gamma\subset L^\infty(\Omega)$.

\begin{remark}[clipping]
    Define the clipping function $T_\Gamma\colon\R\to [m,M]$ by $x\mapsto T(x)\defeq \max(m,\min(M,x))$. Since $m\leq M$, it holds that $T_\Gamma$ is $1$-Lipschitz continuous on $\R$. The corresponding Nemytskii operator $\mathsf{T}_\Gamma\colon L^2(\Omega)\to L^2(\Omega)$ given by $\mathsf{T}_\Gamma(f)(x)=T_\gamma(f(x))$ is also $1$-Lipschitz. In \cref{thm:extension_eit_kernel}, we could have replaced $\widetilde{\Psi}^\star$ with $\widetilde{\Psi}^\star_{\Gamma}\defeq \mathsf{T}_\Gamma\circ \widetilde{\Psi}^\star$ everywhere. Indeed, $\widetilde{\Psi}^\star_{\Gamma}$ extends $\Psi^\star$ because $\widetilde{\Psi}^\star_{\Gamma}(\kappa_\gamma)=\mathsf{T}_\Gamma(\Psi^\star(\kappa_\gamma))=\Psi^\star(\kappa_\gamma)$ for any $\gamma\in\cX_d(R)\subseteq\Gamma$. The continuity bound \eqref{eqn:stability_kernel_out} also remains valid because $\mathsf{T}_\Gamma$ is Lipschitz.
    Such a clipped extension is interesting because its range satisfies $\image(\widetilde{\Psi}^\star_{\Gamma})\subseteq \Gamma\subset L^\infty(\Omega)$ instead of merely $\image(\widetilde{\Psi}^\star)\subseteq L^2(\Omega)$. However, for the purposes of the present paper, we do not need to consider clipping the extension $\widetilde{\Psi}^\star$, but remark that doing so may prove useful in other contexts. For example, the elliptic operator in \eqref{eqn:elliptic} with $\gamma\in \image(\widetilde{\Psi}^\star_{\Gamma})$ is well-defined, while it may not be with $\gamma\in \image(\widetilde{\Psi}^\star)$.
\end{remark}

\Cref{thm:extension_eit_kernel} is particularly useful when $\kappa\in\domain(\Psi^\star)$ and $\kappa_0\not\in\domain(\Psi^\star)$, but $\kappa_0$ is instead a small perturbation away from $\domain(\Psi^\star)$, e.g., $\kappa_0$ represents a noisy continuum EIT measurement. Indeed, the stability estimate \eqref{eqn:stability_kernel_out} controls how far the two reconstructions deviate by the size of the data perturbation. More importantly, \cref{thm:extension_eit_kernel} sets the stage for neural operator approximations of the true inverse map $\Psi^\star$ by way of the extension $\widetilde{\Psi}^\star$. We now turn our attention to this task.

\section{Approximation of the inverse map}\label{sec:approx}
This section employs the extension $\widetilde{\Psi}^\star$ from \cref{sec:extend} to establish an approximation theorem for the solution operator $\Psi^\star$ of EIT. The error is measured in the topology of uniform convergence over compact sets. Toward showing that the relevant set of NtD integral kernel functions is compact in $L^2(\pOmega\times\pOmega)$, \cref{sec:approx_stability} develops H\"older stability estimates for the forward map of EIT. These estimates help deliver the required compactness results in \cref{sec:approx_compactness}.
Another challenge is that off-the-shelf neural operators from \cref{sec:prelim_neural} do not natively handle functions defined on manifolds such as the boundary product $\pOmega\times\pOmega$. \Cref{sec:approx_manifold} sidesteps this issue by developing local coordinate operators for the manifold. \Cref{sec:approx_main} defines a slightly generalized FNO that is compatible with the manifold structure and presents the main result of the paper: the existence of such an FNO that approximates $\Psi^\star$ arbitrarily well, albeit up to the noise level in the boundary data. Finally, \cref{sec:approx_discuss} discusses the implications of the theorem.

\subsection{Stability of the forward map}\label{sec:approx_stability}
The goal of this subsection is to produce a forward stability estimate for the map $\gamma\mapsto \kappa_\gamma$.
To begin, we need the following $L^q(\Omega)$ Lipschitz continuity result for the Neumann problem~\cite[Thm.~4.2 and Remark~4.5, p.~406]{rondi2008regularization}.
\begin{theorem}[forward Lipschitz stability in $L^q(\Omega)$: operator norm]\label{thm:forward_map_lip}
    There exists $Q>2$ depending only on $d$, $m$, and $M$ such that the following holds. For any $q'\in (2,Q)$, define $q\in(2,\infty)$ by the equation $1/q +1/q' +1/2=1$. Then for any $\gamma$ and $\gamma_0$ belonging to $\Gamma'$, it holds that
    \begin{align}
        \norm{\ntd-\ntdz}_{\sL(H_\diamond^{-1/2}(\pOmega);H_{\phantom{\diamond}}^{1/2}(\pOmega)/\C)}\leq C\norm{\gamma-\gamma_0}_{L^q(\Omega)}\,.
    \end{align}
    The constant $C>0$ depends on $d$, $q'$, $m$, $M$, $\Omega$, and $\Omega'$.
\end{theorem}

To translate the preceding theorem for $\ntd$ into one applicable to $\kappa_\gamma$, we must interpolate the operator norm using Hilbert--Schmidt operator truncations as in \cref{lem:hs_to_op} from \cref{sec:extend_bounds}.
The following lemma is another such ``inverse inequality'' that is valid due to the smoothing properties of shifted NtD maps as shown in \cref{lem:ntd_unif_op}.
\begin{lemma}[NtD inverse inequality: operator norm to Hilbert--Schmidt norm]\label{lem:op_to_hs}
    For any real numbers $s$ and $t$, there exists $C>0$ depending on $d$, $s$, $t$, $\Omega$, $\Omega'$, $m$, and $M$ such that for all $\gamma$ and $\gamma_0$ belonging to $\Gamma'$, it holds that
    \begin{align}\label{eqn:op_to_hs}
        \norm{\ntd-\sR_{\gamma_0}}_{\HS(H^s_\diamond(\pOmega);H^t(\pOmega)/\C)}
        \leq 
        C\norm{\ntd-\sR_{\gamma_0}}^{1/2}_{\sL(H^{-1/2}_\diamond(\pOmega);H_{\phantom{\diamond}}^{1/2}(\pOmega)/\C)}\,.
    \end{align}
\end{lemma}
The proof of \cref{lem:op_to_hs} is in \cref{app:proofs}. As a consequence, we obtain our desired forward continuity result.

\begin{corollary}[forward H\"older stability in $L^p(\Omega)$: kernel norm]\label{cor:forward_stability_kernel}
    There exists $q>2$ depending only on $d$, $m$, and $M$ such that for any $p\in[1,\infty)$ and any $\gamma$ and $\gamma_0$ belonging to $\Gamma'$, it holds that
	\begin{align}\label{eqn:forward_stability_kernel}
		\norm[\big]{\kappa_\gamma-\kappa_{\gamma_0}}_{L^2(\partial\Omega\times\partial\Omega)}\leq C\norm[\big]{\gamma-\gamma_0}_{L^p(\Omega)}^{\frac12\min\left(1, \frac{p}{q}\right)}\,.
	\end{align}
	The constant $C>0$ depends on $d$, $p$, $q$, $m$, $M$, $\Omega$, and $\Omega'$.
\end{corollary}
\begin{proof}
    Application of \cref{thm:forward_map_lip} and \cref{lem:op_to_hs} delivers a $q>2$ such that
    \begin{align*}
        \norm{\ntd-\sR_{\gamma_0}}_{\HS(H^s_\diamond(\pOmega);H^t(\pOmega)/\C)}
            \leq 
            C\norm{\gamma-\gamma_0}^{1/2}_{L^q(\Omega)}\,.
    \end{align*}
    By \cref{lem:change_of_norm} and \cref{lem:ntd_hs_any_dim}, choosing $s=r_d$ and $t=0$ (or $s=r_d=0$ if $d=2)$ in the previous display gives
    \begin{align*}
        \norm{\ntd-\sR_{\gamma_0}}_{\HS(H^{r_d}_\diamond;L^2/\C)}
        =
        \norm{\ntd \cC^{r_d/2}-\ntdz \cC^{r_d/2}}_{\HS(L_\diamond^2)}=\norm{\kappa_\gamma-\kappa_{\gamma_0}}_{L^2(\partial\Omega\times\partial\Omega)}\,.
    \end{align*}
    If $p\geq q$, the asserted result \eqref{eqn:forward_stability_kernel} follows by monotonicity of $L^p$ norms. Otherwise, $p<q$ and we use the topological equivalence of the $L^p(\Omega)$ distances on $\Gamma\supseteq\Gamma'$. 
    In this case, for any $\gamma$ and $\gamma_0$ in $\Gamma'$, it holds that
    \begin{align*}
        \norm{\gamma-\gamma_0}_{L^q(\Omega)}&=\biggl(\int_\Omega\abs{\gamma(x)-\gamma_0(x)}^{q-p}\abs{\gamma(x)-\gamma_0(x)}^{p}\dd{x}\biggr)^{1/q}\\
        &\leq (2M)^{1-p/q}\norm{\gamma-\gamma_0}_{L^p(\Omega)}^{p/q}\,.
    \end{align*}
    Putting together the pieces completes the proof.
\end{proof}

In summary, \cref{cor:forward_stability_kernel} establishes that for any $p\in [1,\infty)$, the forward map 
\begin{align}\label{eqn:forward_map_kernel_bigger}
\begin{split}
    F\colon (\Gamma',\sfd_p)&\to L^2(\pOmega\times\pOmega)\,,\\
    \gamma&\mapsto F(\gamma)\defeq \kappa_\gamma\,,
\end{split}
\end{align}
is H\"older continuous. In particular, $F$ is continuous, which is all we will need. This continuity remains valid for the restriction
\begin{align}\label{eqn:forward_map_kernel}
    \sfF \defeq F|_{\cX_d(R)}\colon  (\cX_d(R),\sfd_p)\to L^2(\pOmega\times\pOmega)
\end{align}
of $F$ to the admissible set of conductivities $\cX_d(R)\subseteq \Gamma'$. We now turn our attention to the compactness of this set.

\subsection{Compactness of the admissible conductivities}\label{sec:approx_compactness}
Our next aim is to establish that $\domain(\Psi^\star)$ is compact in the $L^2(\pOmega\times\pOmega)$ topology because this will be the input space for our neural operator approximations. By the definition $\eqref{eqn:domain_of_inverse}$ of the domain $\domain(\Psi^\star)$, we see that $\domain(\Psi^\star)$ is equal to the range $\image(\sfF)$ of the continuous forward map $\sfF$ from \eqref{eqn:forward_map_kernel}. Therefore, the compactness of $\domain(\Psi^\star)=\image(\sfF)=\sfF(\cX_d(R))$ will follow if we can show that the set $\cX_d(R)$ of admissible conductivities is compact in $L^p(\Omega)$. We need the following lemma.

\begin{lemma}[the admissible conductivity set is closed]\label{lem:closed_set}
    For any $p\in[1,\infty)$ and any $d\geq 2$, it holds that $\cX_d(R)$ is closed in $L^p(\Omega)$.
\end{lemma}
\begin{proof}
    Clearly $\cX_d(R)\subseteq L^p(\Omega)$ by definition \eqref{eqn:conductivity_set}. To show that $\cX_d(R)$ is closed in $L^p(\Omega)$, let $\{\gamma_n\}_{n\in\N}\subseteq \cX_d(R)$ be a sequence such that $\gamma_n\to\gamma\in L^p(\Omega)$ in $L^p(\Omega)$ as $n\to\infty$. First, we show that $\gamma\in\Gamma$. Since $\Omega$ has finite Lebesgue measure, convergence in $L^p(\Omega)$ implies the existence of a subsequence $\{\gamma_{n_k}\}$ such that $\gamma_{n_k}(x)\to\gamma(x)$ as $k\to\infty$ for a.e. $x\in\Omega$. Thus,
    \begin{align*}
        \abs{\gamma(x)}\leq \abs{\gamma_{n_k}(x)} + \abs{\gamma(x)-\gamma_{n_k}(x)}
        \leq M + \abs{\gamma(x)-\gamma_{n_k}(x)}\,.
    \end{align*}
    Taking the limit superior shows that $\gamma(x)\leq\abs{\gamma(x)}\leq M$ for a.e. $x\in\Omega$. Thus, $\gamma\in L^\infty(\Omega)$. The strictly positive lower bound is similar because
    \begin{align*}
        \gamma(x)\geq \gamma_{n_k}(x) - \abs{\gamma_{n_k}(x)-\gamma(x)}
        \geq m - \abs{\gamma_{n_k}(x)-\gamma(x)}\,.
    \end{align*}
    Taking the limit inferior, we deduce that $\gamma\in\Gamma$.
    
    Next, we show that $\gamma\in \Gamma'$ is identically one in a neighborhood of the boundary. Let $E=\Omega\setminus\Omega'$, where $\overline{\Omega'}\subset\Omega$. By hypothesis, $\gamma_n=\onebm$ a.e. on $E$. Thus
    \begin{align*}
        \int_E\abs{1-\gamma(x)}^p\dd{x}=\int_E\abs{\gamma_n(x)-\gamma(x)}^p\dd{x}\leq\int_\Omega\abs{\gamma_n(x)-\gamma(x)}^p\dd{x}
    \end{align*}
    for any $n$. By taking the limit superior, it follows that $\int_E\abs{1-\gamma(x)}^p\dd{x}=0$. This implies that $\gamma=\onebm$ a.e. on $E$ as required. Hence, $\gamma\in\Gamma'$.

    To conclude that $\gamma\in\cX_d(R)$, we must prove that the $L^p$-limit $\gamma$ satisfies the \emph{a priori} bounds in the definition~\eqref{eqn:conductivity_set} of $\cX_d(R)$. Beginning with dimension $d=2$, we must establish the BV norm bound. Since $\sup_{n\in\N}\norm{\gamma_n}_{\BV}\leq R$, it holds that
    \begin{align*}
        \norm{\gamma}_{L^1(\Omega)}=\int_\Omega\lim_{k\to\infty}\abs{\gamma_{n_k}(x)}\dd{x}
        \leq \liminf_{k\to\infty}\int_\Omega\abs{\gamma_{n_k}(x)}\dd{x}
        = \liminf_{k\to\infty}\norm{\gamma_{n_k}}_{L^1(\Omega)}\leq R
    \end{align*}
    by Fatou's lemma. Next, the total variation functional $V(\slot;\Omega)\colon L^1_{\mathrm{loc}}(\Omega)\to\R_{\geq 0}\cup\{\infty\}$ is lower semi-continuous in $L^1_{\mathrm{loc}}(\Omega)$ \cite[Exercise~14.3~(iii), p.~460]{leoni2017first}. Since $p\geq 1$, the sequence $\{\gamma_n\}_{n\in\N}$ and limit $\gamma$ both belong to $L^1_{\mathrm{loc}}(\Omega)$. It follows that $V(\gamma;\Omega)\leq \liminf_{n\to\infty} V(\gamma_n;\Omega)\leq R$. Therefore, $\norm{\gamma}_{\BV}\leq R$ as desired.

    Now consider $d\geq 3$. We have $\sup_{n\in\N}\norm{\gamma_n}_{C^{1,\al}}\leq R$. Let $\beta<\al$. By the compactness of the embedding $C^{1,\al}(\overline{\Omega})\hookrightarrow C^{1,\beta}(\overline{\Omega})$ \cite[Thm.~1.34, pp.~11--12]{adams2003sobolev}, there exists a subsequence $\{\gamma_{n_k}\}_{k\in\N}$ and a limit $\gamma'\in C^{1,\beta}(\overline{\Omega})$ such that $\lim_{k\to\infty}\norm{\gamma_{n_k}-\gamma'}_{C^{1,\beta}}=0$. Thus, it also holds that
    \begin{align*}
        \norm{\gamma_{n_k}-\gamma'}_{L^p}\lesssim \norm{\gamma_{n_k}-\gamma'}_{L^\infty}=\sup_{x\in\Omega}\abs{\gamma_{n_k}(x)-\gamma'(x)}\leq \norm{\gamma_{n_k}-\gamma'}_{C^{1,\beta}}\to 0
    \end{align*}
    as $k\to\infty$. By uniqueness of limits, we deduce that $\gamma=\gamma'$ a.e. on $\Omega$ and $\gamma=\gamma'$ in $L^p$. Identifying $\gamma$ with its continuous representative $\gamma'$, in particular it holds that $\gamma_{n_k}\to\gamma$ as $k\to\infty$ in $C^1(\overline{\Omega})$. Thus
    \begin{align*}
        \norm{\gamma}_{C^1}\leq \limsup_{k\to\infty}\norm{\gamma_{n_k}}_{C^1} + \limsup_{k\to\infty}\norm{\gamma - \gamma_{n_k}}_{C^1}\leq R\,.
    \end{align*}
    Finally, let $\nu$ be a multi-index with $\abs{\nu}=1$ and let $(x,x')\in\Omega\times\Omega$. Then
    \begin{align*}
        \abs[\big]{(\partial^\nu\gamma)(x)-(\partial^\nu\gamma)(x')}=\limsup_{k\to\infty}\abs[\big]{(\partial^\nu\gamma_{n_k})(x)-(\partial^\nu\gamma_{n_k})(x')}\leq R\, \abs{x-x'}^\al
    \end{align*}
    because $\{\gamma_{n_k}\}_{k\in\N}\subseteq \cX_d(R)$. Since $\nu$, $x$, and $x'$ are arbitrary, the assertion $\norm{\gamma}_{C^{1,\al}}\leq R $ follows. This completes the proof.
\end{proof}

With the help of \cref{lem:closed_set}, we can now prove that $\cX_d(R)\subseteq L^p(\Omega)$ is compact.
\begin{lemma}[compactness of admissible conductivity set]\label{lem:compact_set}
    For any $1\leq p< \infty$ and any $d\geq 2$, the set $\cX_d(R)$ is compact in $L^p(\Omega)$.
\end{lemma}
\begin{proof}
Since $\cX_d(R)$ is closed in the $L^p(\Omega)$ topology by \cref{lem:closed_set}, compactness will follow by establishing relative compactness (i.e., precompactness) in $L^p(\Omega)$. 

First let $d=2$. In this case, our argument is similar to that in \cite[Appendix C]{bhattacharya2023learning}. The proof uses the Frech\'et--Kolmogorov--Riesz (FKR) compactness theorem \cite[Thm.~4.26, Sec.~4.5, p.~111]{brezis2011functional}. This theorem requires $L^p(\R^d)$ boundedness and $L^p(\R^d)$ equicontinuity estimates on the whole of $\R^d$. To this end, define the set $\widetilde{\cX}\defeq \set{\widetilde{\gamma}-\onebm \colon\R^2\to\R}{\gamma\in\cX_2(R)}$, where $\widetilde{\gamma}\colon\R^2\to\R$ is the continuous extension of $\gamma$ by one outside of $\Omega$. For the boundedness over $\widetilde{\cX}$, notice that
\begin{align*}
    \norm{\widetilde{\gamma}-\onebm}_{L^p(\R^d)}=\norm{\gamma-\onebm}_{L^p(\Omega)}\leq  (M+1)\abs{\Omega}^{1/p}
\end{align*}
because $\widetilde{\gamma}-\onebm\equiv 0$ outside of $\Omega$. For equicontinuity on average, we first recall \eqref{eqn:modulus_sup} and the translation-by-$y$ operator $\tau_y$ defined by $(\tau_y f)(x)= f(x+y)$ for fixed $y\in\R^d\setminus\{0\}$ and for any $x\in\R^d$. For any $h\in\R^d$, we compute that
\begin{align*}
    \norm{\tau_h(\widetilde{\gamma}-\onebm) - (\widetilde{\gamma}-\onebm)}_{L^\infty(\R^d)}=\norm{\tau_h\widetilde{\gamma}-\widetilde{\gamma}}_{L^\infty(\R^d)}
    \leq 2\norm{\widetilde{\gamma}}_{L^\infty(\R^d)}
    \leq 2 M\,.
\end{align*}
Also, by an argument similar to the one used in the proof of \cref{thm:stability},
\begin{align*}
    \norm{\tau_h(\widetilde{\gamma}-\onebm) - (\widetilde{\gamma}-\onebm)}_{L^1(\R^d)}=\norm{\tau_h\widetilde{\gamma}-\widetilde{\gamma}}_{L^1(\R^d)}
    \leq V(\widetilde{\gamma};\R^d)\abs{h}
    = V(\gamma;\Omega)\abs{h}\leq R\abs{h}\,.
\end{align*}
We again invoked \cref{lem:tv_identity}. By interpolating the $L^1$ and $L^\infty$ norms, it follows that
\begin{align*}
    \norm{\tau_h(\widetilde{\gamma}-\onebm) - (\widetilde{\gamma}-\onebm)}_{L^p(\R^d)}&\leq \norm{\tau_h\widetilde{\gamma} - \widetilde{\gamma}}_{L^\infty(\R^d)}^{1-1/p}\norm{\tau_h\widetilde{\gamma} - \widetilde{\gamma}}_{L^1(\R^d)}^{1/p}\\
    &\leq (2 M)^{1-1/p}R^{1/p}\abs{h}^{1/p}\,,
\end{align*}
which tends to zero uniformly over $\widetilde{\cX}$ as $\abs{h}\to0$. Then, the FKR compactness theorem asserts that $\widetilde{\cX}|_{\Omega}\defeq \set{\gamma-\onebm\colon \Omega\to \R}{\gamma\in\cX_2(R)}$ is relatively compact in $L^p(\Omega)$. From this fact, it is easy to see that $\cX_2(R)$ itself is relatively compact in $L^p(\Omega)$.

Now let $d\geq 3$. Let $\{\gamma_n\}\subseteq\cX_d(R)$ be a sequence. Since $\sup_{n\in\N}\norm{\gamma_n}_{C^{1,\al}}\leq R$, the compactness of the embedding $C^{1,\al}(\overline{\Omega})\hookrightarrow C^{1,\beta}(\overline{\Omega})$ for $\beta<\al$ \cite[Thm.~1.34, pp.~11--12]{adams2003sobolev} delivers the existence of a subsequence $\{\gamma_{n_k}\}_{k\in\N}$ and a limit $\gamma'\in C^{1,\beta}(\overline{\Omega})$ such that
\begin{align*}
    \norm{\gamma_{n_k}-\gamma'}_{L^p}\lesssim \norm{\gamma_{n_k}-\gamma'}_{L^\infty}=\sup_{x\in\Omega}\abs{\gamma_{n_k}(x)-\gamma'(x)}\leq \norm{\gamma_{n_k}-\gamma'}_{C^{1,\beta}}\to 0
\end{align*}
as $k\to\infty$. This proves that $\cX_d(R)$ is relatively compact in $L^p(\Omega)$ for $d\geq 3$.
\end{proof}

By combining \cref{cor:forward_stability_kernel} with \cref{lem:compact_set}, we obtain our desired compactness result in the space of kernel functions.
\begin{corollary}[compactness of integral kernels]\label{cor:compact_set_kernel}
    The set $\domain(\Psi^\star)\subseteq L^2(\pOmega\times\pOmega)$ is compact in $L^2(\pOmega\times\pOmega)$.
\end{corollary}
\Cref{cor:compact_set_kernel} plays an important role in the neural operator approximation theory to follow.

\subsection{Local representation of the boundary manifold}\label{sec:approx_manifold}
The neural operators as described in \cref{sec:prelim_neural} map between spaces of functions defined over \emph{bounded domains}. However, in the context of EIT, inputs to the inverse map \eqref{eqn:inverse_map_final_def} are functions defined on $\pOmega\times\pOmega$. The set $\pOmega\times\pOmega$ is not a domain, but is instead a \emph{compact manifold} possibly representing complicated geometry. This fact requires additional care. Using basic properties of manifolds \cite{tu2011manifold}, the goal of this subsection is to construct an injective linear transformation from functions on the manifold to functions on a latent, nonphysical domain.

To this end, recall that $\Omega\subset\R^d$ is the bounded physical domain with $C^\infty$-smooth boundary. It follows that $\pOmega\subset\R^d $ is a $C^\infty$-smooth $(d-1)$-dimensional manifold without boundary that is compact in the subspace topology of $\R^d$ \cite{tu2011manifold}. Let $\{(U_j,\phi_j)\}_{j=1}^J$ be a finite atlas of charts for $\pOmega$. The coordinate neighborhoods $U_j\subset \pOmega$ are open sets with the property that $\cup_{j=1}^J U_j=\pOmega$. Let the diffeomorphic coordinate maps $\phi_j\colon U_j\to\phi_j(U_j)\subset\R^{d-1}$ be chosen such that the coordinate patches $W_j\defeq \phi_j(U_j)$ satisfy $W_j\subseteq \mathfrak{D}$ for all $j$, where $\mathfrak{D}$ is a latent Lipschitz domain with the property that $\overline{\mathfrak{D}}\subset (0,1)^{d-1}$.
Additionally, we require the following assumption.
\begin{assumption}[bi-Lipschitz coordinate maps]\label{ass:bilipschitz}
    The domain $\Omega$ is such that the boundary manifold $\pOmega$ admits bi-Lipschitz coordinate maps $\{\phi_j\}_{j=1}^J$.
\end{assumption}

We next operate on the $C^\infty$-smooth $(2d-2)$-dimensional compact manifold $\pOmega\times\pOmega\subset\R^d\times\R^d $ \cite[Prop.~5.18, p.~56]{tu2011manifold}.
For every $j$ and $k$, define the diffeomorphism
\begin{align}\label{eqn:coord_backward}
    \begin{split}
        \Phi_{jk}^{-1}\colon W_j\times W_k&\to U_j\times U_k\\
        (x,y)&\mapsto \bigl(\phi_j^{-1}(x),\phi_k^{-1}(y)\bigr)
    \end{split}
\end{align}
and its associated pullback linear operator
\begin{align}\label{eqn:coord_backward_pull}
\begin{split}
    (\Phi_{jk}^{-1})^*\colon L^2(\pOmega\times\pOmega) &\to L^2(W_j\times W_k)\\
    h&\mapsto h\circ \Phi_{jk}^{-1}\,.
\end{split}
\end{align}
The linear operator $(\Phi_{jk}^{-1})^*$ is continuous by the change of variables integration formula and the bi-Lipschitz property of the coordinate maps from \cref{ass:bilipschitz}. For any $h$, we also define the zero-extended continuous pullback operator $(\widetilde{\Phi}_{jk}^{-1})^*\colon L^2(\pOmega\times\pOmega) \to L^2(\mathfrak{D}\times \mathfrak{D})$ by
\begin{align}\label{eqn:pull_inv_zero}
    (x,y)\mapsto \bigl((\widetilde{\Phi}_{jk}^{-1})^* h\bigr)(x,y)\defeq
    \begin{cases}
        \bigl(({\Phi}_{jk}^{-1})^* h\bigr)(x,y)\, , & \text{if }\, (x,y)\in W_j\times W_k\,,\\
        0\, , & \text{otherwise}\,.
    \end{cases}
\end{align}
Stacking together all $J^2$ pullback operators $\{(\widetilde{\Phi}_{jk}^{-1})^*\}_{j,k=1}^J$ defines the continuous global-to-local linear map
\begin{align}\label{eqn:global_to_local}
    \begin{split}
        {\sfG}\colon L^2(\pOmega\times\pOmega)&\to L^2(\mathfrak{D}\times \mathfrak{D})^{J\times J}\\
        h&\mapsto {\sfG} h\defeq \bigl\{(\widetilde{\Phi}_{jk}^{-1})^* h\bigr\}_{j,k=1}^J\,.
    \end{split}
\end{align}
We may identify $L^2(\mathfrak{D}\times \mathfrak{D})^{J\times J}$ with $L^2(\mathfrak{D}\times \mathfrak{D};\R^{J^2})$ via the natural isometric isomorphism.
The global-to-local map $\sfG$ allows us to move off of the manifold and work in $J^2$ Euclidean patches. In particular, standard neural operators are applicable on such patches.

It remains to define a suitable inverse of the global-to-local map in order to return to the manifold. We take an approach based on smooth partitions of unity. Standard results \cite[Prop.~13.6, p.~146]{tu2011manifold} show that there exists a partition of unity $\{\psi_j\}_{j=1}^J$ subordinate to $\{U_j\}_{j=1}^J$ such that $\psi_j\in C_c^\infty(U_j)$ for every $j$. The functions $\{\psi_j\}$ are nonnegative and sum to one pointwise.
Define the linear multiplication operators $\sfM_{jk}\colon L^2(\pOmega\times\pOmega) \to  L^2(\pOmega\times\pOmega)$ by
\begin{align}\label{eqn:multiplication}
    (\sfM_{jk} h)(x,y)\defeq \psi_j(x)\psi_k(y)h(x,y)
\end{align}
for every $h$, $x$, and $y$. Each $\sfM_{jk}h$ is supported in $U_j\times U_k$.
The multiplication operators are continuous because $\psi_j$ has compact support in $U_j$ for all $j$ and hence $\norm{\psi_j}_{L^\infty(U_j)}<\infty$.
Now define the diffeomorphism
\begin{align}
    \begin{split}
        \Phi_{jk}\colon U_j\times U_k&\to W_j\times W_k\\
        (x,y)&\mapsto \bigl(\phi_j(x),\phi_k(y)\bigr)
    \end{split}
\end{align}
and its associated continuous pullback linear operators
\begin{align}
    \begin{split}
        \Phi_{jk}^*\colon L^2(\mathfrak{D}\times\mathfrak{D}) &\to L^2(U_j\times U_k)\\
        h&\mapsto h\circ \Phi_{jk}
    \end{split}
\end{align}
and $\widetilde{\Phi}_{jk}^*\colon L^2(\mathfrak{D}\times\mathfrak{D})\to L^2(\pOmega\times\pOmega)$. The latter is given by
\begin{align}
    (x,y)\mapsto (\widetilde{\Phi}_{jk}^* h)(x,y)\defeq
    \begin{cases}
        ({\Phi}_{jk}^* h)(x,y)\, , & \text{if }\, (x,y)\in U_j\times U_k\,,\\
        0\, , & \text{otherwise}\,.
    \end{cases}
\end{align}
Finally, we define the linear local-to-global map
\begin{align}\label{eqn:local_to_global}
    \begin{split}
        {\sfL}\colon L^2(\mathfrak{D}\times \mathfrak{D})^{J\times J}&\to L^2(\pOmega\times\pOmega)\\
        \bigl\{f_{jk}\}_{j,k=1}^J&\mapsto  \sum_{j,k=1}^J \sfM_{jk}\widetilde{\Phi}_{jk}^* f_{jk}\,.
    \end{split}
\end{align}
The map $\sfL$ is continuous because it is the composition of continuous linear maps.
The following lemma shows that ${\sfL}$ is a left inverse of ${\sfG}$.
\begin{lemma}[left inverse]\label{lem:left_inv_pou}
Under \cref{ass:bilipschitz}, it holds that ${\sfL}{\sfG}=\Id_{L^2(\pOmega\times\pOmega)}$.
\end{lemma}
\begin{proof}
    Let $h\in L^2(\pOmega\times\pOmega)$ and $(x,y)\in\pOmega\times\pOmega$. By definition,
    \begin{align*}
        ({\sfL}{\sfG}h)(x,y)&=\sum_{j,k=1}^J \bigl(\sfM_{jk}\widetilde{\Phi}_{jk}^*(\widetilde{\Phi}_{jk}^{-1})^* h\bigr)(x,y)\\
        &=\sum_{j,k=1}^J \psi_j(x)\psi_k(y)\bigl(\widetilde{\Phi}_{jk}^*(\widetilde{\Phi}_{jk}^{-1})^* h\bigr)(x,y)\\
        &=\sum_{(j,k)\in\set{(m,n)}{(x,y)\in U_m\times U_n}}
        \psi_j(x)\psi_k(y)\bigl((\widetilde{\Phi}_{jk}^{-1})^* h\bigr)\bigl(\phi_j(x),\phi_k(y)\bigr)\,.
    \end{align*}
    Also, \eqref{eqn:pull_inv_zero} implies that
    \begin{align*}
        \bigl((\widetilde{\Phi}_{jk}^{-1})^* h\bigr)\bigl(\phi_j(x),\phi_k(y)\bigr)&=\bigl(({\Phi}_{jk}^{-1})^* h\bigr)\bigl(\phi_j(x),\phi_k(y)\bigr)\\
        &=h(x,y)\,.
    \end{align*}
    We deduce that
    \begin{align*}
        ({\sfL}{\sfG}h)(x,y)=\sum_{(j,k)\in\set{(m,n)}{(x,y)\in U_m\times U_n}}
        \psi_j(x)\psi_k(y) h(x,y)=h(x,y)
    \end{align*}
    because $\{\sfM_{jk}\onebm\}_{j,k=1}^J$ is a partition of unity subordinate to $\{U_j\times U_k\}_{j,k=1}^J$ with $\sfM_{jk}\onebm\in C_c^\infty(U_j\times U_k)$. Since $h$, $x$, and $y$ are arbitrary, the assertion is proved.
\end{proof}

With the maps $\sfG$ and $\sfL$ in hand, we are now in a position to develop approximation theory for EIT.

\subsection{Main result on neural operator approximation}\label{sec:approx_main}
This subsection approximates the (extended) EIT inverse map $\widetilde{\Psi}^\star\colon L^2(\pOmega\times\pOmega)\to L^2(\Omega)$ from \cref{thm:extension_eit_kernel} with neural operators. As noted in \cref{sec:approx_manifold}, the standard neural operators defined in \cref{sec:prelim_neural} are only capable of processing functions defined on domains. However, $\pOmega\times\pOmega$ is a compact manifold. We therefore must generalize the definition of FNO to encompass input functions defined on manifolds.

To this end, let $\cM\subset\R^{d_\mathrm{a}}$ be a $d_{\mathrm{m}}$-dimensional compact manifold embedded in $\R^{d_\mathrm{a}}$. Suppose that $d_\mathrm{m}\geq d$, where $\Omega\subset\R^d$. Let $d_\mathrm{e}$, $d_\mathrm{c}$, $d_\mathrm{r}$, and $d_\mathrm{q}$ be natural numbers. We define a \emph{generalized FNO} $\Psi^{(\mathrm{FNO})}\colon L^2(\cM;\R^{d_{\mathrm{in}}})\to L^2(\Omega;\R^{d_{\mathrm{out}}})$ by
\begin{align}\label{eqn:fno_eit}
    \Psi^{(\mathrm{FNO})}\defeq \cQ\circ\cR\circ \cF\circ\cL_L\circ\cdots\circ\cL_1\circ\cE\circ\cS\,.
\end{align}
Comparing \eqref{eqn:fno_eit} to \eqref{eqn:fno_torus}, the layers $\cS$, $\{\cL_\ell\}_{\ell=1}^L$, and $\cQ$ agree with those of the standard neural operator. The first layer $\cS\colon L^2(\cM;\R^{d_{\mathrm{in}}})\to L^2(\cM;\R^{d_{\mathrm{e}}})$ is pointwise defined by a map $S\colon \R^{d_\mathrm{a}}\times \R^{d_{\mathrm{in}}}\to \R^{d_{\mathrm{e}}}$. This means
\begin{align}\label{eqn:pw_lift}
    \bigl(\cS (h)\bigr)(x)\defeq S\bigl(x,h(x)\bigr)
\end{align}
for every $h\in L^2(\cM;\R^{d_{\mathrm{in}}})$ and $x\in\cM$. Each nonlinear operator $\cL_\ell\colon L^2(\T^{d_\mathrm{m}};\R^{d_{\mathrm{c}}})\to L^2(\T^{d_\mathrm{m}};\R^{d_{\mathrm{c}}})$ is as in \eqref{eqn:no_layers} and \eqref{eqn:fno_layer}.
The final operator $\cQ\colon L^2(\Omega;\R^{d_{\mathrm{q}}})\to L^2(\Omega;\R^{d_{\mathrm{out}}})$ is also defined pointwise by
\begin{align}\label{eqn:pw_proj}
    \bigl(\cQ (h)\bigr)(x)\defeq Q\bigl(h(x)\bigr)
\end{align}
for some map $Q\colon \R^{d_{\mathrm{q}}}\to \R^{d_{\mathrm{out}}}$.

However, \eqref{eqn:fno_eit} also introduces $\cE$, $\cF$, and $\cR$.
The operator $\cE\colon L^2(\cM;\R^{d_{\mathrm{e}}})\to L^2(\T^{d_\mathrm{m}};\R^{d_{\mathrm{c}}})$ is continuous and linear. Its role is to transform functions on the manifold to functions on the torus (recall that $\T\simeq[0,1]_{\mathrm{per}}$). As a special case, one can build such an $\cE$ from the global-to-local map $\sfG$ from \eqref{eqn:global_to_local} in \cref{sec:approx_manifold}.
The role of the continuous and linear map $\cF\colon  L^2(\T^{d_\mathrm{m}};\R^{d_{\mathrm{c}}})\to L^2(\T^{d};\R^{d_{\mathrm{r}}})$ is to change the dimension of the input and output domain. To do this, we adopt the linear functional layer of \emph{Fourier neural functionals} \cite{huang2025operator}. For $h\in L^2(\T^{d_\mathrm{m}};\R^{d_{\mathrm{c}}})$ and $x=(x_1,\ldots, x_d)\in\T^d$, define $(\cF h)(x_1,\ldots, x_d)$ to be equal to
\begin{align}\label{eqn:functional_layer}
    \int_{\T^{d_\mathrm{m}-d}} \kappa(x_{d+1},\ldots, x_{d_\mathrm{m}})h(x_1,\ldots, x_d, x_{d+1},\ldots, x_{d_\mathrm{m}})\dd{x_{d+1}}\cdots \dd{x_{d_\mathrm{m}}}\,.
\end{align}
The periodic function $\kappa\colon\T^{d_\mathrm{m}-d}\to\R^{d_\mathrm{r}\times d_\mathrm{c}}$ parametrizes $\cF$. In turn, $\kappa$ is parametrized by its Fourier coefficients, enabling the fast computation of $h\mapsto \cF h$ as an inner product in Fourier space~\cite[Sec.~2.2, Eqn.~6, p.~7]{huang2025operator}.
Last, the continuous linear operator $\cR\colon L^2(\T^{d};\R^{d_{\mathrm{r}}})\to L^2(\Omega;\R^{d_{\mathrm{q}}})$ returns periodic functions to the output domain $\Omega$.

In the setting of Calder\'on's problem, let $\cM=\pOmega\times\pOmega$ be the manifold, $d_\mathrm{m}=2d-2$ be the manifold dimension, and $d_\mathrm{a}=2d$ be the ambient dimension. We need the following conditions on the physical domain and the FNO activation function.
\begin{assumption}[main assumptions for EIT approximation]\label{assump:main}
    Instate the setting of \cref{sec:approx_manifold}. The following hold true:
    \begin{enumerate}[label=(A-\Roman*),leftmargin=2.5\parindent,topsep=1.67ex,itemsep=0.5ex,partopsep=1ex,parsep=1ex]
        \item \sfit{(physical domain)}\label{item:ass_domain} Let $\Omega\subset\R^d$ be a bounded domain with $C^\infty$-smooth boundary such that $\overline{\Omega}\subset (0,1)^d$ and \cref{ass:bilipschitz} holds.

        \item \sfit{(activation function)}\label{item:ass_act} Let $\varsigma \in C^\infty(\R)$ be a fixed globally Lipschitz and non-polynomial activation function. For all $\ell\in\N$, it holds that $\varsigma_\ell=\varsigma$ in \eqref{eqn:no_layers}.
    \end{enumerate}
\end{assumption}

The stability, extension, and compactness results from the previous sections lead to our main approximation theorem for EIT that we now state and prove.
\begin{theorem}[FNO approximation of EIT]\label{thm:fno_approx_main}
    Suppose that \cref{assump:main} holds. Let $C^\star$ and $\omega^\star$ be as in \cref{prop:stability_kernel}. Let $d\geq 2$ and $R>0$. For any compact set  $K \subset L^2(\pOmega\times\pOmega)$ and any $\delta>0$, the following holds. Let $K_\delta\defeq{\{\eta\in K\colon \norm{\eta}_{L^2(\pOmega\times\pOmega)}\leq \delta\}}$. For any $\ep>0$, there exists a generalized FNO $\Psi \colon L^2(\pOmega\times\pOmega)\to L^2(\Omega)$ of the form \eqref{eqn:fno_eit} such that
    \begin{align}\label{eqn:fno_approx_main}
        \sup_{(\gamma,\eta)\in \cX_d(R)\times K_\delta} \norm{\gamma - \Psi(\kappa_\gamma + \eta)}_{L^2(\Omega)} \leq  \ep + C^\star\omega^\star(\delta)\,.
    \end{align}
\end{theorem}
\begin{proof}
    Let $\widetilde{\Psi}^\star$ be the extension of $\Psi^\star$ from \cref{thm:extension_eit_kernel}. Define the set sum
    \begin{align*}
        Z_\delta\defeq \set{\kappa_0+\eta_0}{\kappa_0\in \domain(\Psi^\star)\,\text{ and }\, \eta_0\in K_\delta}\subseteq L^2(\pOmega\times\pOmega)\,.
    \end{align*}
    The set $\domain(\Psi^\star)$ from \eqref{eqn:domain_of_inverse} is compact in $L^2(\pOmega\times\pOmega)$ by \cref{cor:compact_set_kernel}. Also, $K_\delta$ is compact because it is a closed subset of the compact set $K\subset L^2(\pOmega\times\pOmega)$. It follows that the sum $Z_\delta$ is also compact in $L^2(\pOmega\times\pOmega)$ because it is the continuous image of the compact set $\domain(\Psi^\star)\times K_\delta$. Now fix $\ep>0$. We claim that there exists a generalized FNO $\Psi$ of the form \eqref{eqn:fno_eit} such that
    \begin{align}\label{eqn:fno_approx_claim}
        \sup_{\kappa\in Z_\delta}\,\norm{\widetilde{\Psi}^\star(\kappa)-\Psi(\kappa)}_{L^2(\Omega)}\leq\ep\,.
    \end{align}
    Under this claim, let $\gamma\in\cX_d(R)$ and $\eta\in K_\delta$ be arbitrary. By the definition of $\Psi^\star$ and its extension,
    \begin{align*}
        \norm{\gamma - \Psi(\kappa_\gamma + \eta)}_{L^2(\Omega)}=\norm{\Psi^\star(\kappa_\gamma) - \Psi(\kappa_\gamma + \eta)}_{L^2(\Omega)}=\norm{\widetilde{\Psi}^\star(\kappa_\gamma) - \Psi(\kappa_\gamma + \eta)}_{L^2(\Omega)}\,.
    \end{align*}
    By the triangle inequality, the uniform continuity estimate \eqref{eqn:stability_kernel_out}, and the claim~\eqref{eqn:fno_approx_claim},
    \begin{align*}
        \norm{\widetilde{\Psi}^\star(\kappa_\gamma) - \Psi(\kappa_\gamma + \eta)}_{L^2(\Omega)}
        &\leq \norm{\widetilde{\Psi}^\star(\kappa_\gamma) - \widetilde{\Psi}^\star(\kappa_\gamma+\eta)}_{L^2(\Omega)}\\
        &\qquad\qquad\qquad +\norm{\widetilde{\Psi}^\star(\kappa_\gamma+\eta) - \Psi(\kappa_\gamma + \eta)}_{L^2(\Omega)}\\
        &\leq C^\star \omega^\star\bigl(\norm{\eta}_{L^2(\pOmega\times\pOmega)}\bigr) + 
        \sup_{\kappa\in Z_\delta}\,\norm{\widetilde{\Psi}^\star(\kappa) - \Psi(\kappa)}_{L^2(\Omega)}\\
        &\leq C^\star\omega^\star(\delta) + \ep\,.
    \end{align*}
    This implies the assertion \eqref{eqn:fno_approx_main}. It remains to prove the claim \eqref{eqn:fno_approx_claim}.
    
    Recall the global-to-local map $\sfG$ from \eqref{eqn:global_to_local} and the local-to-global map $\sfL$ from \eqref{eqn:local_to_global}. Write
    \begin{align*}
        \Psi\defeq \Psi_0\circ\sfG\quad \text{for some}\quad\Psi_0\colon L^2(\mathfrak{D}\times \mathfrak{D};\R^{J^2})\to L^2(\Omega)
    \end{align*}
    to be determined.
    By \cref{lem:left_inv_pou},
    \begin{align}\label{eqn:fno_approx_claim_h}
        \begin{split}
            \sup_{\kappa\in Z_\delta}\,\norm{\widetilde{\Psi}^\star(\kappa) - \Psi(\kappa)}_{L^2(\Omega)}&=\sup_{\kappa\in Z_\delta}\,\norm{(\widetilde{\Psi}^\star\circ \sfL)(\sfG \kappa)-\Psi_0(\sfG \kappa)}_{L^2(\Omega)}\\
            &= \sup_{h\in \sfG(Z_\delta)}\norm{(\widetilde{\Psi}^\star\circ \sfL)(h)-\Psi_0(h)}_{L^2(\Omega)}\,.
        \end{split}
    \end{align}
    Since $Z_\delta$ is compact in $L^2(\pOmega\times\pOmega)$, the set $\sfG(Z_\delta)$ is compact in $L^2(\mathfrak{D}\times \mathfrak{D};\R^{J^2})$ by the continuity of $\sfG$. To find $\Psi_0$ that uniformly approximates $\widetilde{\Psi}^\star\circ \sfL$ over $\sfG(Z_\delta)$, we adapt ideas from the proof of \cite[Thm.~3.2, pp.~34--35]{huang2025operator}. First, we associate to $\widetilde{\Psi}^\star\circ \sfL$ another operator $\Upsilon^\star$ that is more compatible with existing FNO approximation theorems requiring matching domain dimensions \cite[Sec.~2]{kovachki2021universal}. Second, we relate $\Upsilon^\star$ back to $\widetilde{\Psi}^\star\circ \sfL$ by introducing a linear functional averaging layer. This allows us to bound the display \eqref{eqn:fno_approx_claim_h}. Last, we demonstrate that our approach indeed leads to an approximation of the form \eqref{eqn:fno_eit}.

    \subparagraph*{\emph{\textbf{Step 1.}}}
    Since $\overline{\Omega}\subset(0,1)^d$ and $\T^d\simeq[0,1]^d_{\mathrm{per}}$, the linear restriction operator 
    \begin{align*}
    \sfR \colon L^2(\T^d)\to L^2(\Omega)
    \end{align*}
    given by $u\mapsto u|_{\Omega}$ is well-defined and continuous with norm at most one. Let 
    \begin{align*}
        \sfE\colon L^2(\Omega)\to L^2(\T^d)
    \end{align*}
    be the continuous linear extension operator from \cite[Lemma~B.1, p.~34, applied with $d_u=d_y=1$ and $s=s'=0$]{huang2025operator}. It holds that $\sfR\sfE=\Id_{L^2(\Omega)}$.
    Define 
    \begin{align*}
        \Upsilon_0^\star\defeq \sfE\circ\widetilde{\Psi}^\star\circ \sfL\colon L^2(\mathfrak{D}\times\mathfrak{D};\R^{J^2})\to L^2(\T^d)
    \end{align*}
    so that $\widetilde{\Psi}^\star\circ \sfL= \sfR\circ \Upsilon_0^\star$. 
    Recall that $d\geq 2$.
    Define $\Upsilon^\star\colon L^2(\mathfrak{D}\times\mathfrak{D};\R^{J^2})\to L^2(\T^{2d-2})$ by
    \begin{align*}
        \bigl(\Upsilon^\star (h)\bigr)(x_1,x_2,\ldots, x_{2d-3}, x_{2d-2})\defeq \bigl(\Upsilon^\star_0 (h)\bigr)(x_1,x_2,\ldots, x_{d-1},x_d)
    \end{align*}
    for every $h$ and $x=(x_1,\ldots,x_{2d-2})$. The periodic function  $\Upsilon^\star (h)$ is constant in the last $d-2$ dimensions.
    Since $\widetilde{\Psi}^\star\colon L^2(\pOmega\times\pOmega)\to L^2(\Omega)$ is continuous by \eqref{eqn:stability_kernel_out}, and so are the linear maps $\sfE$ and $\sfL$, the operator $\Upsilon_0^\star$ is also continuous. Moreover,
    \begin{align*}
        \norm{\Upsilon^\star(h)-\Upsilon^\star(h')}^2_{L^2(\T^{2d-2})}=\norm{\Upsilon^\star_0(h)-\Upsilon^\star_0(h')}^2_{L^2(\T^{d})}
    \end{align*}
    for all $h$ and $h'$. We deduce that $\Upsilon^\star$ is continuous.
    Using \cref{assump:main} and the fact that (abusing notation) $\overline{\mathfrak{D}\times\mathfrak{D}}\subset (0,1)^{2d-2}$, the universal approximation theorem for FNOs with periodic outputs \cite[Lemma~B.1, p.~34, applied with $s=s'=0$, $d_u=J^2$, and $d_y=1$]{huang2025operator} shows that there exist a continuous linear extension operator 
    \begin{align*}
    \sfE'\colon L^2(\mathfrak{D}\times\mathfrak{D};\R^{J^2})\to L^2(\T^{2d-2};\R^{J^2})    
    \end{align*}
    and, for every $\ep>0$, an FNO
    \begin{align*}
    \Psi'\colon L^2(\T^{2d-2};\R^{J^2})\to L^2(\T^{2d-2})    
    \end{align*}
    of the form \eqref{eqn:fno_torus} such that
    \begin{align*}
        \sup_{h\in\sfG(Z_\delta)}\norm{\Upsilon^\star(h)-\Psi'(\sfE' h)}_{L^2(\T^{2d-2})}\leq \ep\,.
    \end{align*}

    \subparagraph*{\emph{\textbf{Step 2.}}}
    To relate the preceding periodic approximation result back to the claim \eqref{eqn:fno_approx_claim} and the display \eqref{eqn:fno_approx_claim_h}, define the linear operator $\sfA\colon L^2(\T^{2d-2})\to L^2(\T^d)$ by
    \begin{align*}
        u\mapsto \sfA u\defeq \int_{\T^{d-2}}u(\slot, x_{d+1},\ldots, x_{2d-2})\dd{x_{d+1}}\cdots \dd{x_{2d-2}}\,.
    \end{align*}
    This operator averages the last $d-2$ dimensions of the input function. Since $[0,1]^n$ has Lebesgue measure one for all $n$, Jensen's inequality shows that the map $\sfA$ has operator norm one (with supremum achieved at the constant function $\onebm$). 
    Observe that 
    \begin{align*}
        \Upsilon_0^\star=\sfA\circ \Upsilon^\star\,.
    \end{align*}
    By defining $\Psi_0\defeq \sfR\circ\sfA\circ\Psi'\circ\sfE'$, it holds that
    \begin{align*}
    \sup_{h\in \sfG(Z_\delta)}\norm{(\widetilde{\Psi}^\star\circ \sfL)(h)-\Psi_0(h)}_{L^2(\Omega)}&=\sup_{h\in \sfG(Z_\delta)}\norm{(\sfR\circ\sfA\circ\Upsilon^\star)(h)-\Psi_0(h)}_{L^2(\Omega)}\\
        &\leq \sup_{h\in \sfG(Z_\delta)}\norm{(\sfA\circ\Upsilon^\star)(h)-(\sfA\circ\Psi'\circ\sfE')(h)}_{L^2(\T^d)}\\
        &\leq \sup_{h\in \sfG(Z_\delta)}\norm{\Upsilon^\star(h)-(\Psi'\circ\sfE')(h)}_{L^2(\T^{2d-2})}\\
        &\leq \ep
    \end{align*}
    because both $\sfR$ and $\sfA$ have operator norm at most one. This delivers the display \eqref{eqn:fno_approx_claim}.

    \subparagraph*{\emph{\textbf{Step 3.}}}
    To complete the proof of the claim, we must show that our final approximation
    \begin{align*}
        \Psi\defeq \sfR\circ\sfA\circ\Psi'\circ\sfE'\circ\sfG
    \end{align*}
    is of the form \eqref{eqn:fno_eit} with $\cM=\pOmega\times\pOmega$, $d_\mathrm{a}=2d$, $d_\mathrm{m}=2d-2$, and $d_\mathrm{in}=d_\mathrm{out}=1$. By definition of $\Psi'$, \cite[Lemma~B.1, p.~34]{huang2025operator} asserts that there exist a number of layers $T$, channel dimension $d_\mathrm{c}$, matrix $S'\in\R^{d_\mathrm{c}\times J^2}$ determining operator $\cS'$ according to \eqref{eqn:pw_lift}, and matrix $Q'\in \R^{1\times d_\mathrm{c}}$ determining operator $\cQ'$ according to \eqref{eqn:pw_proj} such that
    \begin{align*}
        \Psi'=\cQ'\circ \cL_L\circ \cdots \circ \cL_1\circ \cS'\,,
    \end{align*}
    where the nonlinear operators $\{\cL_\ell\}_{\ell=1}^L$ map $L^2(\T^{2d-2};\R^{d_\mathrm{c}})$ into itself and are given in~\eqref{eqn:no_layers} and \eqref{eqn:fno_layer}. The preceding two displays show that we can write
    \begin{align*}
        \Psi= \cQ \circ\cR\circ \cF\circ\cL_L\circ\cdots\circ\cL_1\circ\cE\circ \cS
    \end{align*}
    in the form \eqref{eqn:fno_eit} by taking
    \begin{align*}
        S&\defeq 
        \begin{bmatrix}
            \,0\,&\,\cdots\,& \,0\,& \,1\,
        \end{bmatrix}
        \in\R^{1\times (2d+1)}\quad\text{so that}\quad \cS=\Id_{L^2(\pOmega\times\pOmega)}\,,\\
        \cE &\defeq \cS'\circ\sfE'\circ\sfG\in\sL\bigl(L^2(\pOmega\times\pOmega);L^2(\T^{2d-2};\R^{d_\mathrm{c}})\bigr)\,,\\
        \cF&\defeq\sfA\circ \cQ'\in \sL\bigl(L^2(\T^{2d-2};\R^{d_\mathrm{c}}); L^2(\T^d)\bigr)\,,\\
        \cR&\defeq\sfR\in \sL\bigl(L^2(\T^d); L^2(\Omega)\bigr)\,,\qa\\
        Q&\defeq 1\in\R^{1\times 1} \quad\text{so that}\quad \cQ=\Id_{L^2(\Omega)}\,.
    \end{align*}
    The requirement that $\cE$ be continuous holds because the global-to-local map $\sfG\in\sL(L^2(\pOmega\times\pOmega); L^2(\mathfrak{D}\times\mathfrak{D};\R^{J^2}))$, $\sfE'$, and $\cS'$ are all continuous linear operators. Moreover, $\cF$ is the composition of two continuous linear maps so it is itself continuous and linear. It may be identified,  via~\eqref{eqn:functional_layer}, with the constant kernel function $\kappa\colon \T^{d-2}\to\R^{1\times d_\mathrm{c}}$ given by $x\mapsto \kappa(x)=Q'$. Being constant, the function $\kappa$ is periodic and has a well-defined Fourier series.
    All together, these facts deliver the claimed form of the neural operator approximation.
\end{proof}

\subsection{Discussion}\label{sec:approx_discuss}
\Cref{thm:fno_approx_main} delivers the existence of a noise-robust FNO that uniformly reconstructs the unknown conductivity up to an error represented by a function of the noise level in the boundary data. The use of worst-case error in~\eqref{eqn:fno_approx_main} is natural in light of the compactness of the set $\cX_d(R)$ established in \cref{sec:approx_compactness}. However, alternative formulations in terms of average-case error with respect to a probability measure over the conductivity space are possible \cite[e.g., following Thm.~18]{kovachki2021universal} and may also be of interest~\cite{castro2024calderon}.
\Cref{thm:fno_approx_main} holds at the infinite-dimensional continuum level unlike results in related studies \cite{abhishek2024solving,pineda2023deep}, which either work with finite-dimensional conductivity sets, a finite number of electrodes, or a pre-discretization of the entire inverse problem. Moreover, the present theorem is valid for possibly infinite-dimensional sets of admissible conductivities with low regularity. 

The fundamental instability of the Calder\'on problem is not mitigated by our operator learning approach. In fact, we exploit problem-specific inverse stability estimates to control errors in the measurements. Consequently, it is unavoidable that to achieve accuracy $O(\ep)$ in \eqref{eqn:fno_approx_main}, it is sufficient for the magnitude $\delta$ of the data perturbations to be exponentially small, that is, $\delta=O(\exp(-\ep^{-1/\rho}))$. Of course, this assumes a worst-case analysis. In practical average-case operator learning computations based on empirical risk minimization, we might not observe this poor complexity.
Moreover, by restricting to finite-dimensional sets of admissible conductivities with appropriate regularity \cite{abhishek2024solving,pineda2023deep}, it should be possible using Lipschitz stability estimates \cite{alberti2022infinite} to achieve a better polynomial complexity of the form $\delta=O(\ep^{\varrho})$ for some $\varrho>0$. The tradeoff is that the constant corresponding to $C^\star$ in the error bound \eqref{eqn:fno_approx_main} will explode as the dimension increases~\cite{alberti2022infinite}.

We close \cref{sec:approx} by discussing more related work, alternative architectures, and noise distributions.

\subparagraph*{\emph{\textbf{Comparison to related work.}}}
Although \cref{thm:fno_approx_main} is formulated for kernel functions associated to NtD maps $\ntd$, it could be adapted to DtN maps $\dtn$ as well. This aligns with related work \cite{castro2024calderon,abhishek2024solving}. The extension program from \cref{sec:extend} would go through by similar arguments, following the ideas in~\cite{abraham2019statistical}. Regarding FNO approximation, one must smooth $\dtn$ as in \cref{sec:extend_kernel} \emph{in all dimensions} $d\geq 2$ or work with the smoothing difference operator $\ntd-\sR_{\onebm}$ in order to identify a square integrable input kernel function. Beyond FNOs, alternative neural operator architectures---such as those summarized in \cite[Sec.~4.2]{nelsen2025operator}---could directly process linear operator data $\ntd$ \cite{de2019deep,castro2024calderon} or distribution-based relaxations thereof~\cite{guerra2025learning,molinaro2023neural}. These methods are also suitable for use in our approximation framework.

The approximation result \eqref{eqn:fno_approx_main} is related to error bounds found in \cite{abhishek2024solving,castro2024calderon,pineda2023deep} for other operator learning architectures. As in the current paper, these works reformulate the inverse map to act between Hilbert spaces and extend the domain of the inverse in order to apply universal approximation theorems for neural operators. One approach along these lines is to approximate Calder\'on's inverse map on average with respect to a probability measure compacted supported on $\domain(\Psi^\star)$ using deep operator networks (DeepONets)~\cite[Thm.~4.4, p.~786]{castro2024calderon}. However, by only using a zero extension that ignores the continuity properties of $\Psi^\star$, this approach requires detailed information about $\domain(\Psi^\star)$ and cannot handle noisy boundary data.

A similar DeepONet approach is adopted by \cite{abhishek2024solving} for the complete electrode model of EIT. Unlike that in \cite{castro2024calderon} and $\cX_d(R)$ from the present paper, the admissible set of conductivities in \cite{abhishek2024solving} is taken to be a \emph{finite-dimensional} subset of $\Gamma'$. It is shown in \cite[Thm.~3.1, p.~7]{abhishek2024solving} that one can find a sufficiently large number of electrodes and a DeepONet with sufficiently many parameters to ensure that the worst case $L^2(\Omega)$ reconstruction error is arbitrarily small. Although \cite[Thm.~3.1, p.~7]{abhishek2024solving} does not track the effect of perturbations in the NtD map data, the approach taken in the present paper to prove \cref{thm:fno_approx_main} could be used to derive a noise-robust estimate akin to \eqref{eqn:fno_approx_main} in the setting of \cite{abhishek2024solving}. Here in particular, the finite-dimensionality of the conductivity set implies that the inverse map and its Benyamin--Lindenstrauss extension are Lipschitz continuous \cite{alberti2019calderon,alberti2022infinite,harrach2019uniqueness}, in contrast to the logarithmic stability shown in the current paper.

Work in \cite{pineda2023deep} takes this finite-dimensional Lipschitz analysis further by going beyond EIT to a broad class of nonlinear inverse problems. Building on~\cite{alberti2022infinite} by working with sufficiently many finite-dimensional measurements in addition to finite-dimensional parameter sets, the paper \cite{pineda2023deep} studies an effective finite-dimensional inverse problem that admits a Lipschitz continuous inverse map. The resulting analysis establishes that ReLU neural networks can approximate the finite-dimensional inverse map with average-case robustness to Gaussian noise perturbing the data \cite[Thm.~3.9]{pineda2023deep}. Moreover, empirical risk minimization realizes such Lipschitz neural networks \cite[Thm.~4.3]{pineda2023deep}. Although the approach avoids case-by-case analysis, this generality comes at the cost of obscuring ties to the underlying infinite-dimensional inverse problem of interest. The finite-dimensionalization procedure also links the architecture to the number of measurements and requires re-training the network whenever the dimension changes. This contrasts with the present work on EIT, which adopts an operator learning learning perspective, works in infinite dimensions throughout, and obtains worst-case robustness to noise.

We refer the reader to \cite[Sec.~4.3.1--4.3.2]{nelsen2025operator} for a more detailed comparison between the related works \cite{castro2024calderon,abhishek2024solving,pineda2023deep} and the present paper.

\subparagraph*{\emph{\textbf{Neural operators and boundary functions.}}}
The construction of the global-to-local and local-to-global linear operators $\sfG$ and $\sfL$, respectively, in \cref{sec:approx_manifold} is crucial to enable the application of standard universal approximation theorems for (Fourier) neural operators. Although the proof of \cref{thm:fno_approx_main} identifies a specific choice based on $\sfG$, the continuous linear map $\cE$ in the generalized FNO definition \eqref{eqn:fno_eit} is allowed to be general. Indeed, $\sfG$ may be difficult to access or work with in practical computations. More constrained parametrizations of $\cE$ that can be easily identified or learned from data are of interest. Examples include constant extensions to the whole domain $\Omega$ \cite{de2022cost}, handcrafted global parametrizations of $\pOmega$, e.g., polar coordinates if $\Omega$ is the unit disk and $d=2$ (see \cref{sec:numerics}), or learnable kernel integral operators \cite[Sec.~3.5, pp.~18--19]{wu2025learning}.
In the latter case, \cref{thm:fno_approx_main} would remain valid as long as we can use such boundary integral operators to uniformly approximate $\sfG$ over the set $Z_\delta$ defined in the proof of the theorem. Beyond EIT, these ideas could also provide a theoretical path for showing that various practical geometry-aware neural operators~\cite{li2023fourier,li2023geometry,li2025geometric,chen2024learning,wu2024neural} are universal approximators of operators defined on $L^2(\cM)$ for a manifold $\cM$, or even on smaller sets such as $H^s(\cM)$ or $C^\al(\cM)$; see \cite[p.~7]{lanthaler2023nonlocality} for similar remarks.

\subparagraph*{\emph{\textbf{Statistical noise models.}}}
\Cref{thm:fno_approx_main} only requires that the perturbation $\eta$ belong to a compact set. The result does not enforce a statistical description of the noise. However, modeling $\eta$ as a random variable is quite natural due to the imprecise nature of real EIT hardware systems \cite[Sec.~4]{garde2017convergence}. One way to do this while still satisfying the hypotheses of \cref{thm:fno_approx_main} is by using random series expansions. Consider the following example. Recalling the Laplace--Beltrami eigenfunctions $\{\varphi_j\}$ from \cref{sec:prelim_func}, define $\eta\in L^2(\pOmega\times\pOmega)$ in law by
\begin{align}\label{eqn:kle_2d_boundary}
    (x,x')\mapsto \eta(x,x')\defeq \sum_{i=1}^\infty\sum_{j=1}^\infty \sqrt{c_{ij}}\zeta_{ij}\varphi_i(x)\varphi_j(x')
\end{align}
in the mean square sense. The $\{\zeta_{ij}\}$ are independent and identically distributed (i.i.d.) bounded random variables with mean zero and variance one. We want to ensure that $\mu\defeq\Law(\eta)\in\sP(L^2(\pOmega\times\pOmega))$ has compact support. Fix positive $\delta$, $s$, and $t$. If $t>s+d-1$ and the non-increasing rearrangement of the doubly-indexed coefficients $\{c_{ij}\}$ into a singly-indexed sequence $\{c'_k\}$ satisfies $c'_k=O(\delta^2k^{-t/(d-1)})$ as $k\to\infty$, then we can show that there exists $C>0$ such that $\norm{\eta}_{H^s(\pOmega\times\pOmega)}\leq C$ and $\norm{\eta}_{L^2(\pOmega\times\pOmega)}\leq \delta$ almost surely. Since the closed $H^s$ ball of radius $C$ is compact in $L^2$, we deduce that \cref{thm:fno_approx_main} is valid with $K_\delta$ replaced by $\supp(\mu)$. High probability versions of \cref{thm:fno_approx_main} are also conceivable; these would be able to accommodate probability measures $\mu$ with unbounded support, such as (sub)Gaussian noise distributions \cite{pineda2023deep}.

\section{Numerical experiments}\label{sec:numerics}
The primary purpose of this section is to demonstrate the practical performance of the FNO in reconstructing challenging conductivities from noisy NtD kernels. Although \cref{thm:fno_approx_main} asserts the \emph{existence} of an FNO that stably approximates the Calder\'on inversion operator \eqref{eqn:inverse_map_final_def}, the theorem is not constructive. We actually \emph{find} high quality FNO approximations using standard empirical risk minimization over a noisy training dataset of size $N$. We pay particular attention to the test error of the trained FNOs as the noise level varies.

This section specializes to the $d=2$ setting of EIT (\cref{sec:prelim_eit}) in which $\Omega\defeq\mathbb{D}\subset\R^2$ is the unit disk. We identify the boundary $\partial\D\simeq\T\simeq [0,2\pi]_{\mathrm{per}}$. We take $\Omega'$ in \eqref{eqn:compact_support_set} to be an open ball of (possibly variable) radius less than one.
We consider three different random field models for the conductivities: binary, three phase, and lognormal-type. \Cref{app:numerics} contains details on their contruction. To summarize, the binary fields are obtained by thresholding a single Gaussian random field to two values, leading to a shape detection task; see \cref{fig:tile_shape}. The three phase inclusions are the result of thresholding two Gaussian random fields to three values; see \cref{fig:tile_three_phase}. The final model applies a smooth cutoff to a lognormal random field; see \cref{fig:tile_lognormal}.
We emphasize that the support of each of these three data-generating probability measures is genuinely infinite-dimensional and is not constructed from any finite-dimensional manifold. 

By construction, sampled conductivities belong to $\Gamma'$ in \eqref{eqn:conductivity_set_prelim} almost surely under all three distributions. Moreover, the binary and three phase fields belong to $\BV(\D)$ almost surely. However, for large enough $R$, we can only expect samples to belong to $\cX_2(R)=\cX_{\BV}(R)$ in \eqref{eqn:conductivity_set} with high probability. Unless we condition on this high probability event, \cref{thm:fno_approx_main} does not directly apply to the support of our data-generating distributions. Precise probabilistic statements require estimating moments of the perimeter of excursion sets \cite{adler2007random} and are beyond the scope of this paper.

For sample size $N$ and relative noise level $\delta$, we generate paired i.i.d. samples $\{(\kappa_{\gamma_n}^{(\delta)},\gamma_n)\}_{n=1}^N$ under the forward map \eqref{eqn:forward_map_kernel_bigger} and the noise model
\begin{align}\label{eqn:ntd_noise_model}
    \kappa_\gamma^{(\delta)}\defeq \kappa_\gamma + \delta\norm{\kappa_\gamma}_{L^2(\T^2;\R)}\xi\,.
\end{align}
In \eqref{eqn:ntd_noise_model}, $\kappa_\gamma^{(\delta)}$ is a perturbed version of the NtD kernel $\kappa_\gamma$ under $(100\times \delta)\%$ relative noise. The factor $\xi$ is either an independent Gaussian random field or a bounded random field; see \cref{app:numerics}. 
In the experiments, $\delta$ is allowed to vary between training and test input data, as shown in \cref{tab:rates_all}. We take $N=9500$ unless otherwise stated.

\begin{table}[tb]%
    \captionsetup{width=\textwidth,skip=10pt}
    \caption{Sample complexity for the three sets of FNO experiments. For each dataset, the table entries record the experimental convergence rate exponents $r$ in 
    $O(N^{-r})$ of the average relative $L^1(\D)$ test error 
    for either noisy or clean test data (rows) as the training data noise level (last four columns) increases. Higher values indicate faster convergence.}
    \label{tab:rates_all}
    \centering
    \renewcommand{\arraystretch}{1.2}
    \begin{tabular}{lccccc}
        \toprule
        & & \multicolumn{4}{c}{Training data noise level} \\
        \cmidrule(l){3-6}
        Dataset & Test data noise level & $0\%$ & $3\%$ & $10\%$ & $30\%$ \\
        \midrule
        \multirow{2}{*}{Shape detection} 
            & Same as train &  $0.231$ & $0.261$ & $0.293$ & $0.252$\\
            & $0\%$ &  $0.231$ & $0.254$ & $0.291$ & $0.321$\\
        \midrule
        \multirow{2}{*}{Three phase inclusions} 
            & Same as train &  $0.100$   & $0.095$ & $0.081$ & $0.058$ \\
            & $0\%$ & $0.100$  & $0.098$& $0.090$ & $0.070$ \\
        \midrule
        \multirow{2}{*}{Lognormal conductivities} 
            & Same as train &  $0.332$ & $0.432$ & $0.406$ & $0.360$\\
            & $0\%$ & $0.332$ & $0.445$ & $0.485$ & $0.431$ \\
        \bottomrule
    \end{tabular}
\end{table}

The FNO discretization and implementation is described in \cref{app:numerics}. To summarize, the architecture $\Psi_\theta^{(\mathrm{FNO})}$ is of the form \eqref{eqn:fno_torus} and \eqref{eqn:fno_layer} for a finite number trainable parameters $\theta\in\Theta$. We fix $L=2$ layers, $12$ Fourier modes per dimension, and channel width $d_\mathrm{c}=48$. To train FNOs, we approximately solve
\begin{align}\label{eqn:erm}
	\min_{\theta\in\Theta}\frac{1}{N}\sum_{n=1}^N\ell\Bigl(\gamma_n , \Psi_\theta^{(\mathrm{FNO})}\bigl(\kappa_{\gamma_n}^{(\delta)}\bigr)\Bigr)\,,\qw 	\ell(\gamma, \gamma')\defeq \frac{\norm{\gamma-\gamma'}_{L^1(\D)}}{\norm{\gamma}_{L^1(\D)}+\epsilon}
\end{align}
is a relative loss function and $\epsilon>0$ is a numerical stability hyperparameter. We emphasize the use of $L^1(\D)$ in the output space; the use of $L^2(\D)$ empirically delivered worse reconstructions that were substantially oversmoothed.
See \cite[Secs.~3.1 and 4.1.1]{nelsen2025operator} for more details about training neural operators. Reconstruction accuracy, i.e., test error, is mainly evaluated pointwise in relative $L^1(\D)$ loss for individual samples or in average relative $L^1(\D)$ loss over a test dataset. The test dataset is unseen during training but has the same distribution as the training data.

The main observations of this section are summarized as follows.
\begin{enumerate}[label=(O\arabic*),leftmargin=2.5\parindent,topsep=1.67ex,itemsep=0.5ex,partopsep=1ex,parsep=1ex]
    \item Trained FNO reconstructions are robust to large noise levels. Conversely, as the noise level tends to zero, the test error closely follows a logarithmic law as anticipated by \eqref{eqn:fno_approx_main} in \cref{thm:fno_approx_main}.

    \item The convergence rate of trained FNOs is algebraic in $N$ for all three datasets, as reported in \cref{tab:rates_all}. The rate varies depending on the dataset.

    \item FNOs can be trained with Gaussian noise---which does not satisfy \cref{thm:fno_approx_main}---and deployed with bounded noise, which does satisfy it.

    \item Training FNOs on highly noisy boundary data has a regularizing effect. These models can still be accurate when the test data has lower noise levels.

    \item Although the paper does not emphasize reconstruction speed, nevertheless the FNO can map $500$ NtD kernels to reconstructed conductivities on a $256\times 256$ resolution grid in under one second on a graphics processing unit (GPU).
\end{enumerate}

As a point of reference, \Cref{fig:compare_dbar} compares FNOs trained on $N$ pairs of noise-free data to the regularized D-bar direct inversion method~\cite{mueller2012linear,knudsen2009regularized,mueller2020d}. Here $\delta=10^{-2}$ for boundary data corresponding to the three piecewise constant conductivities depicted. This noise level is high for the D-bar method. As a result, D-bar solutions are oversmoothed and even fail to detect some of the inclusions. In contrast, when $N$ is sufficiently large, the trained FNOs reconstruct inclusion locations correctly, but still can get the conductivity values on the inclusions wrong, as in the bottom row. 

Although such comparisons are insightful, the focus of this section is to illuminate the robustness, accuracy, and complexity characteristics of a fixed method, namely, the FNO. To this end, the remainder of the section presents and describes numerical results in which the noise level $\delta$, sample size $N$, and type of conductivity are varied. \Cref{sec:numerics_shape} presents results for shape detection, \cref{sec:numerics_three} for three phase inclusion recovery, and \cref{sec:numerics_lognormal} for lognormal conductivity imaging.

\subsection{Shape detection}\label{sec:numerics_shape}
The binary conductivities in the shape detection dataset have a contrast ratio of $100$ (inclusion) to $1$ (background), which places the inverse problem in the highly nonlinear regime. Different interior structures within highly conducting or highly insulating regions can produce almost the same boundary data, making it difficult to reconstruct such hidden structure. \Cref{fig:tile_shape} shows an example of this phenomenon. The reconstructions in row three are obtained by inputting the top row of standardized NtD kernels with $2\%$ bounded noise into a single FNO trained on data samples corrupted by $3\%$ Gaussian noise. The true conductivity in row two, column four (a random sample from the test set) has a low conductivity ``hole'' within a high conductivity surrounding region. The FNO reconstruction in row three, column four misses the hole and fills it in with the high conductivity value, as expected.

\begin{figure}[tb]
	\centering
    \includegraphics[width=0.99\textwidth]{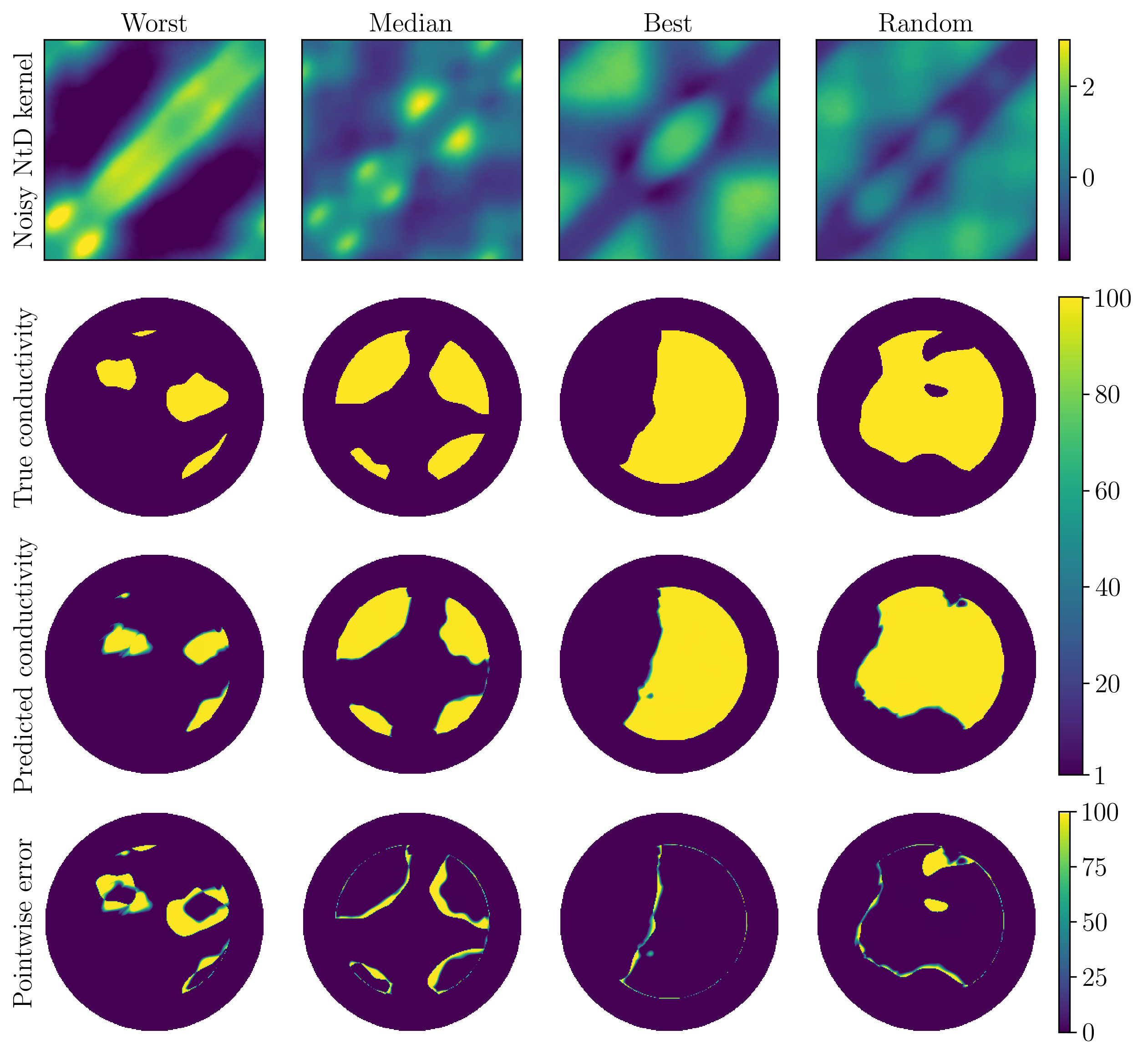}
    \caption{Sample reconstructions from the shape detection dataset.}
	\label{fig:tile_shape}
\end{figure}

Overall, the FNO succeeds greatly at shape reconstruction. The last row of \cref{fig:tile_shape} shows that most of the error is made at the boundary of the inclusions. The number of inclusions is usually predicted correctly, ranging from one to four or more. We now comment on accuracy quantitatively. The first three columns of \cref{fig:tile_shape} represent the test set samples with the worst, median, and best \emph{Dice scores} (larger is better), which is a similarity coefficient commonly used to assess image segmentation quality~\cite{zou2004statistical}. The test set average Dice score is $0.865$ and relative $L^1(\D)$ error is $0.264$. We also compute an average ``$L^0(\D)$ norm'' error of $0.197$, which is the area of the region where the true and predicted conductivities differ. These are excellent scores considering how challenging this discontinuous dataset is.

\begin{figure}[tb]
	\centering
	\begin{subfigure}[]{0.42\textwidth}
		\centering
		\includegraphics[width=\textwidth]{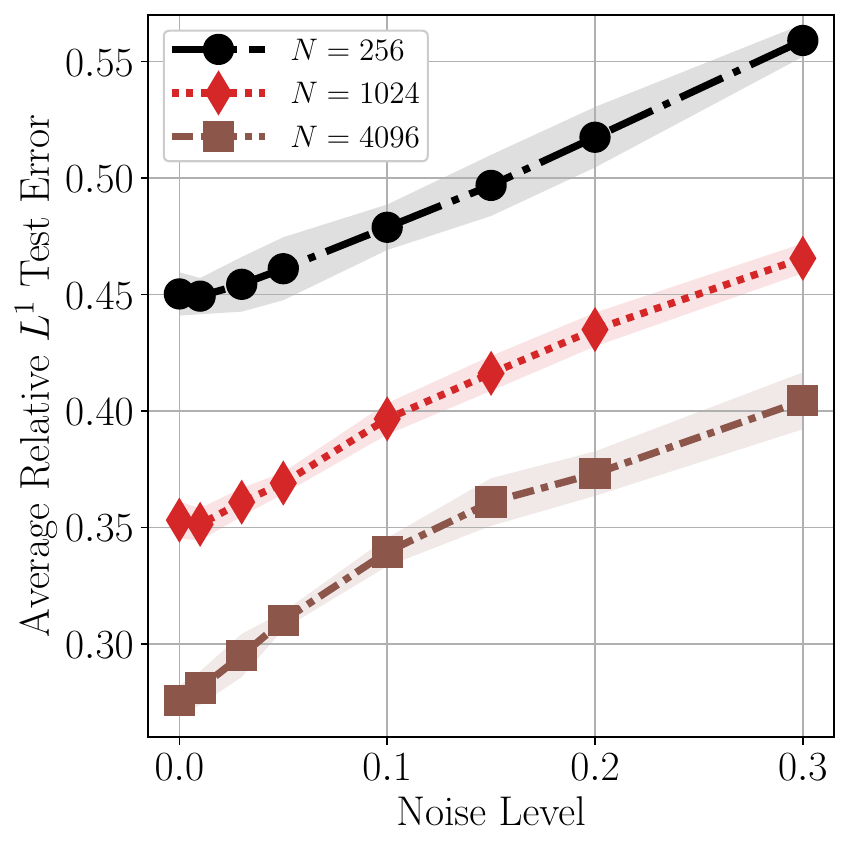}
		\caption{Robustness to large noise levels}
		\label{subfig:noise_sweep_two_phase_noisy}
	\end{subfigure}
	\begin{subfigure}[]{0.42\textwidth}
		\centering
		\includegraphics[width=\textwidth]{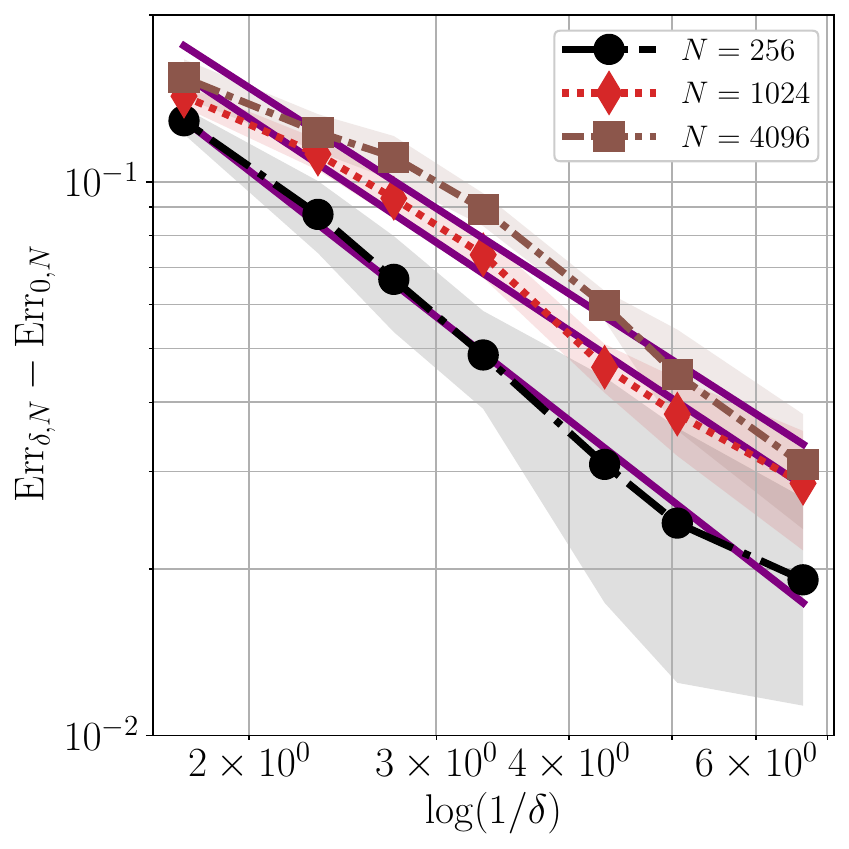}
		\caption{Fit to logarithmic modulus of continuity}
		\label{subfig:noise_log_two_phase_noisy}
	\end{subfigure}
    \begin{subfigure}[]{0.43\textwidth}
		\centering
		\includegraphics[width=\textwidth]{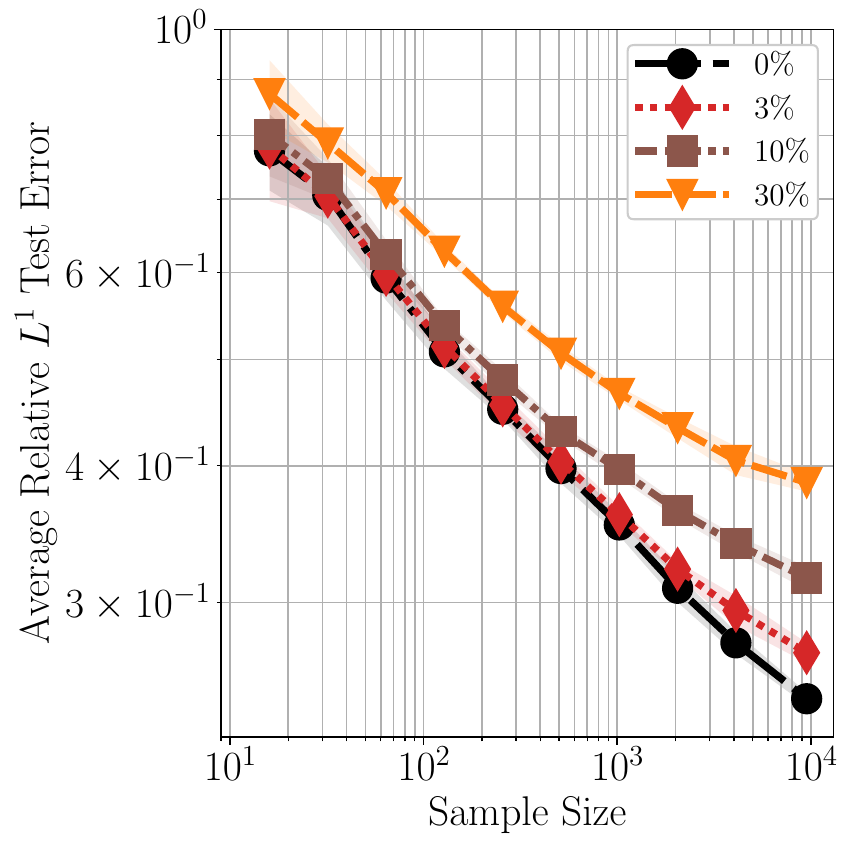}
		\caption{Algebraic sample complexity}
		\label{subfig:data_sweep_two_phase_noisy}
	\end{subfigure}
	\begin{subfigure}[]{0.42\textwidth}
		\centering
		\includegraphics[width=\textwidth]{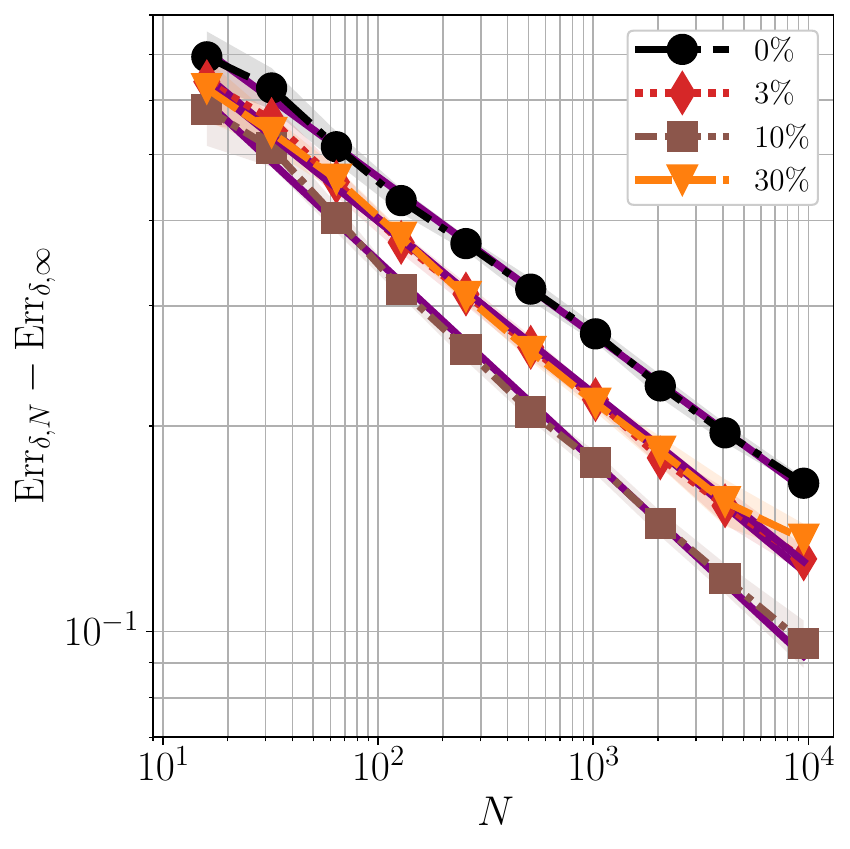}
		\caption{Fit of excess test error to power law}
		\label{subfig:data_power_two_phase_noisy}
	\end{subfigure}
    \caption{Shape detection noise robustness (top row) and sample complexity (bottom row) experiments. For training dataset size $N$, \cref{subfig:noise_sweep_two_phase_noisy} shows how the average relative $L^1(\D)$ test error $\mathrm{Err}_{\delta, N}$ increases sublinearly as the training and test data noise level $\delta$ increases. \Cref{subfig:noise_log_two_phase_noisy} plots the excess error versus $\log(1/\delta)$ to validate the logarithmic bound in \cref{thm:fno_approx_main}. \Cref{subfig:data_sweep_two_phase_noisy} shows how $\mathrm{Err}_{\delta, N}$ decreases algebraically with $N$. \Cref{subfig:data_power_two_phase_noisy} plots the excess error versus $N$ to obtain precise estimates of the convergence rate (\cref{tab:rates_all}). Purple lines are linear least squares fits on a logarithmic scale. Shaded bands denote two standard deviations from the mean (excess) error over five independent training runs.
    }
	\label{fig:two_phase_noisy}
\end{figure}

Next, the top row of \cref{fig:two_phase_noisy} suggests that FNOs trained on sample size $N$ are robust to increasing amounts of noise in the sense that the average test error only increases sublinearly with $\delta$ (\cref{subfig:noise_sweep_two_phase_noisy}). In \cref{subfig:noise_log_two_phase_noisy}, we also fit the difference between the test error and the noise-free test error to the logarithmic relationship implied by the right-hand side of \eqref{eqn:fno_approx_main} from \cref{thm:fno_approx_main}. Since the fit is approximately linear on a log-log scale, it suggests that the severe ill-posedness of EIT does indeed manifest numerically for FNOs. Last, we observe that increasing the training data noise level serves to greater regularize reconstructions by returning less detailed, larger lengthscale ``filled in'' inclusions (still with sharp boundaries), regardless of the test data noise level. See, e.g., \cref{fig:tile_shape_30noise20} in \cref{app:numerics}. These observations align with classical statistical intuition that training on noisy inputs is a form of regularization~\cite{bishop1995training}.

From the point of view of offline computational cost, understanding the sample complexity of learned inverse solvers is paramount. By taking the test error as a function of sample size $N$ (\cref{subfig:data_sweep_two_phase_noisy}) and subtracting a regressed $N$ independent bias, we obtain linear decay on a log-log scale in \cref{subfig:data_power_two_phase_noisy}. The slopes of these lines are the statistical rate of convergence exponents that we record in row one of \cref{tab:rates_all}. The rate is approximately $O(N^{-1/4})$---which is slower than the fast parametric $O(N^{-1/2})$ rate of estimation---and mostly independent of $\delta$ as expected. Similar rates also hold when only the training data is noisy and the test data is clean. The observed polynomial complexity further supports the observed success of FNOs for EIT shape detection.

We conclude the shape detection results by commenting on out-of-distribution performance. Since the dataset consists of conductivities with random shapes but fixed values on the shapes, reconstructing conductivities with new shapes is possible but different values or contrast ratios is challenging. In row two of \cref{fig:compare_dbar} and \cref{fig:heart_shape_noisy}, we provide an FNO trained on shape detection with a noisy NtD kernel corresponding to a deterministic heart and lungs phantom. It has the same contrast ratio as the dataset, but consists of perfect ellipses instead of inclusions with rough random boundaries. The FNO is still able to detect the inclusion locations with reasonable accuracy.

\subsection{Three phase inclusions}\label{sec:numerics_three}
The most difficult dataset considered in this paper corresponds to the three phase inclusions model, as implied by the small statistical convergence rate exponents of $0.1$ or less in the middle row of \cref{tab:rates_all}. Unlike in shape detection, this dataset features two random conductivity values on piecewise constant inclusions in addition to the background value of one. This leads to complicated triple junction patterns as shown in row two of \cref{fig:tile_three_phase}. As a result, the conductivity contrast ratio is a random variable with an upper bound of $100$. FNO training overfits almost immediately on this dataset, so the best model is usually found after just a few epochs. We report results using this model. Training on higher noise levels provides additional regularization that can improve reconstruction accuracy (\cref{subfig:tile_three_phase_noisy}). The noise robustness and sample complexity results in the top row of \cref{fig:three_phase} are analogous to those for shape detection.

\begin{figure}[tb]
	\centering
    \includegraphics[width=0.99\textwidth]{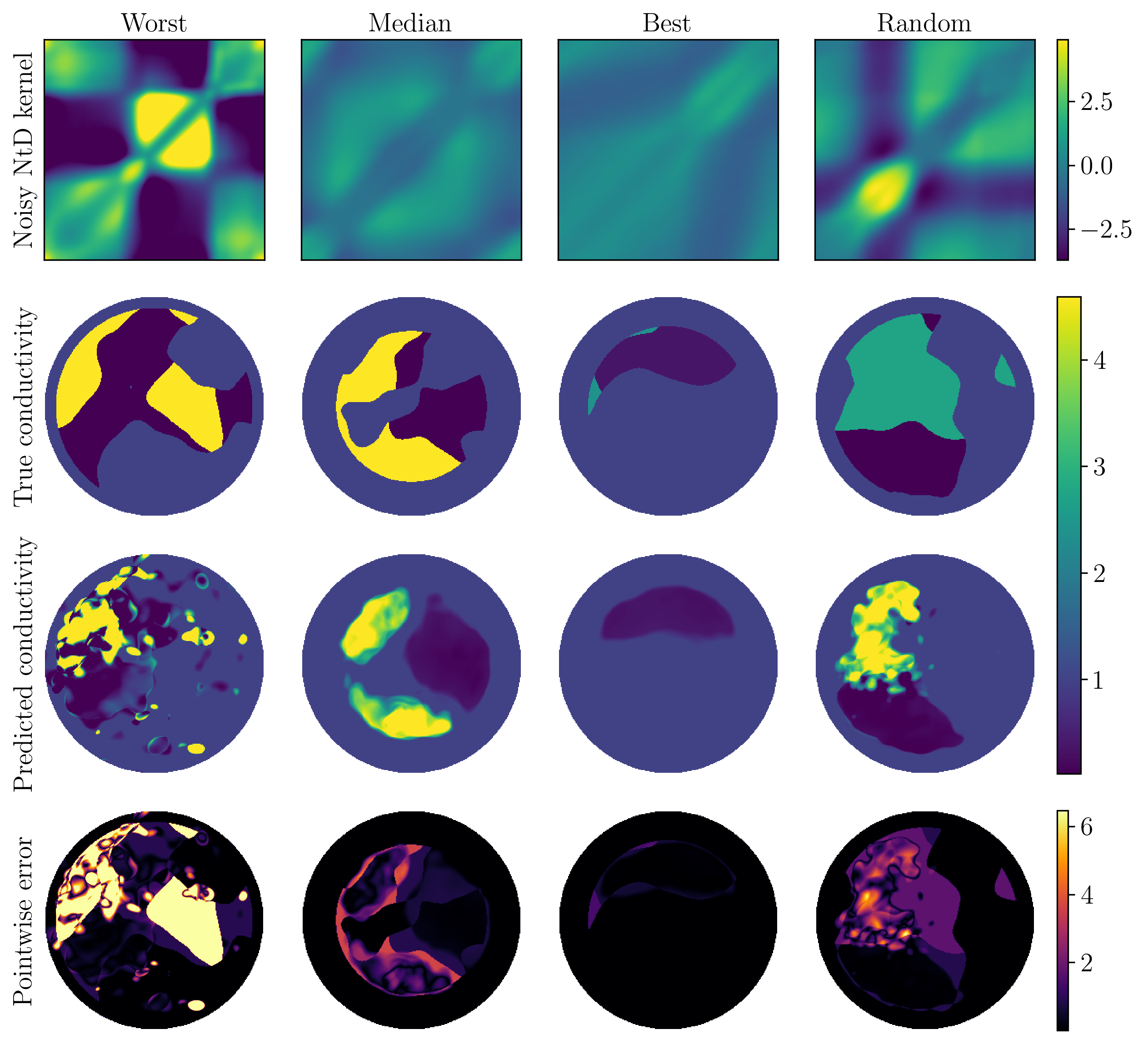}
    \caption{Sample reconstructions from the three phase inclusion dataset.
}
	\label{fig:tile_three_phase}
\end{figure}

\Cref{fig:tile_three_phase} is analogous to \cref{fig:tile_shape}, except now with $1\%$ bounded noise in the NtD kernels along the top row of the figure. The first three columns represent the samples achieving the maximum ($0.854$), median ($0.243$), and minimum ($0.030$) relative $L^1(\D)$ test errors over the $400$ sample test dataset. The average test error is $0.252$. The worst NtD kernel spans a larger range of function values than the others. The FNO is challenged by this outlier kernel and produces a very irregular output prediction. Although the FNO also has a hard time detecting the sharp transition between random phases, it is able to roughly identity the inclusion locations and values. However, the reconstruction boundaries are more blurred than those for shape detection.

\begin{figure}[tb]
	\centering
	\begin{subfigure}[]{0.33\textwidth}
		\centering
		\includegraphics[width=\textwidth]{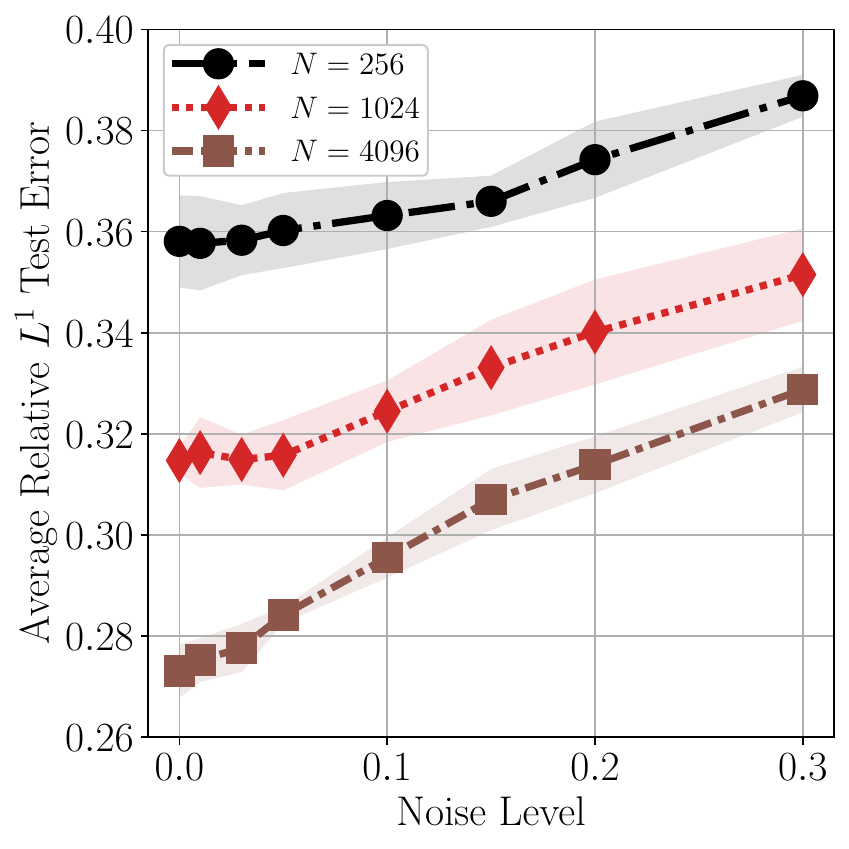}
        \caption{Robustness to large noise levels}
    	\label{subfig:noise_sweep_three_phase_noisy}
	\end{subfigure}
	\begin{subfigure}[]{0.33\textwidth}
		\centering
		\includegraphics[width=\textwidth]{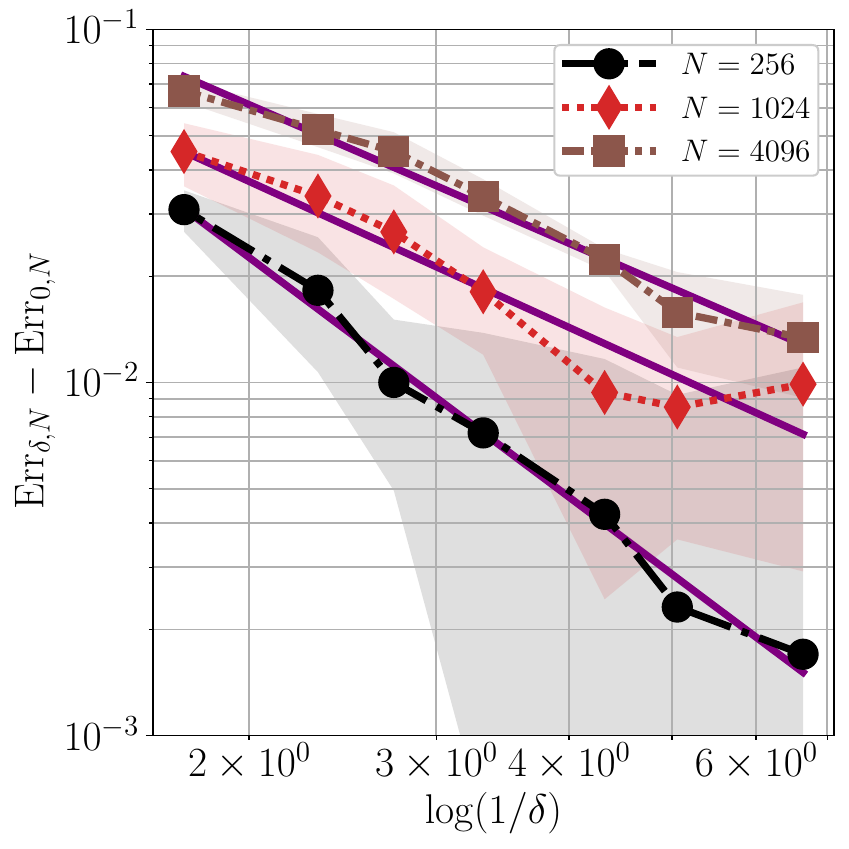}
        \caption{Fit to log modulus of continuity}	
		\label{subfig:noise_log_three_phase_noisy}
	\end{subfigure}
	\begin{subfigure}[]{0.32\textwidth}
		\centering
		\includegraphics[width=\textwidth]{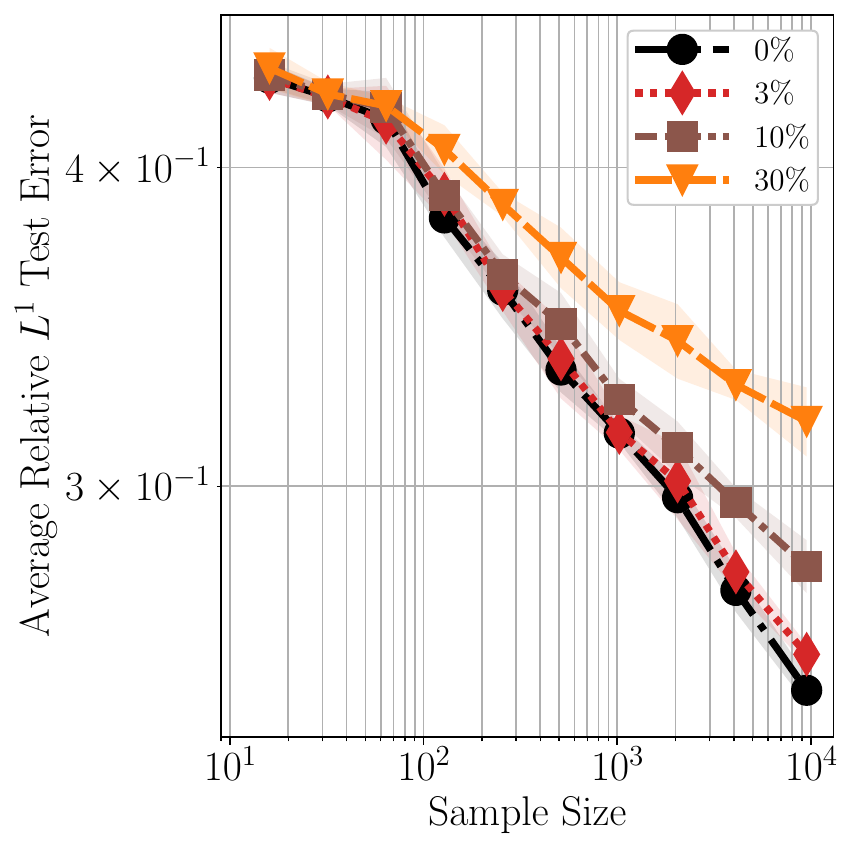}
		\caption{Sample complexity}
		\label{subfig:data_sweep_three_phase_noisy}
	\end{subfigure}
    \begin{subfigure}[]{0.99\textwidth}
    	\centering
        \includegraphics[width=\textwidth]{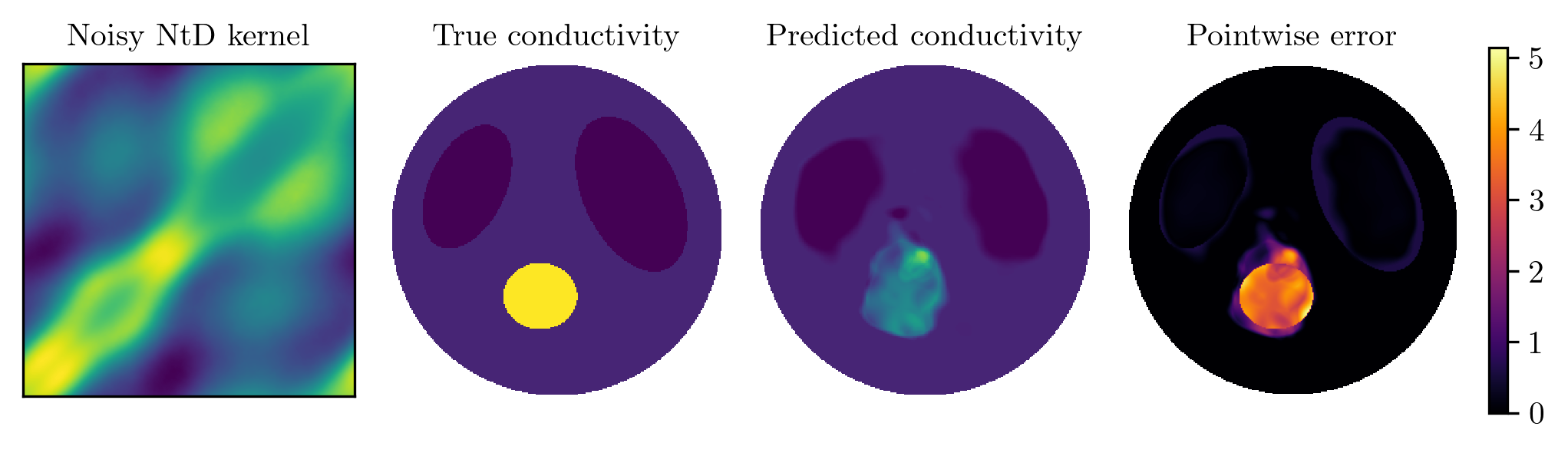}
        \caption{FNO reconstruction of a benchmark phantom from an NtD kernel with $1\%$ noise.
    }
    	\label{subfig:heart_three_phase}
    \end{subfigure}
    \caption{Three phase inclusion noise robustness and sample complexity (top row) and distribution shift (bottom row) experiments. \Cref{subfig:noise_sweep_three_phase_noisy} and \cref{subfig:noise_log_three_phase_noisy} are analogous to the top row of \cref{fig:two_phase_noisy}, and \cref{subfig:data_sweep_three_phase_noisy} is analogous to \cref{subfig:data_sweep_two_phase_noisy}. \Cref{subfig:heart_three_phase} plots the FNO reconstruction of an out-of-distribution heart and lungs phantom; the training data is noise-free.}
    \label{fig:three_phase}
\end{figure}

Although performance on the three phase inclusion dataset is worse than that on the other two datasets, trained FNOs perform well out-of-distribution. Indeed, the diversity of the three phase dataset now enables the FNO to handle input samples corresponding to new shapes \emph{and} varying contrast. In \cref{subfig:heart_three_phase} and the last row of \cref{fig:compare_dbar}, we illustrate this with a deterministic heart and lungs phantom with realistic conductivities values of $6.3$ for the heart and $0.4$ for the lungs~\cite[Table~12.1, p.~160]{mueller2012linear}. The FNO reconstructs the two lungs accurately and the location of the heart. However, the conductivity value of the heart is underpredicted by a factor of around two.

\subsection{Lognormal conductivities}\label{sec:numerics_lognormal}
The lognormal-type conductivity dataset is the easiest of the three. Conductivity samples are a small random perturbation away from the constant unit background (\cref{fig:tile_lognormal}, row two). The (random) contrast ratio is bounded and small. As a result, we observe higher sensitivity to noise than in the discontinuous conductivity datasets.

\begin{figure}[tb]
	\centering
    \includegraphics[width=0.99\textwidth]{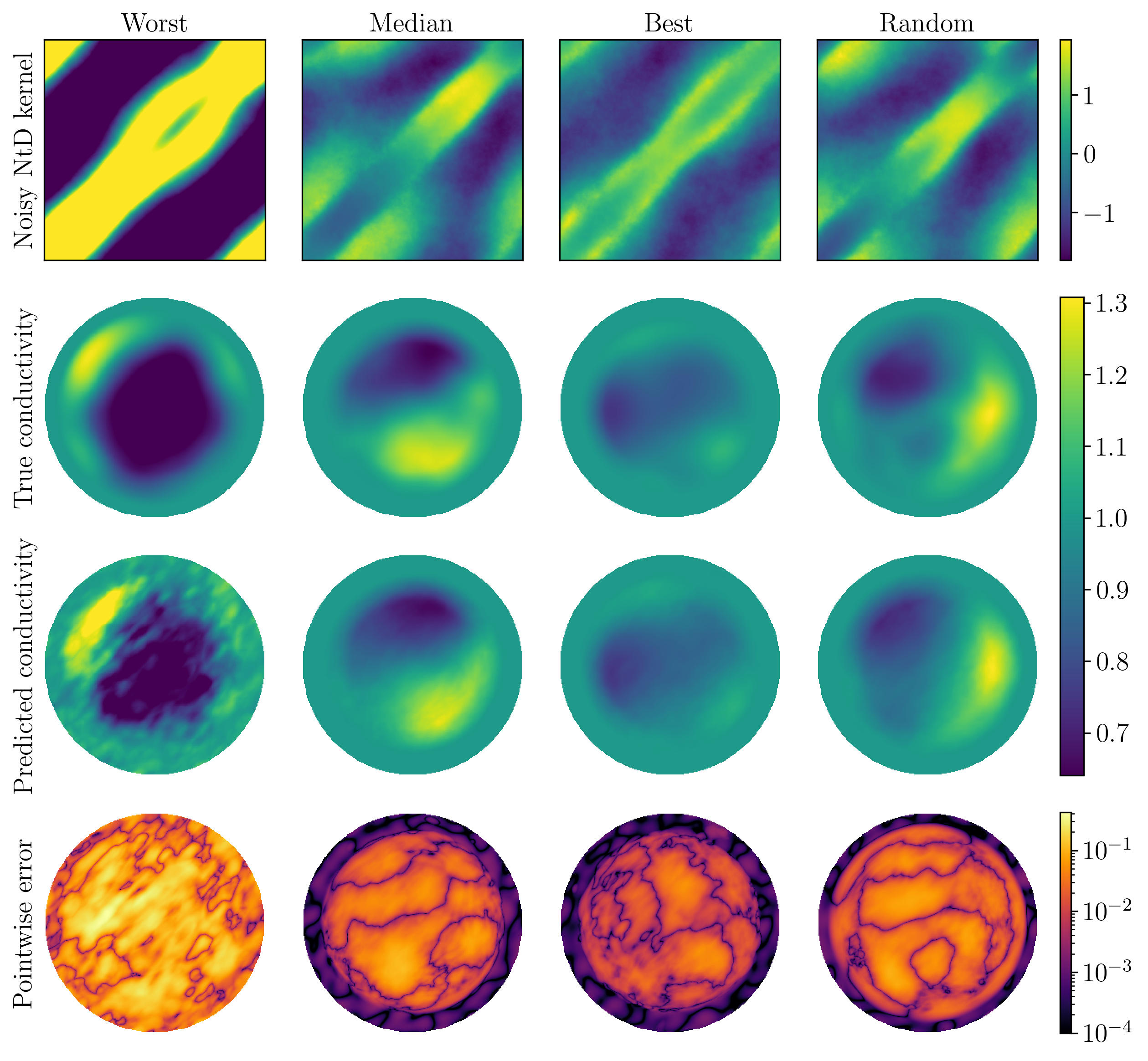}
    \caption{Sample reconstructions from the lognormal conductivity dataset.
}
	\label{fig:tile_lognormal}
\end{figure}

\Cref{fig:tile_lognormal} shows reconstructions corresponding to the maximum ($0.067$), median ($0.013$), minimum ($0.007$), and random ($0.014$) relative $L^1(\D)$ errors over a $400$ sample test set. The average test error is $0.014$. Here the training noise is $3\%$ and the testing noise is $1\%$. The noise is clearly visible in the top row of NtD kernels. As in \cref{sec:numerics_shape,sec:numerics_three}, the worst sample corresponds to an NtD kernel with large absolute function values. Trained FNOs are sensitive to this scale. Row three shows that FNO reconstructions are affected by high frequency artifacts that are likely due to the noise. Training on noisier data can dampen this behavior to some extent (\cref{fig:tile_lognormal_10noise6}). The pointwise errors in row four of \cref{fig:tile_lognormal} are lowest near the boundary of $\D$. There are also one-dimensional curves embedded in high error regions where the error along the curves is much smaller.

\begin{figure}[tb]
	\centering
	\begin{subfigure}[]{0.42\textwidth}
		\centering
		\includegraphics[width=\textwidth]{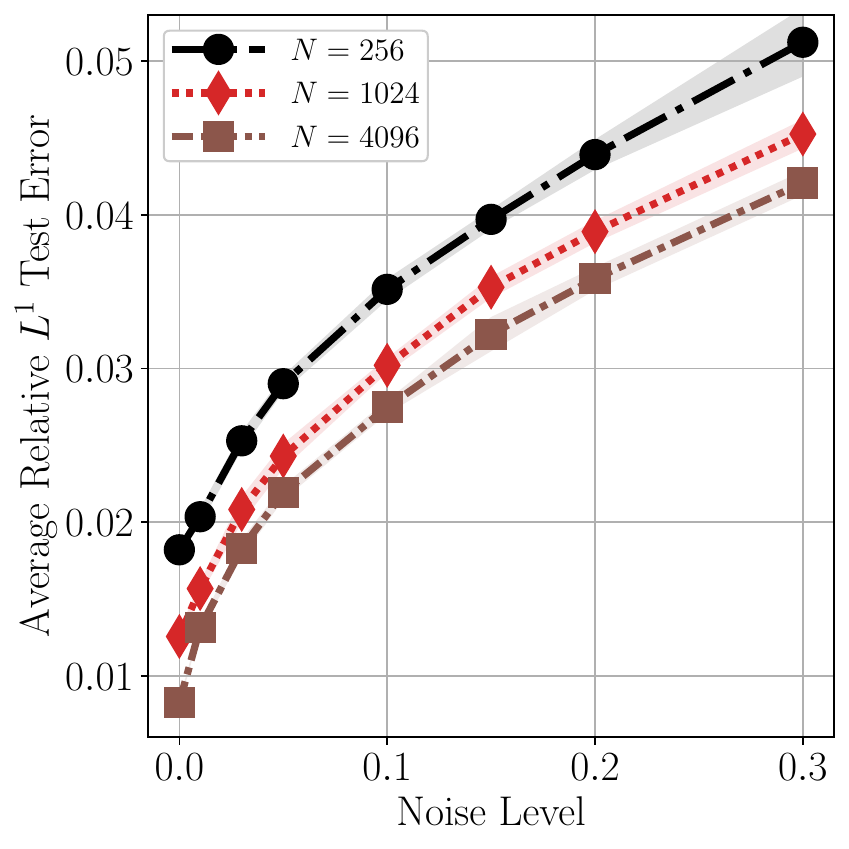}
		\caption{Robustness to large noise levels}
		\label{subfig:noise_sweep_lognormal_noisy}
	\end{subfigure}
	\begin{subfigure}[]{0.42\textwidth}
		\centering
		\includegraphics[width=\textwidth]{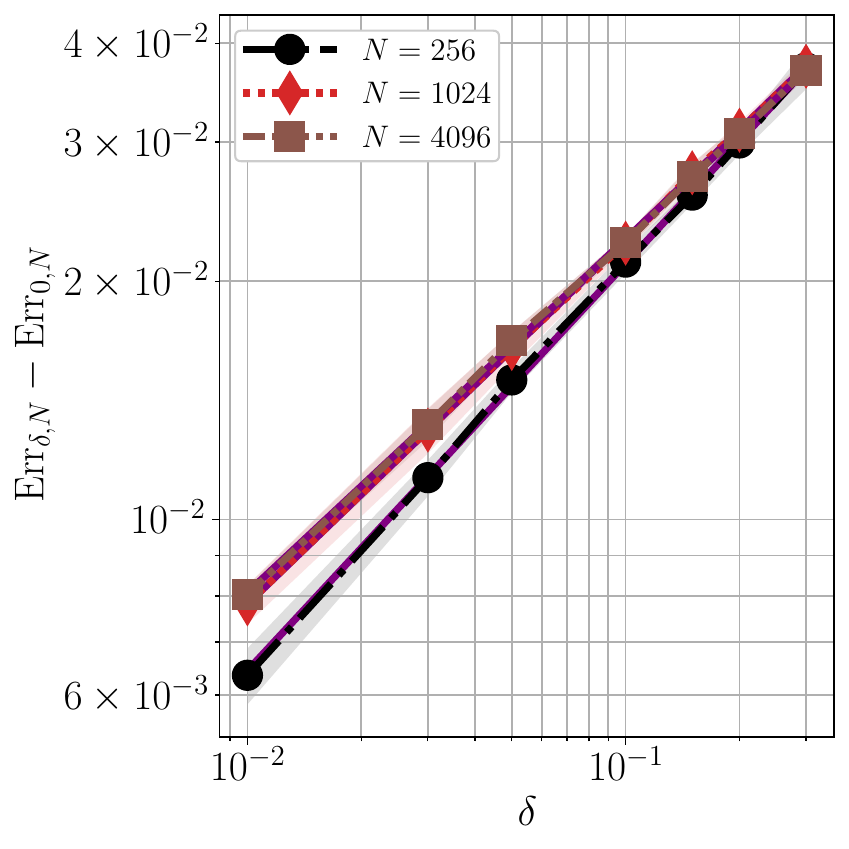}
		\caption{Fit to H\"older modulus of continuity}
		\label{subfig:noise_power_lognormal_noisy_scaled}
	\end{subfigure}
	\begin{subfigure}[]{0.42\textwidth}
		\centering
		\includegraphics[width=\textwidth]{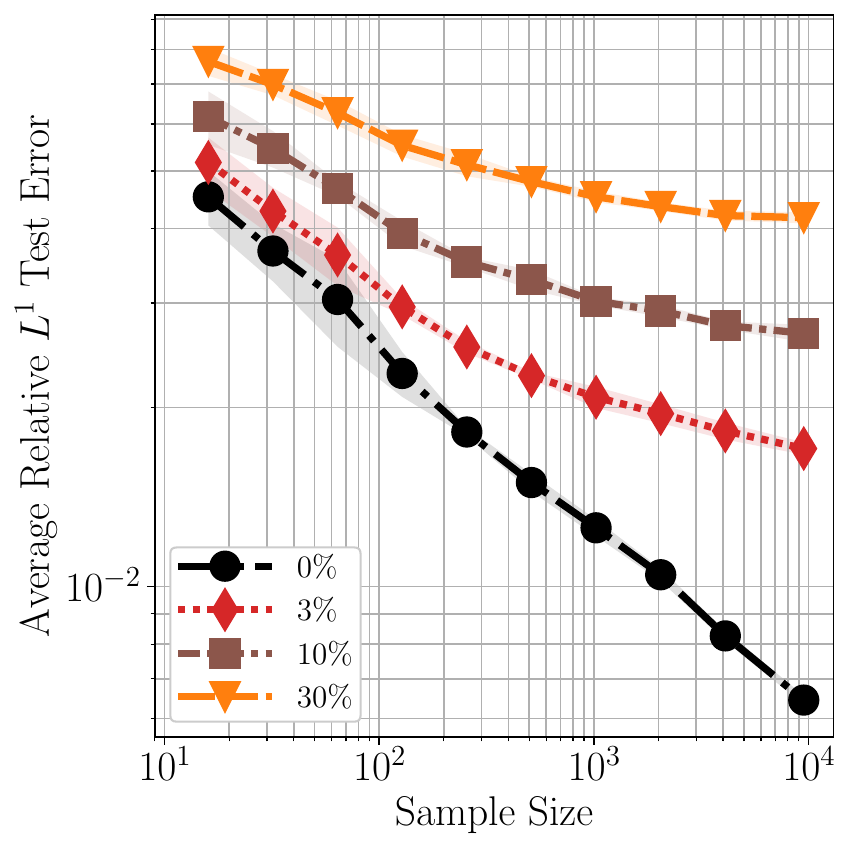}
		\caption{Algebraic sample complexity}
		\label{subfig:data_sweep_lognormal_noisy}
	\end{subfigure}
	\begin{subfigure}[]{0.42\textwidth}
		\centering
		\includegraphics[width=\textwidth]{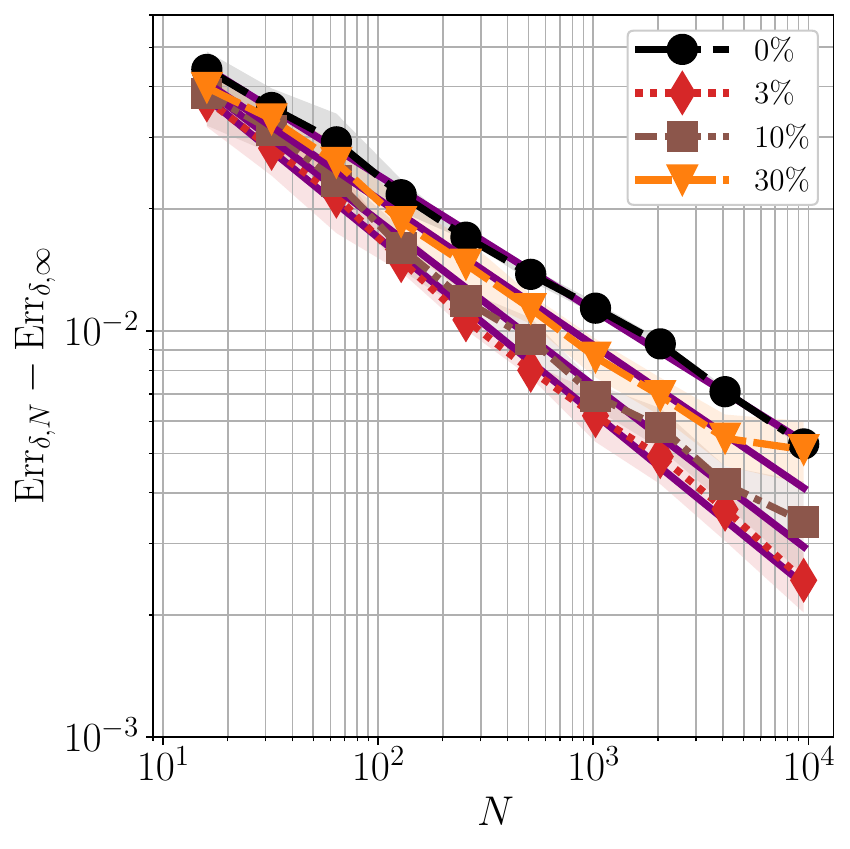}
		\caption{Fit of excess test error to power law}
		\label{subfig:data_power_lognormal_noisy}
	\end{subfigure}
    \caption{Lognormal conductivity noise robustness (top row) and sample complexity (bottom row) experiments. \Cref{subfig:noise_sweep_lognormal_noisy,subfig:noise_power_lognormal_noisy_scaled} are analogous to the top row of \cref{fig:two_phase_noisy}, except \cref{subfig:noise_power_lognormal_noisy_scaled} plots the excess error versus $\delta$ on a logarithmic scale. This strongly indicates a power law behavior in contrast to the logarithmic bound in \cref{thm:fno_approx_main}. \Cref{subfig:data_sweep_lognormal_noisy,subfig:data_power_lognormal_noisy} are analogous to the bottom row of \cref{fig:two_phase_noisy}.}
	\label{fig:lognormal_noisy}
\end{figure}

The sensitivity to noise is quantified in the top row of \cref{fig:lognormal_noisy} for the lognormal conductivity dataset. In the leftmost plot, the relative error grows substantially as $\delta$ increases, albeit still respecting sublinear growth. The actual relationship between relative error and $\delta$ is best explained by a power law, as shown in \cref{subfig:noise_power_lognormal_noisy_scaled}. For large enough $N$, the slope of the power law is independent of $N$. This suggests a H\"older-type modulus of continuity instead of the weaker logarithmic one from \cref{thm:fno_approx_main}. The sample complexity is also strongly influenced by noise. Nonasymptotically, \cref{subfig:data_sweep_lognormal_noisy} shows noise-induced saturation of the average relative test error as $N$ increases. This saturation is absent for clean train and test data (the black line with circle markers). \Cref{subfig:data_power_lognormal_noisy} subtracts out this constant bias component to isolate the statistical error and fill in the last row of \cref{tab:rates_all}. We observe near parametric rates of $O(N^{-0.45})$ for noise-free test data and noisy training data. With noisy test data, the rate exponent reduces slightly, but is still better than those corresponding to the shape detection and three phase inclusion datasets.

\section{Conclusion and outlook}\label{sec:conclusion}
This paper investigates Calder\'on's inverse conductivity problem within the framework of operator learning. The work directly approximates the inverse map that sends continuum boundary data to the interior electrical conductivity. The core of the theoretical analysis involves extending the domain of the inversion operator to a Hilbert space of integral kernel functions associated with Neumann-to-Dirichlet maps. This delivers a continuous extension that retains the logarithmic stability properties of the original inverse problem under minimal regularity assumptions on the conductivities. Such sets of admissible conductivities are compact in Lebesgue spaces. These facts yield the main result of the paper, which asserts the existence of a Fourier neural operator that uniformly approximates the inverse map with respect to admissible conductivities and noisy boundary measurements.

Numerical evidence suggests that an accurate Fourier neural operator can actually be found by computing an empirical risk minimizer corresponding to a finite training dataset of noisy samples. The trained models successfully reconstruct piecewise constant and lognormal conductivities, even in regimes that do not satisfy the theoretical assumptions. The numerical results also underscore the models' robustness to noise in the boundary data and their statistical capacity to generalize beyond the training distribution. 

In summary, the proposed framework is a principled integration of classical inverse problems theory with modern operator learning methods. Although focused on Calder\'on's problem, the methodology developed in this paper should remain valid for a wider class of nonlinear inverse problems whose solution maps can be stably extended to appropriate Hilbert spaces. This is the subject of the discussion in \cref{sec:conclusion_general}. The paper concludes in \cref{sec:conclusion_future} by identifying several directions for future research.

\subsection{Toward a general inverse map approximation framework}\label{sec:conclusion_general}
The results in \cref{sec:extend,sec:approx} leading to \cref{thm:fno_approx_main} suggest a more general framework for approximation of nonlinear inverse problem solution operators. This approach requires checking several technical conditions that must be verified on a case-by-case basis for the specific inverse problem under consideration, as done in \cref{sec:extend,sec:approx} for Calder\'on's problem. This subsection now lists these abstract conditions.

\begin{condition}[forward map, inverse map, embeddings, and compactness]\label{condition:general}
    \phantom{The}
	\begin{enumerate}[label=(\roman*),leftmargin=1.25\parindent,topsep=1.67ex,itemsep=0.5ex,partopsep=1ex,parsep=1ex]
		\item Let $X$ and $Z$ be Banach spaces and $B\subseteq X$ be a compact subset. There exists a ``forward map'' $\mathfrak{F}\colon B\subseteq X\to Z$ that is uniformly continuous.

        \emph{In Calder\'on's problem, $X=L^q(\Omega)$, $Z=\sL(H^{-1/2}_\diamond(\pOmega);H_{\phantom{\diamond}}^{1/2}(\pOmega)/\C)$, $B=\cX_d(R)$, and $\mathfrak{F}\colon \gamma\mapsto \ntd$ is H\"older continuous. See \cref{sec:prelim_eit,sec:approx_stability,sec:approx_compactness}.}
		
		\item Let $A\subseteq Z$ be a bounded subset and $Y$ be another Banach space. There exists a uniformly continuous ``inverse map'' $\cG \colon A\subseteq Z\to Y$ such that 
		\begin{align}
			\cG\bigl(\mathfrak{F}(x)\bigr)=x \qfa x\in B\,,
		\end{align}
		the range of $\mathfrak{F}$ satisfies $\mathfrak{F}(B)\subseteq A$, and $\cG$ has modulus of continuity $\omega_{\cG}$.

        \emph{In Calder\'on's problem, $\cG\colon \ntd \mapsto \gamma$, $A=\mathfrak{F}(B)$, and $Y\subseteq L^p(\Omega)$. The function $\omega_{\cG}$ is of logarithmic type. See \cref{sec:extend_stability,sec:extend_main}.}
		
		\item There exist separable Hilbert spaces $\cH_1\supseteq A$ and $\cH_2\supseteq \cG(A)$ and maps $E_1^{-1}\colon {A\subseteq Z}\to \cH_1$ and $E_2\colon {\cG(A)\subseteq Y}\to \cH_2$ such that the following holds. For some modulus of continuity $\omega_1'$, the map $E_1^{-1}$ satisfies
		\begin{align}
			\norm{E_1^{-1}(u)-E_1^{-1}(u_0)}_{\cH_1}\leq \omega_1'(\norm{u-u_0}_Z)
		\end{align}
        for all $u\in A$ and $u_0\in A$.
		For some modulus of continuity $\omega_2$, the map $E_2$ satisfies
		\begin{align}
			\norm{E_2(x)-E_2(x_0)}_{\cH_2}\leq \omega_2(\norm{x-x_0}_Y)
		\end{align}
        for all $x\in \cG(A)$ and $x_0\in \cG(A)$. 
        
		\emph{The modulus of continuity functions $\omega_1'$ and $\omega_2$ may be obtained using interpolation of Banach spaces or related arguments. In Calder\'on's problem, $\cH_1=\HS(H^s_\diamond(\pOmega);H^t(\pOmega)/\C)$ and $\cH_2=L^2(\Omega)$. Each map has action $x\mapsto x$. The modulus $\omega_1'$ can be taken to be $r\mapsto \sqrt{r}$. If $Y\subseteq L^2(\Omega)$, then $E_2$ can be taken to be the identity operator on $Y$ with modulus $\omega_2(r)=r$. See \cref{sec:extend_bounds,sec:approx_stability}.}
		
		\item There exist $E_1\colon {E_1^{-1}(A)\subseteq \cH_1} \to Z $ and  $E_2^{-1} \colon {E_2(\cG(A))\subseteq \cH_2} \to Y$ such that
		\begin{align}
			E_1\circ E_1^{-1}=\Id \qon A\qa E_2^{-1}\circ E_2=\Id\qon \cG(A)\,.
		\end{align}
		Furthermore, for some moduli of continuity $\omega_1$ and $\omega_2'$, it holds that
		\begin{align}
			\norm{E_1(v)-E_1(v_0)}_Z \leq \omega_1(\norm{v-v_0}_{\cH_1})   
		\end{align}
        for all $v$ and $v_0$ belonging to $E_1^{-1}(A)$, and
        \begin{align}
            \norm{E_2^{-1}(w)-E_2^{-1}(w_0)}_Y \leq \omega_2'(\norm{w-w_0}_{\cH_2}) 
        \end{align}
        for all $w$ and $w_0$ belonging to $E_2(\cG(A))$.
        
		\emph{In Calder\'on's problem, both maps are inclusion maps $x\mapsto x$.
		The modulus of continuity functions are $\omega_1\colon r\mapsto\sqrt{r}$ and $\omega_2'\colon r\mapsto r$. See \cref{sec:extend_bounds,sec:approx_stability}.
        }
	\end{enumerate}
\end{condition}

\begin{figure}[tb]
	\centering
    \input{figures/factorize.tex}
    \vspace{2mm}
	\caption{Factorization of an abstract inverse map $\cG$ according to \cref{condition:general} when viewed as a map $\Psi$ between suitable Hilbert spaces $\cH_1$ and $\cH_2$.}
	\label{fig:diagram}
\end{figure}

If the assertions of \cref{condition:general} are valid, then 
\begin{align}
	\Psi\defeq E_2\circ \cG\circ E_1
\end{align}
is a map $\Psi\colon E_1^{-1}(A)\subseteq \cH_1\to \cH_2$ between subsets of separable Hilbert spaces. Moreover, it holds that 
\begin{align}
	x=\cG\bigl(\mathfrak{F}(x)\bigr)= \bigl(E_2^{-1} \circ\Psi \circ E_1^{-1} \circ \mathfrak{F}\bigr) (x) \qfa x\in B
\end{align}
and so $\cG=E_2^{-1}\circ\Psi\circ E_1^{-1}$ on $\mathfrak{F}(B)$. \Cref{fig:diagram} illustrates the factorizations implied by the preceding two displays.

The map $\Psi$ has modulus of continuity $\omega_2\circ \omega_\cG\circ \omega_1$. By \cref{thm:hilbert_extension_ref}, $\Psi$ has an extension $\widetilde{\Psi}\colon \cH_1\to \cH_2$ with the same modulus of continuity. The extension $\widetilde{\Psi}$ may then be uniformly approximated by a neural operator over the compact subset $E_1^{-1}(\mathfrak{F}(B))$ of the Hilbert space $\cH_1$ up to noisy perturbations belonging to a compact subset of $\cH_1$. In the setting of EIT, these ingredients are accomplished in \cref{sec:approx_main}. The current paper further passes to kernel functions $\kappa_\gamma\in L^2(\pOmega\times\pOmega)$ for the purposes of applying FNOs, but this is not strictly necessary. Architectures that accommodate Hilbert--Schmidt operators (such as $\ntd$ when $d=2$) can be used directly \cite{de2019deep}.

One should note that the preceding discussion is not applicable to all inverse problems. For example, the required embedding estimates and stability results may not hold for inverse boundary value problems involving hyperbolic PDEs whose difference NtD maps do not exhibit exponential smoothing (cf.~\cref{lem:ntd_unif_op}); a different approach is needed to handle such situations.

\subsection{Future directions}\label{sec:conclusion_future}
A number of important directions remain open for future work. It is natural to specialize the analysis and experiments in this paper to finite-dimensional \cite{alberti2022infinite} or otherwise structured \cite{garde2025infinite} families of conductivities. Similarly, accommodating a finite number of measurements---possibly adopting a measure-centric perspective \cite{nelsen2025operator}---would bring the theory closer to EIT practice and interface with contemporary work on optimal experimental design \cite{bui2022bridging,guerra2025learning,huan2024optimal,jin2024optimal,jin2024continuous}. Another direction is to develop \emph{quantitative} approximation rates \cite{kratsios2024mixture} and statistical bounds for empirical risk minimizers \cite{de2019deep,kovachki2024data,lanthaler2023error} in the EIT setting. There is also a need for a comprehensive numerical comparison of neural operator inverse solvers against alternative operator learning, classical regularization, and hybrid reconstruction methods.
Last, the framework set forth in this paper offers a theoretical foundation that can support the data-driven solution of other nonlinear inverse problems beyond EIT, as anticipated in \cref{sec:conclusion_general}.

\backmatter

\bmhead*{Data and code availability.}
All code used to produce the data, numerical results, and figures in this paper is available at
\begin{center}
    \url{https://github.com/nickhnelsen/learning-eit}\,.
\end{center}

\bmhead*{Funding.}
M.V. de Hoop gratefully acknowledges the support of the Department of Energy BES, under grant DE-SC0020345, Oxy, the corporate members of the Geo-Mathematical Imaging Group at Rice University and the Simons Foundation under the MATH + X Program.
N.B.K. is grateful to the Nvidia Corporation for support through full-time employment. 
M.L. was partially supported by PDE-Inverse project of the European Research Council of the European Union, the FAME and Finnish Quantum flagships and the grant 336786 of the Research Council of Finland. Views and opinions expressed are those of the authors only and do not necessarily reflect those of the European Union or the other funding organizations. Neither the European Union nor the other funding organizations can be held responsible for them.
N.H.N. acknowledges support from a Klarman Fellowship through Cornell University's College of Arts \& Sciences, the U.S. National Science Foundation (NSF) under award DMS-2402036, the NSF Graduate Research Fellowship Program under award DGE-1745301, the
Amazon/Caltech AI4Science Fellowship, the Air Force Office of Scientific Research under MURI award FA9550-20-1-0358 (Machine Learning and Physics-Based Modeling and Simulation), and the Department of Defense Vannevar Bush Faculty Fellowship held by Andrew M. Stuart under Office of Naval Research award N00014-22-1-2790. 

\bmhead*{Acknowledgments.}
The authors are grateful to Richard Nickl and Andrew Stuart for many helpful discussions during the course of this project. A significant portion of this work was completed when N.B.K. and N.H.N. were with the Department of Computing and Mathematical Sciences at the California Institute of Technology. The computations presented in this paper were partially conducted on the Resnick High
Performance Computing Center, a facility supported by the Resnick Sustainability Institute at the California Institute of Technology.

\appendix

\begin{appendices}

\makeatletter
\def\@seccntformat#1{\@ifundefined{#1@cntformat}%
	{\csname the#1\endcsname\mysecspace}
	{\csname #1@cntformat\endcsname}}
\newcommand\section@cntformat{\appendixname\ \thesection.\mysecspace}       
\newcommand\subsection@cntformat{\thesubsection.\mysecspace} 				
\newcommand\subsubsection@cntformat{\thesubsubsection.\mysecspace} 				
\makeatother

\section{Remaining proofs}\label{app:proofs}
This appendix provides the remaining proofs of results in the main text as well as additional lemmas specific to the Calder\'on problem that support those proofs. While several of these results may be familiar to experts, we provide full proofs for completeness, especially because those regarding NtD maps are much less common in the literature than analogous statements for DtN maps.

First, we need a lemma showing that the total variation $V(\widetilde{\gamma};\R^d)$ of the unit extension $\widetilde{\gamma}$ in $\R^d$ is equal to the total variation $V(\gamma;\Omega)$ of the original $\gamma$ in $\Omega$. This result is used in the proofs of \cref{thm:stability,lem:compact_set}.
\begin{lemma}[total variation identity]\label{lem:tv_identity}
    If $\gamma\in\cX_{\BV}(R)$, then $V(\widetilde{\gamma};\R^d)=V(\gamma;\Omega)$.
\end{lemma}
\begin{proof}
    For any $\gamma\in\cX_{\BV}(R)$, it holds that $\widetilde{\gamma}\in L^\infty(\R^d)$ and thus $\widetilde{\gamma}\in L^1_{\mathrm{loc}}(\R^d)$. Hence, since $\R^d$ is open, $V(\widetilde{\gamma};\R^d)$ is well-defined \cite[Definition~14.2, p.~460]{leoni2017first}. For a set $U\subseteq\R^d$, write $\Phi_U$ for the set of all $\phi\in C_c^{\infty}(U;\R^d)$ such that $\norm{\phi}_{L^\infty(U;\R^d)}\leq 1$. Then
    \begin{align*}
        V(\widetilde{\gamma};\R^d) &=
        \sup_{\phi\in\Phi_{\R^d}} \sum_{i=1}^d\int_{\R^d}\widetilde{\gamma}(x)\partial_i\phi_i(x)\dd{x}\\
        &=
        \sup_{\phi\in\Phi_{\R^d}} \sum_{i=1}^d\biggl(\int_{\Omega}\widetilde{\gamma}(x)\partial_i\phi_i(x)\dd{x} + \int_{\R^d\setminus\Omega}\partial_i\phi_i(x)\dd{x}\biggr)\\
        &= \sup_{\phi\in\Phi_{\R^d}} \sum_{i=1}^d \biggl(\int_{\Omega} \gamma(x)\partial_i\phi_i(x)\dd{x} + \biggl[\int_{\R^d}\partial_i\phi_i(x)\dd{x} - \int_{\Omega}\partial_i\phi_i(x)\dd{x}\biggr]\biggr)\\
        &= \sup_{\phi\in\Phi_{\R^d}} \sum_{i=1}^d \biggl(\int_{\Omega} \bigl(\gamma(x)-1\bigr)\partial_i\phi_i(x)\dd{x}\biggr)\,.
    \end{align*}
    In the third line, the second integral equals zero by the divergence theorem due to the compact support of $\phi\in\Phi_{\R^d}$. It remains to show that
    \begin{align*}
        \sup_{\phi\in\Phi_{\R^d}} \sum_{i=1}^d \biggl(\int_{\Omega} \bigl(\gamma(x)-1\bigr)\partial_i\phi_i(x)\dd{x}\biggr) = \sup_{\phi\in\Phi_{\Omega}} \sum_{i=1}^d \biggl(\int_{\Omega} \gamma(x)\partial_i\phi_i(x)\dd{x}\biggr)
    \end{align*}
    because the rightmost quantity equals $V(\gamma;\Omega)$.    
    If $\phi\in \Phi_\Omega$, then $\sup_{x\in\R^d}\abs{\phi(x)}=\sup_{x\in\Omega}\abs{\phi(x)}\leq 1$ because $\phi$ is compactly supported in $\Omega$. Since $\Omega\subseteq\R^d$, it holds that $\phi\in C_c^\infty(\R^d;\R^d)$ as well. We have implicitly extended $\phi$ by zero to all of $\R^d$. This implies that $\Phi_\Omega\subseteq \Phi_{\R^d}$ and thus
    \begin{align*}
        V(\widetilde{\gamma};\R^d)&\geq \sup_{\phi\in\Phi_{\Omega}} \sum_{i=1}^d \biggl(\int_{\Omega} \bigl(\gamma(x)-1\bigr)\partial_i\phi_i(x)\dd{x}\biggr)\\
        &= V(\gamma;\Omega)\,.
    \end{align*}
    To achieve the final equality, we use the fact that $\supp(\phi)\subset\Omega$ and apply the divergence theorem.
    
    To show the upper bound, fix any $\phi\in\Phi_{\R^d}$ and $i\in\{1,\ldots, d\}$. Let $\rchi_U$ denote the indicator function of a set $U\subseteq\R^d$. Since $\gamma=\onebm$ on $\Omega\setminus\Omega'$ and $\widetilde{\gamma}=\onebm$ on $\R^d\setminus\Omega$ from \eqref{eqn:compact_support_set} and \eqref{eqn:conductivity_set_prelim}, it follows that $\widetilde{\gamma}=\onebm$ on $\R^d\setminus\Omega'$, all in the a.e. sense. Then
    \begin{align*}
        \int_{\Omega} \bigl(\gamma(x)-1\bigr)\partial_i\phi_i(x)\dd{x}&=\int_{\R^d}\rchi_{\Omega'}(x) \bigl(\widetilde{\gamma}(x)-1\bigr)\partial_i\phi_i(x)\dd{x}\\
        &=\int_{\R^d}\rchi_{\Omega'}(x) \bigl(\widetilde{\gamma}(x)-1\bigr)\partial_i\xi_i(x)\dd{x}\\
        &=\int_{\Omega'} \bigl(\gamma(x)-1\bigr)\partial_i\xi_i(x)\dd{x}
    \end{align*}
    for any smooth enough $\xi\colon\R^d\to\R^d$ such that $\xi_i=\phi_i$ on $\Omega'$ for each $i$. By compact containment, we can find another domain $\Omega''\subset\R^d$ such that $\overline{\Omega'}\subset\Omega''\subset\overline{\Omega''}\subset\Omega$.
    Let $\varrho\in C_c^\infty(\R^d;\R)$ be a smooth cutoff function such that $\varrho \equiv 1$ on $\overline{\Omega'}$ and $\varrho\equiv 0$ on $\R^d\setminus \Omega''$. Define $\xi\defeq \varrho\phi$ in the pointwise sense. Then $\xi$ belongs to $C_c^\infty(\Omega;\R^d)$ because smoothness is preserved under pointwise multiplication and $\supp(\xi)\subseteq \supp(\varrho)\subseteq \Omega''\subset\Omega $. Furthermore, $\sup_{x\in\Omega}\abs{\xi(x)}\leq \sup_{x\in\R^d}\abs{\xi(x)}\leq 1$ because both $\sup_{x\in\R^d}\abs{\varrho(x)}$ and $\sup_{x\in\R^d}\abs{\phi(x)}$ are bounded above by one. We deduce that $\xi\in\Phi_{\Omega}$. Summing the previous display and taking the supremum on the right yields
    \begin{align*}
        \sum_{i=1}^d \biggl(\int_{\Omega} \bigl(\gamma(x)-1\bigr)\partial_i\phi_i(x)\dd{x}\biggr)\leq \sup_{\xi'\in\Phi_\Omega}\sum_{i=1}^d \biggl(\int_{\Omega'} \bigl(\gamma(x)-1\bigr)\partial_i\xi'_i(x)\dd{x}\biggr)\,.
    \end{align*}
    By replacing $\Omega'$ with $\Omega$ in the integral, taking the supremum over $\phi\in\Phi_{\R^d}$ on the left-hand side of the preceding inequality, and applying the divergence theorem a final time, the desired result $V(\widetilde{\gamma};\R^d)\leq V(\gamma;\Omega)$ follows.
\end{proof}

The proof of \cref{lem:ntd_to_dtn} relies on the following standard fact.
\begin{lemma}[uniformly bounded DtN maps]\label{lem:dtn_uniform}
    There exists a constant $c>0$ depending on $d$, $\Omega$, $m$, and $M$ such that
    \begin{align}\label{eqn:dtn_uniform}
        \sup_{\gamma\in\Gamma}\,\norm{\dtn}_{\sL\left(H^{1/2}(\pOmega)/\C;H^{-1/2}(\pOmega)\right)}\leq c\,.
    \end{align}
\end{lemma}
\begin{proof}
    Fix $\gamma\in\Gamma$. The proof uses the variational definition of the DtN map. For any representative $f$ of equivalence class $[f]\in H^{1/2}(\pOmega)/\C$ and any $h\in H^{1/2}(\pOmega)$, it holds by definition that
    \begin{align*}
        \dtn f(h)= \int_\Omega \gamma(x)\nabla u_{\gamma,f}(x) \cdot \overline{\nabla (Eh)(x)} \dd{x}\,.
    \end{align*}
    In the preceding display, $u=u_{\gamma,f}$ solves \eqref{eqn:elliptic} with Dirichlet data $f=u|_{\pOmega}$ and $E\colon H^{1/2}(\pOmega)\to H^1(\Omega)$ is the continuous right inverse of the trace operator $v\mapsto v|_{\partial\Omega}$~\cite[Thm.~8.3, p.~39]{lions2012non}. We estimate
    \begin{align*}
        \abs{\dtn f(h)}&\leq M\int_\Omega \abs{\nabla u_{\gamma,f}(x)}\abs{\nabla(Eh)(x)}\dd{x}\\
        &\leq  M\norm{\nabla u_{\gamma,f}}_{L^2(\Omega)}\norm{\nabla (Eh)}_{L^2(\Omega)}\\
        &\leq M\norm{\nabla u_{\gamma,f}}_{L^2(\Omega)}\norm{Eh}_{H^1(\Omega)}\\
        &\leq M\norm{E}_{\sL(H^{1/2}(\pOmega);H^1(\Omega))}\norm{\nabla u_{\gamma,f}}_{L^2(\Omega)}\norm{h}_{H^{1/2}(\pOmega)}\,.
    \end{align*}
    We applied the Cauchy--Schwarz inequality twice in the preceding display. Thus,
    \begin{align*}
        \norm{\dtn f}_{H^{-1/2}(\pOmega)} &= \sup_{\norm{h}_{H^{1/2}(\pOmega)}\leq 1} \abs{\dtn f(h)}\\
        &\leq M\norm{E}_{\sL(H^{1/2}(\pOmega);H^1(\Omega))}\norm{\nabla u_{\gamma,f}}_{L^2(\Omega)}\,.
    \end{align*}
    By \cref{lem:diri_grad_estimate}, there exists $c_0>0$ independent of $\gamma$ (but depending on $M$ and $m$) such that $\norm{\nabla u_{\gamma,f}}_{L^2(\Omega)}\leq c_0 \norm{f}_{H^{1/2}(\pOmega)}$. Replacing $f$ by $f-z\onebm$, we deduce that $\norm{\dtn f}_{H^{-1/2}(\pOmega)}\leq c_0M\norm{E}\inf_{z\in\C}\norm{f-z\onebm}_{H^{1/2}(\pOmega)}$.
    The result follows by defining $c\defeq c_0M\norm{E}$ and taking the supremum over $[f]$ in the $H^{1/2}(\pOmega)/\C$ unit ball.
\end{proof}

The proof of \cref{lem:dtn_uniform} requires the following gradient estimate for the Dirichlet problem.
\begin{lemma}[Dirichlet gradient estimate]\label{lem:diri_grad_estimate}
    Let $u=u_{\gamma,f}$ denote the unique weak solution to the Dirichlet problem $\nabla\cdot(\gamma\nabla u)=0$ in $\Omega$ and $u=f$ on $\partial\Omega$. Then there exists $c>0$ depending on $d$, $M$, $m$, and $\Omega$ such that $\sup_{\gamma\in\Gamma}\norm{\nabla u_{\gamma,f}}_{L^2(\Omega)}\leq c \norm{f}_{H^{1/2}(\pOmega)}$.
\end{lemma}
\begin{proof}
    Fix $\gamma\in\Gamma$ and $f\in H^{1/2}(\pOmega)$. Consider the weak solution $u$ to the Dirichlet problem $\nabla\cdot(\gamma\nabla u)=0$ in $\Omega$ and $u=f$ on $\partial\Omega$. Recall that $E\colon H^{1/2}(\pOmega)\to H^1(\Omega)$ is the continuous right inverse of the trace operator $v\mapsto v|_{\partial\Omega}$~\cite[Thm.~8.3, p.~39]{lions2012non}. Define $w$ such that $u\defeq w+ Ef$. Then $w$ satisfies the weak formulation
    \begin{align*}
        \int_\Omega \gamma\nabla w\cdot\overline{\nabla v}\dd{x} = -\int_\Omega\gamma\nabla(Ef)\cdot\overline{\nabla v}\dd{x}\qfa v\in H^1_0(\Omega)\,.
    \end{align*}
    Taking $v=w$ and using the uniform bounds on $\gamma\in\Gamma$ gives
    \begin{align*}
        m\norm{\nabla w}_{L^2(\Omega)}^2&\leq \abs[\bigg]{\int_\Omega\gamma\nabla(Ef)\cdot\overline{\nabla w}\dd{x}}\\
        &\leq  M\norm{\nabla(Ef)}_{L^2(\Omega)}\norm{\nabla w}_{L^2(\Omega)}\\
        &\leq M\norm{E}_{\sL(H^{1/2}(\partial\Omega);H^1(\Omega))}\norm{f}_{H^{1/2}(\partial\Omega)}\norm{\nabla w}_{L^2(\Omega)}\,.
    \end{align*}
    Rearranging terms and applying the triangle inequality delivers the required estimate
    \begin{align*}
        \norm{\nabla u}_{L^2(\Omega)}\leq (1 + M/m)\norm{E}_{\sL(H^{1/2}(\partial\Omega);H^1(\Omega))}\norm{f}_{H^{1/2}(\partial\Omega)}\,.
    \end{align*}
    Taking the supremum over $\gamma$ completes the proof.
\end{proof}

Using the fact that $\gamma\in\Gamma'$ is identically one near the boundary, the next lemma provides a Lipschitz-type stability estimate for the forward solution operator of the homogeneous Neumann problem. In particular, the bound is uniform over $\Gamma'$ and allows for arbitrary Sobolev regularity of the input Neumann data. The proof follows that of an analogous result for the Dirichlet problem \cite[Lemma~19, p.~197]{abraham2019statistical}.
\begin{lemma}[Neumann problem: generalized Lax--Milgram continuity bound]\label{lem:lax_milgram_general}
    For any $s\in\R$ and any $g\in {H}_\diamond^s(\partial\Omega)$, there exists $C_s>0$ such that
    \begin{align}\label{eqn:uniform_w_continuity_bound}
        \sup_{\gamma\in\Gamma'}\norm{u_{\gamma,g}-u_{\onebm,g}}_{H^1(\Omega)/\C}\leq C_s\norm{g}_{{H}_\diamond^s(\partial\Omega)}\,,
    \end{align}
    where $u=u_{\gamma',g}$ solves \eqref{eqn:elliptic} with conductivity $\gamma=\gamma'\in\Gamma'$ and Neumann data $g$.
\end{lemma}
\begin{proof}
    Let $s\in\R$ and $g\in {H}_\diamond^s(\partial\Omega)$ be arbitrary. Let $u_{\onebm,g}$ solve the Neumann Laplace equation with Neumann data $g$, i.e., \eqref{eqn:elliptic} with conductivity $\gamma\equiv 1$. This solution exists in $H^{s+3/2}(\Omega)/\C$ and is unique by classical results \cite[Chp.~2, Remark 7.2, pp.~188--189, taking $\Omega=\Omega$, $A=-\Delta$, $f\equiv 0$, $s=s+3/2$, $m=m_0=1$, $g_0=g$, and $N=\{z\onebm\colon z\in\C\}$]{lions2012non}.    
    Also fix $\gamma\in\Gamma'$. Consider the PDE
    \begin{equation}\label{eqn:elliptic_temp}
        -\nabla\cdot(\gamma \nabla w)=\nabla\cdot(\gamma\nabla u_{\onebm,g}) \ \text{ in } \Omega \qa \dfrac{\partial w}{\partial \mathsf{n}}=0 \ \text{ on } \pOmega\, .
    \end{equation}
    Define sesquilinear form $B_\gamma$ and conjugate linear functional $f_{\gamma,g}$, respectively, by
    \begin{align*}
        B_\gamma(w,v)\defeq \int_\Omega \gamma\nabla w\cdot\overline{\nabla v}\dd{x} \qa f_{\gamma,g}(v)\defeq -\int_\Omega \gamma\nabla u_{\onebm,g}\cdot \overline{\nabla v}\dd{x}\,.
    \end{align*}
    The weak form of the Neumann problem \eqref{eqn:elliptic_temp} is to find $w=w_{\gamma,g}\in H^1(\Omega)/\C$ such that
    \begin{align}\label{eqn:weak_neumann_temp}
        B_\gamma(w,v)=f_{\gamma,g}(v)\qfa v\in H^1(\Omega)/\C\,.
    \end{align}
    Since such a weak solution $w_{\gamma,g}$ defines a weak solution $u_{\gamma,g}\defeq w_{\gamma,g} + u_{\onebm,g}$ to the original problem \eqref{eqn:elliptic} \cite[Thm.~A.2 and its proof, pp.~392--394]{hanke2011convex}, it remains to bound the $H^1(\Omega)/\C$ norm of $w_{\gamma,g}$ uniformly over the set $\Gamma'$ of conductivities.
    
    To this end, we apply the Lax--Milgram theorem. To show continuity of $B_\gamma$, we apply the Cauchy--Schwarz inequality to obtain
    \begin{align*}
        \abs{B_\gamma(w,v)}&\leq \int_\Omega \abs{\gamma(x)}\abs{\nabla w(x)}\abs{\nabla v(x)}\dd{x}\\
        &\leq M\norm{\nabla w}_{L^2(\Omega)}\norm{\nabla v}_{L^2(\Omega)}\\
        &\leq M\norm{w-z\onebm}_{H^1(\Omega)}\norm{v-z\onebm}_{H^1(\Omega)}
    \end{align*}
    for any $z\in\C$ and any $w$ and $v$ in $H^1(\Omega)/\C$. The previous display implies that
    \begin{align*}
        \abs{B_\gamma(w,v)}\leq M \norm{w}_{H^1(\Omega)/\C}\norm{v}_{H^1(\Omega)/\C}\,,
    \end{align*}
    which is the desired continuity.
    
    By the Poincar\'e--Wirtinger inequality \cite[Remark~1, p.~376]{hanke2011convex},
    \begin{align*}
        B_\gamma(v,v)=\int_\Omega \gamma(x)\abs{\nabla v(x)}^2\dd{x}
        \geq m \norm{\nabla v}^2_{L^2(\Omega)}
        \geq Cm\norm{v}^2_{H^1(\Omega)/\C}\,.
    \end{align*}
    This furnishes coercivity of $B_\gamma$.
    
    To conclude, we bound the operator norm of the conjugate linear functional $f_{\gamma,g}$. Since $\overline{\Omega'}\subset \Omega$, it holds that $\Omega=(\Omega\setminus\Omega')\cup\Omega'$. Thus,
    \begin{align*}
        -f_{\gamma,g}(v)=\int_{\Omega\setminus\Omega'} \nabla u_{\onebm,g}\cdot \overline{\nabla v}\dd{x} + \int_{\Omega'}\gamma\nabla u_{\onebm,g}\cdot \overline{\nabla v}\dd{x}
    \end{align*}
    because $\gamma\equiv 1$ on $\Omega\setminus\Omega'$. Since $\int_{\partial\Omega}g\dd{x}=0$ and $ u_{\onebm,g}\in H^1(\Omega)/\C $ solves the Neumann Laplace equation inside $\Omega$, we deduce that $\int_{\Omega} \nabla u_{\onebm,g}\cdot \overline{\nabla v}\dd{x}=0$. Therefore,
    \begin{align*}
        \int_{\Omega\setminus\Omega'} \nabla u_{\onebm,g}\cdot \overline{\nabla v}\dd{x}&=
        \int_{\Omega} \nabla u_{\onebm,g}\cdot \overline{\nabla v}\dd{x}
        - \int_{\Omega'} \nabla u_{\onebm,g}\cdot \overline{\nabla v}\dd{x}\\
        &=-\int_{\Omega'} \nabla u_{\onebm,g}\cdot \overline{\nabla v}\dd{x}\,.
    \end{align*}
    Using the previous display, we now estimate
    \begin{align*}
        \abs{f_{\gamma,g}(v)}&=\abs[\bigg]{\int_{\Omega'} \bigl(1-\gamma(x)\bigr)\nabla u_{\onebm,g}(x)\cdot \overline{\nabla v}(x) \dd{x}}\\
        &\leq \norm{\onebm-\gamma}_{L^\infty(\Omega)}\norm{\nabla u_{\onebm,g}}_{L^2(\Omega')}\norm{\nabla v}_{L^2(\Omega')}\\
        &\leq (1+M)\norm{u_{\onebm,g}}_{H^1(\Omega')/\C}\norm{v}_{H^1(\Omega)/\C}\,.
    \end{align*}
    We used an argument similar to the one showing the continuity of $B_\gamma$. It follows that the operator norm $\norm{f_{\gamma,g}}_*$ of $f_{\gamma,g}$ is bounded above by $(1+M)\norm{u_{\onebm,g}}_{H^1(\Omega')/\C}$.
    
    By interior harmonic regularity \cite[Lemma~22, p.~197, applied with exponents $1$ and $s+3/2$]{abraham2019statistical},
    \begin{align*}
        \norm{u_{\onebm,g}}_{H^{1}(\Omega')/\C}\leq C_s \norm{u_{\onebm,g}}_{H^{s+3/2}(\Omega)/\C}\,.
    \end{align*}
    Again by \cite[Chp. 2, Remark 7.2, pp.~188--189]{lions2012non}, it holds that
    \begin{align*}
        \norm{u_{\onebm,g}}_{H^{s+3/2}(\Omega)/\C}\leq C_s' \norm{g}_{H^s(\partial\Omega)}=C_s' \norm{g}_{{H}_\diamond^s(\partial\Omega)}\,.
    \end{align*}
    The equality of norms holds by \cite[Remark~(iii), p.~193]{abraham2019statistical}. Hence, $\norm{f_{\gamma,g}}_*\leq C_s'' (1+M) \norm{g}_{{H}_\diamond^s(\partial\Omega)}<\infty$. The Lax--Milgram theorem \cite[Thm.~1, Sec.~6.2.1, pp.~297--298]{evans1998partial} asserts the existence of a unique $w_{\gamma,g}\in H^1(\Omega)/\C$ solving \eqref{eqn:weak_neumann_temp}. In particular, $B_\gamma(w_{\gamma,g},w_{\gamma,g})=f_{\gamma,g}(w_{\gamma,g})$. This equality, the operator norm bound on $f_{\gamma,g}$, and the coercivity of $B_\gamma$ imply \eqref{eqn:uniform_w_continuity_bound} as required.    
\end{proof}

\Cref{lem:lax_milgram_general} is used in the following proof of \cref{lem:ntd_unif_op}, which pertains to the Neumann problem. The proof closely follows that of \cite[Lemma~20, p.~197]{abraham2019statistical} for the Dirichlet problem.
\begin{proof}[Proof of \cref{lem:ntd_unif_op}]
    In the proof, we freely allow $C$ to denote a constant that changes from line to line.
    Consider the Neumann problem~\eqref{eqn:elliptic} with $\gamma\in\Gamma'$. To set up the proof, define connected open sets $D$, $U$, and $U_0$ in $\R^d$ such that
	\begin{align*}
		\overline{\Omega'}\subset D\subset\overline{D}\subset\Omega 
	\end{align*}
	and
	\begin{align*}
		\partial D \subseteq U\subset \overline{U}\subset U_0\subset\overline{U_0}\subset \Omega\setminus\Omega'\,.
	\end{align*}
	Let the Neumann data be denoted by $g\in{H}_\diamond^s(\pOmega)$. Define the difference of solutions
	\begin{align*}
		w\defeq u_{\gamma,g}-u_{\onebm,g}\,,
	\end{align*}
	where $u_{\gamma,g}$ solves \eqref{eqn:elliptic} with fixed but arbitrary conductivity $\gamma$ and $u_{\onebm,g}$ solves \eqref{eqn:elliptic} with $\gamma\equiv 1$ (i.e., the Laplace equation). In the following, it will be important to recognize that
	\begin{align*}
		\lap w=0 \qon \Omega\setminus \Omega'
	\end{align*}
	because $\gamma\equiv 1$ on $\Omega\setminus \Omega'$. Thus, $w=w_{\gamma,g}$ is also harmonic on the domain $\Omega\setminus\overline{D}\subset \Omega\setminus \Omega'$. This fact will allow us to control the dependence of constants on $\gamma$.
	
	Let $t\in\R$. We begin by estimating the boundary traces of $w$. To this end, we first note by \Cref{lem:app_boundary}~\ref{item:lem_bdry1} that $\pOmega\subseteq \partial(\Omega\setminus\overline{D})$ because $\overline{D}\subseteq\Omega$. Thus, $\norm{w}_{H^t(\pOmega)/\C}\leq C\norm{w}_{H^t(\partial(\Omega\setminus\overline{D}))/\C}$ by monotonicity and the norm equivalence of the spectral (recall Eqn.~\ref{eqn:defn_boundary_sobolev_spectral}) and interpolation definitions of fractional Sobolev norms on compact boundary manifolds 
	\cite[Chp.~1, Sec.~7.3, pp.~34--37]{lions2012non}.
	
	Next, for any $z\in\C$, we apply appropriate boundary trace theorems to the harmonic function $w-z\onebm$ to obtain
	\begin{align*}
		\norm{w}_{H^t(\partial(\Omega\setminus\overline{D}))/\C}\leq \norm{w-z\onebm}_{H^t(\partial(\Omega\setminus\overline{D}))}\leq C \norm{w-z\onebm}_{H^{t+1/2}(\Omega\setminus\overline{D})}\,.
	\end{align*}
	For $t>0$, the preceding result follows from \cite[Chp.~1, Thm.~9.4, pp.~41--42, taking $\mu=0$ and $s=t+1/2$]{lions2012non}. For $t\leq -1/2$, we invoke \cite[Chp.~2, Thm.~6.5 and Remark 6.4, pp.~175--177, taking $m=1$, $m_0=0$, $r=-t-1/2$, and $A=-\Delta$]{lions2012non}. The final case $-1/2<t\leq 0$ is due to \cite[Chp.~2, Thm.~7.3, pp.~187--188, taking $m=1$, $m_0=0$, $0<s=t+1/2<2$, and $A=-\Delta$]{lions2012non}. Taking the infimum over $z\in\C$ in the previous display delivers the bound
	\begin{align*}
		\norm{w}_{H^t(\partial(\Omega\setminus\overline{D}))/\C}\leq C \norm{w}_{H^{t+1/2}(\Omega\setminus\overline{D})/\C}\,.
	\end{align*}
	
	To bound the right hand side from above in the last display, we use stability bounds for the solution operator of the Neumann problem applied to the Laplace  equation on the domain $\Omega\setminus\overline{D}$. First, since the boundary of a set equals the boundary of its complement,
	\begin{align*}
		\partial(\Omega\setminus\overline{D}) = \partial(\Omega \cap (\overline{D})^\comp)=\partial(\Omega^\comp\cup \overline{D}) \subseteq \partial \Omega \cup \partial \overline{D}\subseteq \partial \Omega \cup \partial D\,.
	\end{align*}
	The first inclusion follows from $\partial(A\cup B)\subseteq \partial A\cup \partial B$ and the second from $\partial\overline{A}\subseteq \partial A$ for any sets $A$ and $B$ in a metric space. Next, application of \cite[Chp.~2, Remark 7.2, pp.~188--189, taking $\Omega=\Omega\setminus\overline{D}$, $A=-\Delta$, $f\equiv 0$, $s=t+1/2$, $m=m_0=1$, and $N=\{z\onebm\colon z\in\C\}$]{lions2012non} and the monotonicity implied by the previous display deliver the bounds
	\begin{align*}
		\norm{w}_{H^{t+1/2}(\Omega\setminus\overline{D})/\C}&\leq C \norm[\bigg]{\frac{\partial w}{\partial\mathsf{n}}}_{H^{t-1}(\partial(\Omega\setminus\overline{D}))}\\
		&\leq C\left(\norm[\bigg]{\frac{\partial w}{\partial\mathsf{n}}}_{H^{t-1}(\partial\Omega)} + \norm[\bigg]{\frac{\partial w}{\partial\mathsf{n}}}_{H^{t-1}(\partial D)}\right)\\
		&= C\norm[\bigg]{\frac{\partial w}{\partial\mathsf{n}}}_{H^{t-1}(\partial D)}\,.
	\end{align*}
	The equality in the final line holds because the Neumann data satisfies 
	\begin{align*}
		\frac{\partial w}{\partial\mathsf{n}}=\frac{\partial u_{\gamma,g}}{\partial\mathsf{n}}-\frac{\partial u_{\onebm,g}}{\partial\mathsf{n}}=g-g=0 \qon \pOmega\,.
	\end{align*}
	Next, we bound the Neumann data by the interior solution. Let $S\defeq D\cap U$. Since $\partial D\subseteq U$ by construction, \cref{lem:app_boundary}~\ref{item:lem_bdry2} implies that 
	\begin{align*}
		\partial S=\partial(D\cap U)\supseteq (U\cap\partial D)\cup(D\cap\partial U)=\partial D \cup(D\cap\partial U)\supseteq \partial D\,.
	\end{align*}
	Similar application of boundary trace theorems to the domain $S\subseteq U$ leads to
	\begin{align*}
		\norm[\bigg]{\frac{\partial w}{\partial\mathsf{n}}}_{H^{t-1}(\partial D)} = \norm[\bigg]{\frac{\partial (w-z\onebm)}{\partial\mathsf{n}}}_{H^{t-1}(\partial D)}
		&\leq C \norm[\bigg]{\frac{\partial (w-z\onebm)}{\partial\mathsf{n}}}_{H^{t-1}(\partial S)}\\
		&\leq C \norm{w-z\onebm}_{H^{t+1/2}(S)}\\
		&\leq C \norm{w-z\onebm}_{H^{t+1/2}(U)}
	\end{align*}
	for any $z\in\C$. For $t>1$, \cite[Chp.~1, Thm.~9.4, pp.~41--42, taking $\mu=1$, $m=1$, and $s=t+1/2$]{lions2012non} gives the second inequality. The case $t\leq -1/2$ follows from \cite[Chp.~2, Thm.~6.5 and Remark 6.4, pp.~175--177, taking $m=1$, $m_0=1$, $r=-t-1/2$, and $A=-\Delta$]{lions2012non}. Finally, the inequality holds for $-1/2<t\leq 1$ by \cite[Chp.~2, Thm.~7.3, pp.~187--188, taking $m=1$, $m_0=1$, $0<s=t+1/2<2$, and $A=-\Delta$]{lions2012non}.
	
	Continuing, we use the preceding display and the interior regularity of harmonic functions \cite[Lemma~22, p.~197]{abraham2019statistical} to yield
	\begin{align*}
		\norm[\bigg]{\frac{\partial w}{\partial\mathsf{n}}}_{H^{t-1}(\partial D)}\leq C \norm{w}_{H^{t+1/2}(U)/\C}
		\leq C'\norm{w}_{H^r(U_0)/\C}
	\end{align*}
	for any $r\in\R$ (with $C'$ depending on $r$), and in particular for $r=1$. Then by the set monotonicity $U_0\subset \Omega\setminus \Omega'\subseteq \Omega$, it holds that
	\begin{align*}
		\norm[\bigg]{\frac{\partial w}{\partial\mathsf{n}}}_{H^{t-1}(\partial D)}\leq C'\norm{w}_{H^1(U_0)/\C}
		\leq C'\norm{w}_{H^1(\Omega)/\C}\,.
	\end{align*}
	Use of the generalized Lax--Milgram result from \cref{lem:lax_milgram_general} delivers the upper bound
	\begin{align*}
		\norm{w}_{H^1(\Omega)/\C}\leq \norm{w}_{H^1(\Omega)}\leq C_s\frac{M}{m}\norm{g}_{{H}_\diamond^s(\partial\Omega)}
	\end{align*}
	for a constant $C_s$ that is independent of the conductivity $\gamma$.
	
    Putting together the pieces gives
    \begin{align*}
        \norm{w}_{H^t(\partial \Omega)/\C} = \norm{(\ntd - \sR_{\onebm})g}_{H^t(\partial \Omega)/\C}\leq \widetilde{C}_{s,t}\frac{M}{m}\norm{g}_{H_{\diamond}^s(\partial\Omega)}
    \end{align*}
    for a constant $\widetilde{C}_{s,t}>0$ depending on $s$ and $t$. Therefore,
    \begin{align}\label{eqn:proof_uniform_shift_ntd}
        \sup_{\gamma\in\Gamma'}\norm{\ntd - \sR_{\onebm}}_{\sL({H}_\diamond^s(\partial\Omega);H^t(\partial\Omega)/\C)}\leq \widetilde{C}_{s,t}\frac{M}{m}<\infty
    \end{align}
    as asserted.
	
    To finish the proof, we must show the first assertion of the lemma. The usual boundary trace and stability theorems applied to the harmonic function  $u_{\onebm,g}$ give
    \begin{align*}
        \norm{\sR_{\onebm}g}_{H^{s+1}(\pOmega)/\C} = \norm{u_{\onebm,g}}_{H^{s+1}(\pOmega)/\C}&\leq C\norm{u_{\onebm,g}}_{H^{s+3/2}(\Omega)/\C}\\
        &\leq C' \norm{g}_{H^s(\pOmega)}\,.
    \end{align*}
    Specifically, for $s>-1$, the first inequality is due to \cite[Chp.~1, Thm.~9.4, pp.~41--42, taking $\mu=1$, $m=1$, and $s=s+3/2$]{lions2012non}, while for $s\leq -3/2$ and $-3/2<s\leq -1 $ it follows from \cite[Chp.~2, Thm.~6.5 and Remark 6.4, pp.~175--177, taking $m=1$, $m_0=0$, $r=-s-3/2$, and $A=-\Delta$]{lions2012non} and \cite[Chp.~2, Thm.~7.3, pp.~187--188, taking $m=1$, $m_0=0$, $0<s=s+3/2<2$, and $A=-\Delta$]{lions2012non}, respectively. The second inequality in the previous display is a consequence of \cite[Chp.~2, Remark 7.2, pp.~188--189, taking $\Omega=\Omega$, $A=-\Delta$, $f\equiv 0$, $s=s+3/2$, $m=m_0=1$, $g_0=g$, and $N=\{z\onebm\colon z\in\C\}$]{lions2012non}.
    From \eqref{eqn:proof_uniform_shift_ntd} and the preceding display,
    \begin{align*}
        \sup_{\gamma\in\Gamma'} \norm{\ntd}_{\sL({H}_\diamond^s;H^{s+1}/\C)}&\leq \sup_{\gamma\in\Gamma'} \Bigl(\norm{\ntd - \sR_{\onebm}}_{\sL(H^s_\diamond;H^{s+1}/\C)} + \norm{\sR_{\onebm}}_{\sL(H^s_\diamond;H^{s+1}/\C)}\Bigr)\\
        &\leq \widetilde{C}_{s,s+1}\frac{M}{m} + C'\,.
    \end{align*}
    This is the required result.
\end{proof}

This appendix concludes with the proof of \cref{lem:op_to_hs} from \cref{sec:approx_stability}.
\begin{proof}[Proof of \cref{lem:op_to_hs}]
    The proof is similar to that of \cref{lem:hs_to_op}. 
	Let $\gamma\in\Gamma'$, $\gamma_0\in\Gamma'$, and $s$ and $t$ be arbitrary. Denote the shifted NtD maps by $\tntd\defeq \ntd-\sR_{\onebm}$ and $\tntdz\defeq \ntdz - \sR_{\onebm}$. Let $\norm{\slot}$ denote the $\HS(H^s_\diamond(\pOmega);H^t(\pOmega)/\C)$ norm for brevity. We begin with the triangle inequality. For any $J\in\N$, it holds that
	\begin{align*}
		\norm{\ntd-\sR_{\gamma_0}} &\leq \norm{{\tntd} - P_{J}\tntd} + \norm{P_{J}\tntd - P_{J}\tntdz}  + \norm{P_{J}\tntdz - \tntdz}\\
		&\leq 2\sup_{\gamma'\in\Gamma'}\norm{\widetilde{\sR}_{\gamma'} - P_{J}\widetilde{\sR}_{\gamma'}} + \norm{P_{J}\tntd - P_J\tntdz}\,.
	\end{align*}
	We used the linear projection operator $P_J=\Pi_{JJ}$ from \eqref{eqn:projection_operator}. For any $\varrho \geq 0$, \cref{lem:op_to_hs_rn} (applied with $K=J$, $r=s$, $s=t$, $p=s-\varrho(d-1)$, and $q=t+\varrho(d-1)$) shows that
	\begin{align*}
		\norm{{\tntd} - P_J\tntd}\leq C \norm{\tntd}_{\sL(H^{s-(\varrho+1)(d-1)}_\diamond(\pOmega);H_{\phantom{\diamond}}^{t+\varrho(d-1)}(\pOmega)/\C)} J^{-\varrho}\leq C_{\varrho,s,t,d} J^{-\varrho}\,.
	\end{align*}
	The final inequality follows by \cref{lem:ntd_unif_op} (applied with $s=s-(\varrho+1)(d-1)$ and $t=t+\varrho(d-1)$). The constant $C_{\varrho,s,t,d}>0$ depends on $\varrho$, $s$, $t$, and $d$ but does not depend on $\gamma$. Thus, we have also obtained a bound on the supremum.
    
    Next, \Cref{lem:ntd_norm_equiv} (with $K=J$, $r=s$, $s=t$, $p=(d-1)-1/2$, and $q=1/2$) yields
	\begin{align*}
		\norm{P_J\tntd - P_J\tntdz}&\leq C J^{\frac{\beta(s,t,d)}{d-1}}\norm{\tntd -\tntdz}_{\HS(H^{(d-1)-1/2}_\diamond(\pOmega);H_{\phantom{\diamond}}^{1/2}(\pOmega)/\C)}\\
		&\leq C' J^{\frac{\beta(s,t,d)}{d-1}}\norm{\ntd -\ntdz}_{\sL(H^{-1/2}_\diamond(\pOmega);H_{\phantom{\diamond}}^{1/2}(\pOmega)/\C)}\,,
	\end{align*}
	where $\beta(s,t,d)\defeq (d-3/2-s)_+ + (t-1/2)_+\geq 0$. The inequality in the last line of the previous display is due to \Cref{lem:op_to_hs_rn} (applied with $r=p=(d-1)-1/2$ and $s=q=1/2$). 
    
    By the triangle inequality and the first assertion of \Cref{lem:ntd_unif_op}, it holds that 
    \begin{align*}
    \sup_{(\gamma,\gamma_0)\in\Gamma'\times\Gamma'}\norm{\ntd -\ntdz}_{\sL(H^{-1/2}_\diamond(\pOmega);H_{\phantom{\diamond}}^{1/2}(\pOmega)/\C)}\leq C\,.
    \end{align*}
    Thus, we may choose the natural number $J$ such that
	\begin{align*}
		J \asymp \norm{\ntd -\ntdz}_{\sL(H^{-1/2}_\diamond(\pOmega);H_{\phantom{\diamond}}^{1/2}(\pOmega)/\C)}^{-\frac{(d-1)}{\varrho(d-1) + \beta(s,t,d)}}
	\end{align*}
	balances the two terms $J^{-\varrho} \asymp J^{\beta(s,t,d)/(d-1)}\norm{\ntd -\ntdz}_{\sL(H^{-1/2}_\diamond(\pOmega);H_{\phantom{\diamond}}^{1/2}(\pOmega)/\C)}$. We deduce that
	\begin{align*}
		\norm{\ntd-\sR_{\gamma_0}}\leq C \norm{\ntd -\ntdz}_{\sL(H^{-1/2}_\diamond(\pOmega);H_{\phantom{\diamond}}^{1/2}(\pOmega)/\C)}^{\frac{\varrho(d-1)}{\varrho(d-1) + \beta(s,t)}}\,.
	\end{align*}
	Finally, choosing $\varrho\defeq \beta(s,t,d)/(d-1)$ if $\beta(s,t,d)>0$ and $\varrho>0$ otherwise completes the proof.
\end{proof}

\section{Auxiliary lemmas}\label{app:lemmas}
This appendix collects several auxiliary technical lemmas that are agnostic to the EIT application setting.

The first lemma develops set inclusions involving boundaries. We recall that the boundary of a set is the intersection of its closure with the closure of its complement.
\begin{lemma}[comparison of boundaries]\label{lem:app_boundary}
	Let $(X, \mathsf{d})$ be a metric space and $A\subseteq X$ and $B\subseteq X$ be subsets. The following holds true.
	\begin{enumerate}[label=(\roman*)]
		\item If $\overline{B}\subseteq \Int(A)$, then $\partial A\subseteq \partial(A\setminus B)$.\label{item:lem_bdry1}
		
		\item If $A$ and $B$ are both open in $(X, \mathsf{d})$, then $(A\cap\partial B)\cup (\partial A\cap B)\subseteq \partial(A\cap B)\setminus (\partial A\cap\partial B)$.\label{item:lem_bdry2}
	\end{enumerate}
\end{lemma}
\begin{proof}
	We begin with \cref{item:lem_bdry1}. By definition of boundary and by basic set operations,
	\begin{align*}
		\partial A &=\overline{A^\comp}\cap \overline{A}\\
		&\subseteq \overline{A^\comp}\\
		&\subseteq \overline{A^\comp}\cup\overline{B}\\
		&=\overline{A^\comp\cup B}\\
		&=\overline{(A\cap B^\comp)^\comp}\\
		&=\overline{(A\setminus B)^\comp}\,.
	\end{align*}
	Moreover, $A=(A\setminus B)\cup (A\cap B)\subseteq (A\setminus B)\cup B$. Therefore,
	\begin{align*}
		\partial A &=\overline{A^\comp}\cap \overline{A}\\
		&\subseteq \overline{A^\comp}\cap \overline{(A\setminus B)\cup B}\\
		& = \overline{A^\comp}\cap\bigl(\overline{A\setminus B} \cup \overline{B}\bigr)\\
		&= \bigl(\overline{A^\comp}\cap \overline{A\setminus B}\bigr) \cup \bigl(\overline{A^\comp}\cap\overline{B}\bigr)\\
		&\subseteq \bigl(\overline{A\setminus B}\bigr) \cup \bigl(\overline{A^\comp}\cap\overline{B}\bigr)\,.
	\end{align*}
	By the hypothesis $\overline{B}\subseteq \Int(A)$, it holds that $\overline{A^\comp}\cap\overline{B}\subseteq \overline{A^\comp}\cap \Int(A) = \emptyset$. Thus, $\partial A\subseteq \overline{A\setminus B}$. It follows that $\partial A\subseteq \overline{A\setminus B} \cap \overline{(A\setminus B)^\comp}=\partial(A\setminus B)$ as asserted in \cref{item:lem_bdry1}.
	
	Next, we prove \cref{item:lem_bdry2} under the hypothesis that $A$ and $B$ are open. By symmetry, it suffices to prove that $(A\cap\partial B)\subseteq \partial(A\cap B)\setminus (\partial A\cap\partial B)$. If $x\in A\cap\partial B$, then $x\in A$. Thus, $x\notin\partial A=\overline{A}\setminus\Int(A)$ because $A=\Int(A)$ is open. We have shown that $x\notin (\partial A\cap\partial B)$. It remains to show that $x\in\partial(A\cap B)$. By the first part of the proof of \cref{item:lem_bdry1}, it holds that
	\begin{align*}
		A\cap\partial B=A\cap \overline{B}\cap\overline{B^\comp}\subseteq \overline{B^\comp}\subseteq \overline{(B\cap S)^\comp}
	\end{align*}
	for any set $S$. In particular, taking $S=A$ gives $x\in \overline{(A\cap B)^\comp}$. By the second part of the proof of \cref{item:lem_bdry1}, for any set $T$ it holds that
	\begin{align*}
		\partial B = \overline{B^\comp}\cap \overline{B}\subseteq  \overline{B^\comp}\cap\overline{(B\setminus T)\cup T}= (\overline{B^\comp} \cap \overline{B\setminus T}) \cup (\overline{B^\comp} \cap \overline{T})\,.
	\end{align*}
	Thus, $A\cap \partial B\subseteq (A\cap \overline{B^\comp} \cap \overline{B\setminus T}) \cup (\overline{B^\comp} \cap \overline{T}\cap A)$. The choice $T=A^\comp$ and the fact that $A^\comp$ is closed yield the set inclusions
	\begin{align*}
		A\cap \partial B& \subseteq (A\cap \overline{B^\comp} \cap \overline{B\setminus A^\comp}) \cup (\overline{B^\comp} \cap A^\comp \cap A)\\
		&= (A\cap \overline{B^\comp} \cap \overline{A\cap B}) \cup (\overline{B^\comp} \cap \emptyset)\\
		& = (A\cap \overline{B^\comp} \cap \overline{A\cap B})\\
		&\subseteq \overline{A\cap B}\,.
	\end{align*}
	We deduce that $x\in \overline{A\cap B}\cap \overline{(A\cap B)^\comp} = \partial (A\cap B)$ as required.
\end{proof}

The statements of the next two lemmas involve the family of orthogonal projection operators $\{\Pi_{JK}\}_{(J,K)\in\N\times\N}$ from \eqref{eqn:projection_operator}. The first result of the two is due to~\cite[Lemma~18, p. 195]{abraham2019statistical}. It controls Hilbert--Schmidt norms and Hilbert--Schmidt projection errors by operator norms involving dimension-dependent input space smoothness.
\begin{lemma}[general operator norm and Hilbert--Schmidt norm comparisons]\label{lem:op_to_hs_rn}
    Let $d\geq 2$ and $\Omega\subset \R^d$ be a bounded domain with smooth boundary $\pOmega$. 
	Let $p\leq r$ and $q\geq s$. There exist positive numbers $C$ and $C'$ depending only on $\Omega$ and the differences $p-r$ and $q-s$ such that the following holds. If $T\in\sL(H^{p-(d-1)}_\diamond(\pOmega);H_{\phantom{\diamond}}^q(\pOmega)/\C)$, then $T\in\HS(H^r_\diamond(\pOmega);H^s(\pOmega)/\C)$ and
	\begin{align}
		\norm{T}_{\HS(H^r_\diamond(\pOmega);H^s(\pOmega)/\C)}\leq C \norm{T}_{\sL(H^{p-(d-1)}_\diamond(\pOmega);H_{\phantom{\diamond}}^q(\pOmega)/\C)}\,.
	\end{align}
    Additionally, for any $J\in\N$ and $K\in\N$, it holds that
	\begin{align}
		&\norm{T-\Pi_{JK}T}_{\HS(H^r_\diamond(\pOmega);H^s(\pOmega)/\C)}\\
        &\qquad\qquad\qquad\qquad \leq C' \norm{T}_{\sL(H^{p-(d-1)}_\diamond(\pOmega);H_{\phantom{\diamond}}^q(\pOmega)/\C)}\max\Bigl(J^{-\bigl(\frac{r-p}{d-1}\bigr)}, K^{-\bigl(\frac{q-s}{d-1}\bigr)}\Bigr)\,.\notag
    \end{align}
\end{lemma}

Recall that $x_+\defeq\max(x,0)$ is the positive part of a number $x$. The final lemma of this appendix is implied by \cite[Lemma~17, p.~195]{abraham2019statistical}. It can be viewed as a quantitative form of finite-dimensional norm equivalence for Hilbert--Schmidt operators.
\begin{lemma}[finite-dimensional Hilbert--Schmidt norms]\label{lem:ntd_norm_equiv}
    Let $d\geq 2$ and $\Omega\subset \R^d$ be a bounded domain with smooth boundary $\pOmega$. For any real $p$, $q$, $r$, and $s$, there exists $C>0$ depending only on $\Omega$, $p-r$, and $s-q$ such that for all $T\in\HS(H^p_\diamond(\pOmega);H_{\phantom{\diamond}}^q(\pOmega)/\C)$,
	\begin{align}
    \norm{\Pi_{JK}T}_{\HS(H_\diamond^r(\pOmega);H_{\phantom{\diamond}}^s(\pOmega)/\C)}\leq C J^{\frac{(p-r)_+}{d-1}}K^{\frac{(s-q)_+}{d-1}}\norm{\Pi_{JK}T}_{\HS(H_\diamond^p(\pOmega);H_{\phantom{\diamond}}^q(\pOmega)/\C)}\,.
	\end{align}
\end{lemma}

\section{Numerical experiment details}\label{app:numerics}
This appendix expands upon \cref{sec:numerics} by providing supplementary details about the numerical results, dataset generation, implementation of FNOs, and computing environment used in the experiments.

In this appendix, we take $\Omega\defeq\D$ to be the unit disk and interpret $\T\simeq[0, 2\pi]_{\mathrm{per}}\simeq\partial\D$. For $a<b$, let $\Unif[a,b]$ be the uniform distribution on the interval $[a,b]$. For some $x\in\R^2$ and radius $0<r<1$, define the open ball $B_r(x)\defeq\{y\in\R^2\colon \abs{y-x}<r\}$. Thus, $\D=B_1(0)$. The set of eigenfunctions $\{\varphi_j\}_{j\in\Z_{\geq 0}}$ of the Laplace--Beltrami operator $-\Delta_{\partial\D}$ is the standard real Fourier basis (i.e., sines and cosines) of $L^2(\T;\R)$. For computational convenience, we re-index this basis in terms of complex exponentials as $\{\varphi_j\}_{j\in\Z}\subset L^2(\T;\C)$. That is, $\theta\mapsto\varphi_j(\theta)=\exp(\mathsf{i}j\theta)/\sqrt{2\pi}$ for each $j\in\Z$.

\subparagraph*{\emph{\textbf{Further results and discussion.}}}
We now present supplemental results that support those in \cref{sec:numerics}. To set the notation and conventions, we recall that shaded bands in all figures denote two standard deviations from the mean (excess) error over five independent training runs. The rates in \cref{tab:rates_all} are estimated using nonlinear curve fitting based on the mean expected relative $L^1(\D)$ test errors over the five training runs. In addition to $L^1(\D)$ errors, we also considered the $L^0(\D)$ ``distance'' and Dice score for the shape detection experiments. These are defined by
\begin{align}
    \sfd_{L^0(\D)}(\gamma,\gamma')\defeq \int_{\D} \rchi_{\{\rchi_{\{\gamma(x)>c\}}\neq \rchi_{\{\gamma'(x)>c\}}\}} \dd{x}
\end{align}
and
\begin{align}
    \mathsf{DSC}(\gamma,\gamma')\defeq \frac{2\int_{\D}\rchi_{\{\gamma(x)>c\}}\rchi_{\{\gamma'(x)>c\}}\dd{x}}{\epsilon + \int_{\D}\rchi_{\{\gamma(x)>c\}}\dd{x} + \int_{\D}\rchi_{\{\gamma'(x)>c\}}\dd{x}}\,,
\end{align}
respectively. We take $\epsilon=10^{-8}$ and $c=50$ because the true $\gamma$ takes values in $\{1,100\}$.

The purple lines in \cref{fig:two_phase_noisy,fig:three_phase,fig:lognormal_noisy} are linear least squares fits (on log-log scales) to either noise laws of the form
$\mathrm{Err}_{N,\delta} - e_N=C \delta^{\rho}$ (H\"older power law) or $\mathrm{Err}_{N,\delta} - e_N=C (\log(1/\delta))^{-\rho}$ (logarithmic modulus of continuity), and a sample complexity power law $\mathrm{Err}_{N,\delta} - e_\delta=cN^{-r}$. The parameters $C$, $c$, $\rho$, $e_N$, $e_\delta$, and $r$ are found by nonlinear curve fitting. The terms $e_N$ and $e_\delta$ represent irreducible errors due to fixed sample size or noise level, respectively; they contribute to the saturation of the curves in \cref{fig:two_phase_noisy,fig:three_phase,fig:lognormal_noisy}. We find that the shape detection and three phase inclusion dataset errors are best explained by logarithmic noise laws, while those of the lognormal dataset fit a power law better as in \cref{subfig:noise_power_lognormal_noisy_scaled}. The same conclusions here and in \cref{sec:numerics} hold for noisy training data and \emph{noise-free} test data. Clean test data error still increases as the training data noise level increases, but more mildly as seen in \cref{fig:noise_all_clean}. The corresponding FNOs are regularized by absorbing the noise from the training data.

\begin{figure}[tb]
	\centering
	\begin{subfigure}[]{0.325\textwidth}
		\centering
		\includegraphics[width=\textwidth]{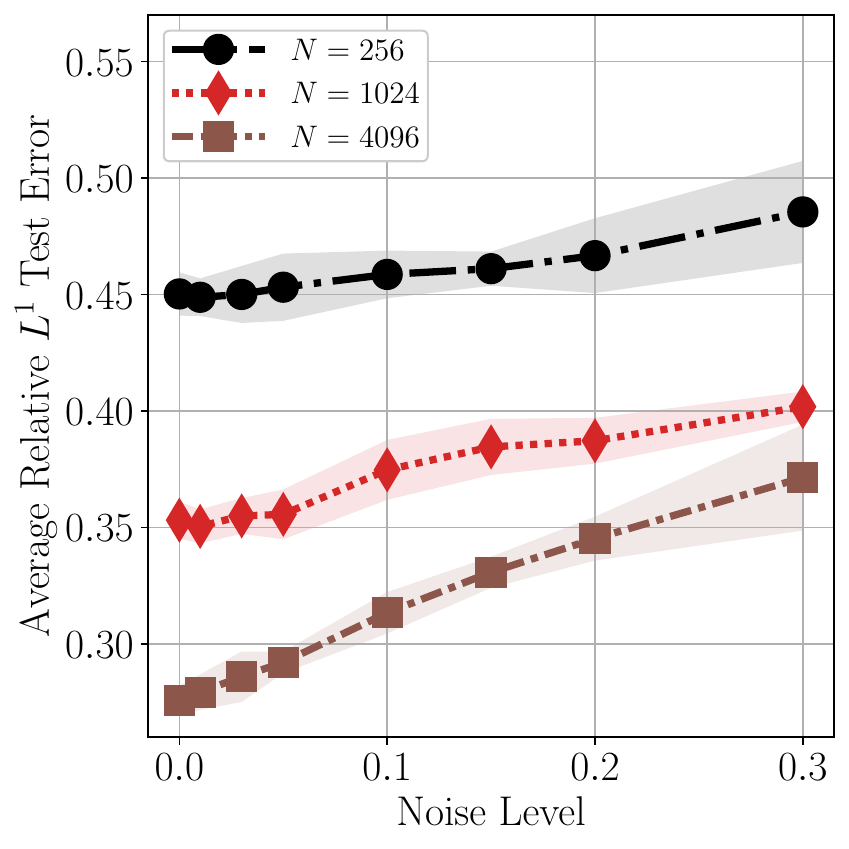}
		\caption{Shape detection}
		\label{subfig:noise_sweep_three_phase_clean}
	\end{subfigure}
	\begin{subfigure}[]{0.325\textwidth}
		\centering
		\includegraphics[width=\textwidth]{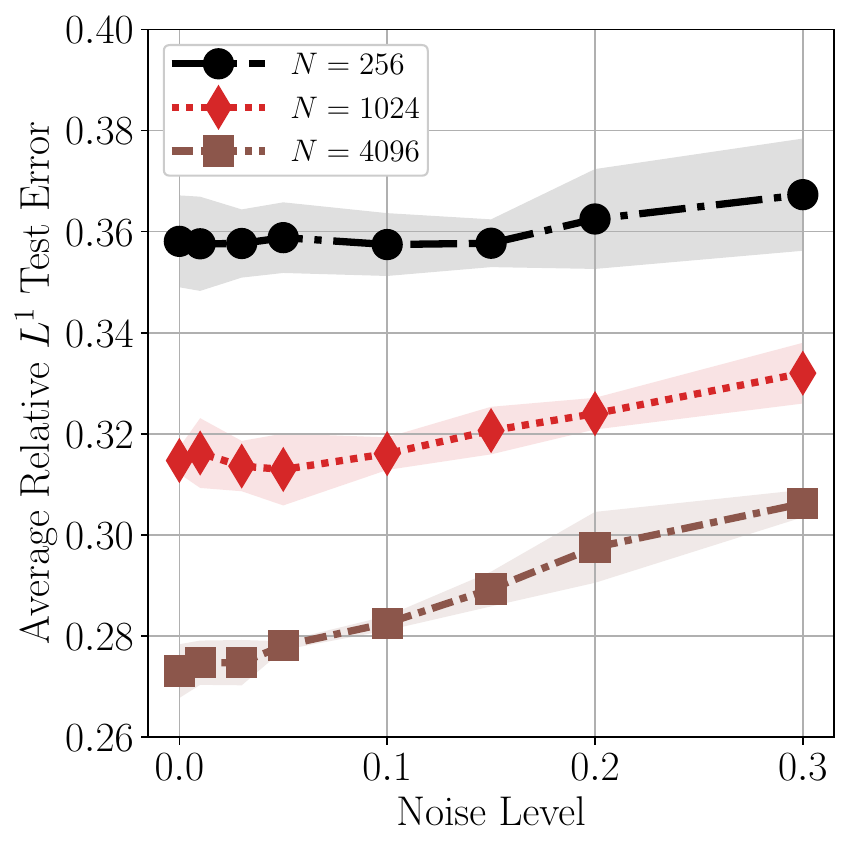}
		\caption{Three phase}
		\label{subfig:noise_log_three_phase_clean}
	\end{subfigure}
    \begin{subfigure}[]{0.32\textwidth}
		\centering
		\includegraphics[width=\textwidth]{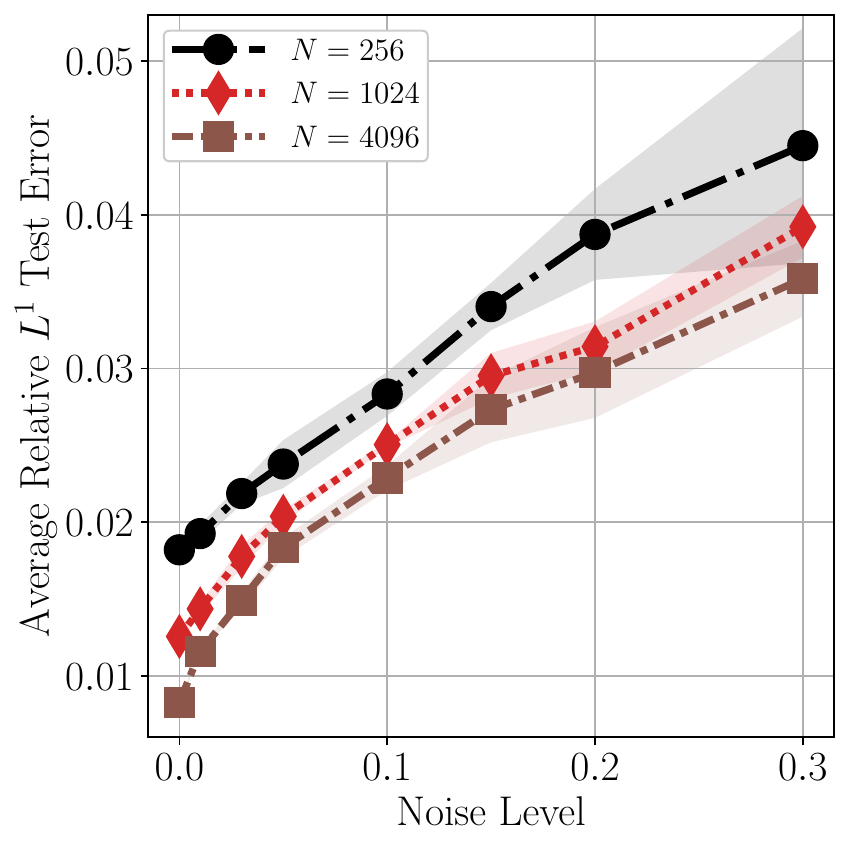}
		\caption{Lognormal}
		\label{subfig:data_sweep_three_phase_clean}
	\end{subfigure}
    \caption{Noise robustness for all three datasets under \emph{clean test data}.}
	\label{fig:noise_all_clean}
\end{figure}

\begin{figure}[tb]
    \centering
    \captionsetup{skip=10pt}
    \includegraphics[width=\textwidth]{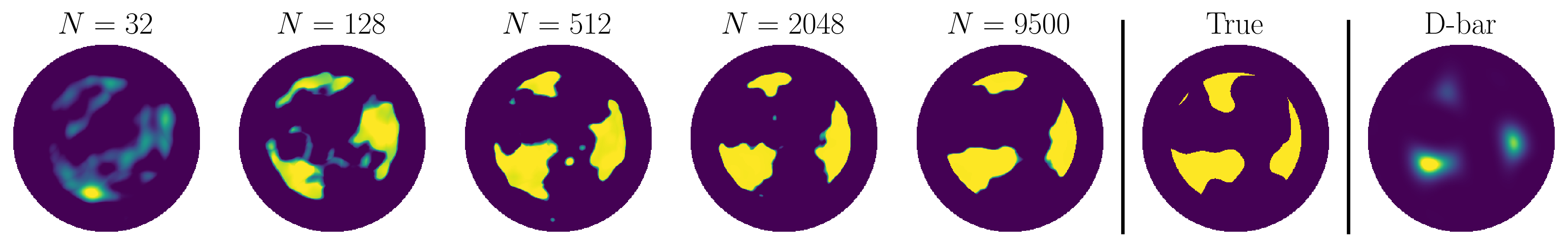}\\[-0.3\baselineskip]
    \includegraphics[width=\textwidth]{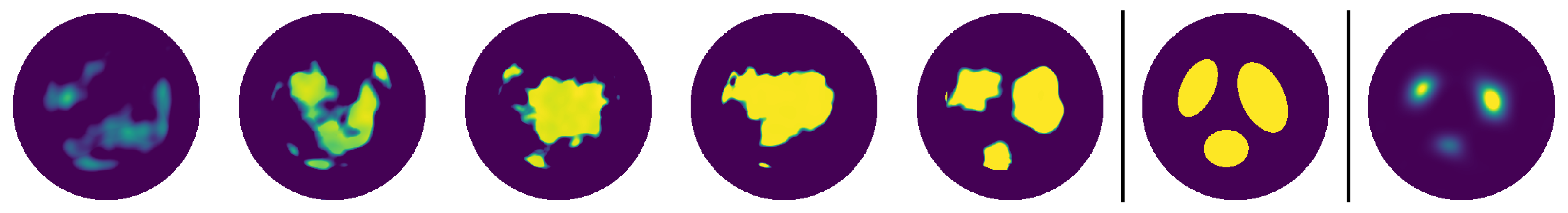}\\[-0.3\baselineskip]
    \includegraphics[width=\textwidth]{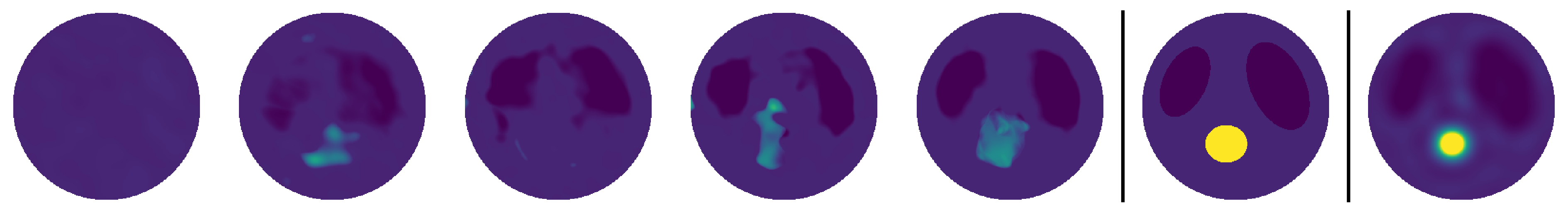}
    \caption{FNO reconstruction of three discontinuous conductivities as training dataset size $N$ increases. D-bar reconstructions are shown in comparison. The data are noise-free.}
    \label{fig:compare_dbar_clean}
\end{figure}

\Cref{fig:compare_dbar_clean} is analogous to \cref{fig:compare_dbar}, but now with both noise-free training and test data instead of just noise-free training data. The first two rows correspond to FNOs trained on the shape detection dataset and the last row to the three phase inclusion dataset. The conductivity in the first row is a sample from the test set, while the latter two are deterministic heart and lungs phantoms that are ``out-of-distribution.'' Due to perfect boundary data (up to discretization error), the D-bar method performs much better than shown in \cref{fig:compare_dbar} and now detects both high and low conductivity inclusions correctly. However, the results are still substantially oversmoothed compared to the crisp FNO reconstructions with large enough $N$.
Both \cref{fig:compare_dbar,fig:compare_dbar_clean} use the bounded noise model \eqref{eqn:kle_2d_boundary} for test data corruption. The regularized D-bar method implementation is discussed later in this appendix.

\begin{figure}[tb!]
    \centering
    \includegraphics[width=\textwidth]{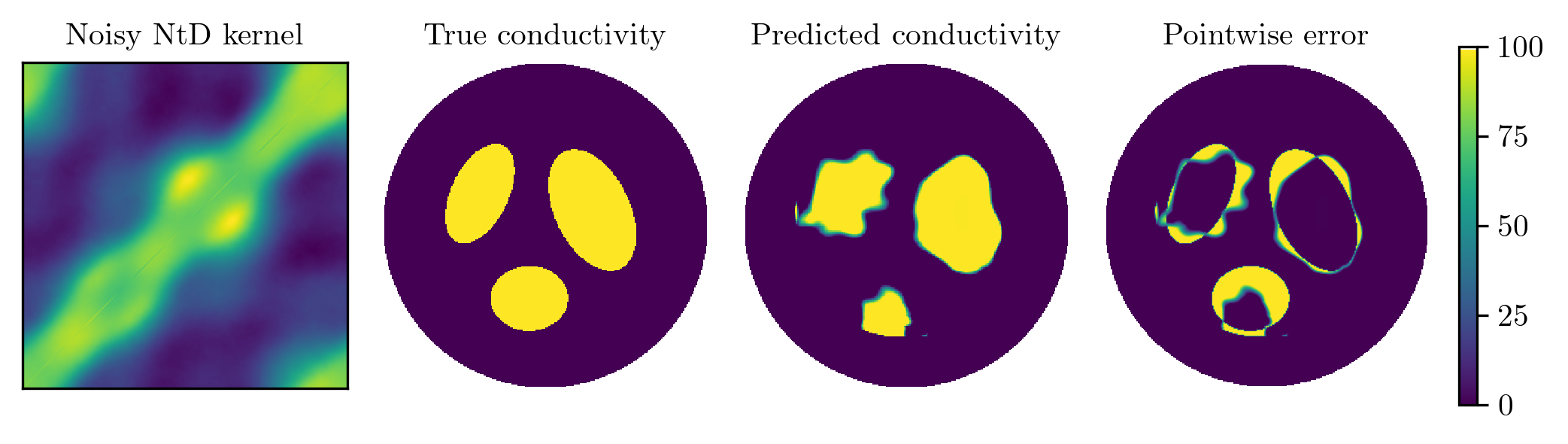}
    \caption{Heart and lungs phantom reconstruction with no training noise and $1\%$ test noise.}
	\label{fig:heart_shape_noisy}
\end{figure}

In \cref{fig:tile_shape,fig:tile_three_phase,fig:tile_lognormal}, we train on Gaussian process noise perturbations to NtD kernel boundary data and test on boundary data corrupted by noise constructed via random series with uniformly distributed coefficients \eqref{eqn:kle_2d_boundary}. This is already a form of input distribution shift that the FNO can easily handle. We further evaluate the out-of-distribution capabilities of the trained FNOs in \cref{subfig:heart_three_phase,fig:compare_dbar,fig:compare_dbar_clean,subfig:heart_three_phase_noisy,fig:heart_shape_noisy} for deterministic heart and lungs phantoms \cite[Chp.~12]{mueller2012linear} unseen during training.

\begin{figure}[tbhp!]
	\centering
    \includegraphics[width=0.99\textwidth]{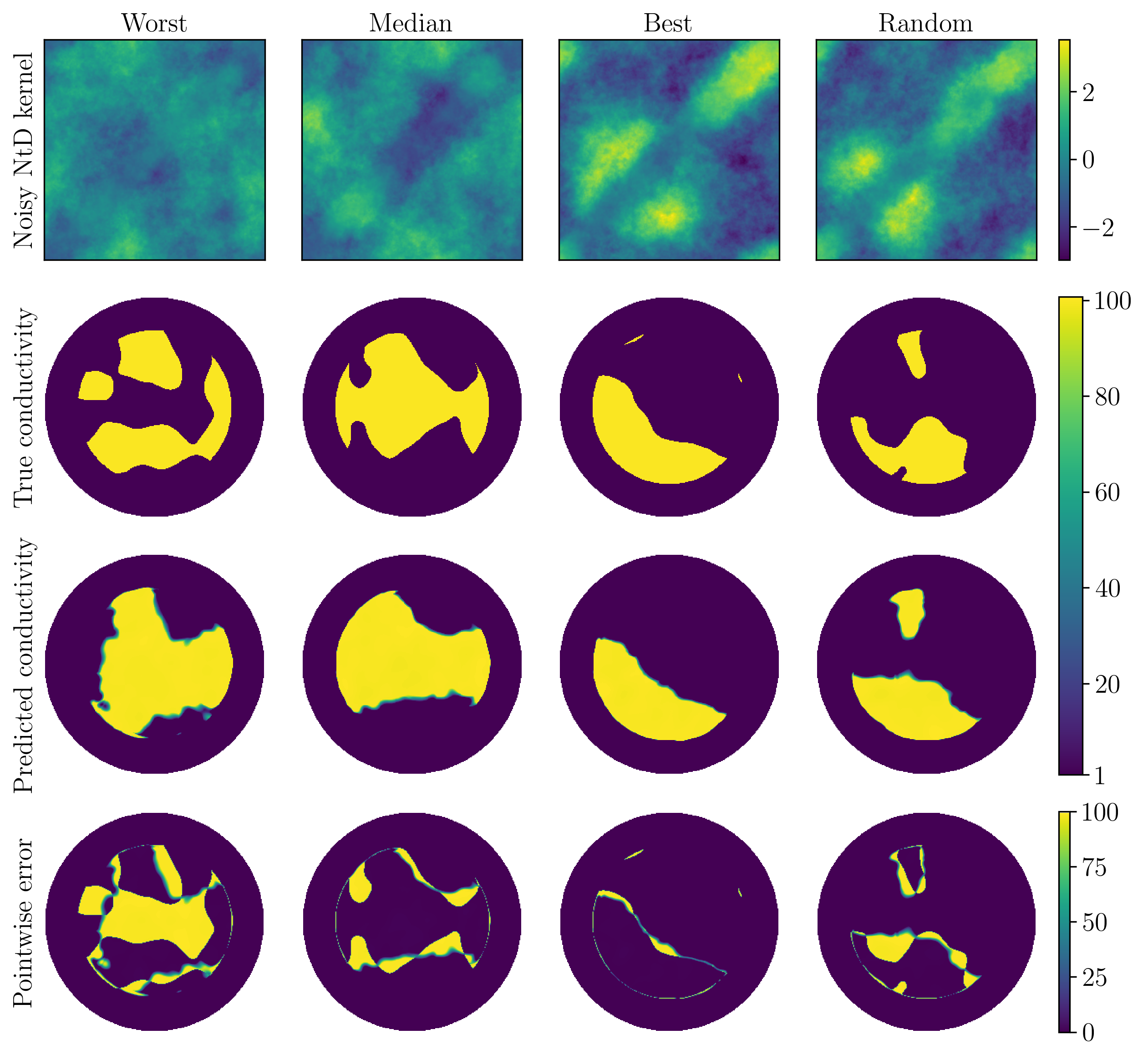}
    \caption{Sample reconstructions from the shape detection dataset, with max, median, and min errors with respect to the $L^0$ ``distance.'' Training noise is $30\%$ and test noise is $20\%$.}
	\label{fig:tile_shape_30noise20}
\end{figure}

We conclude our discussion by visualizing reconstructions obtained with very high training and test data noise levels. 
We also tried training on clean data and testing on highly noisy NtD kernels, but the reconstructions were generally poor and contained numerical artifacts.
\Cref{fig:tile_shape_30noise20} is analogous to \cref{fig:tile_shape} for shape detection, but now with $30\%$ training noise and $20\%$ testing noise. In the top row of noisy NtD kernels, the individual features are almost completely washed out by the severe noise. Nonetheless, reconstructions have acceptable accuracy. Row three columns one and two shows that the reconstructions are regularized: the finer scale cavities in the true conductivities are filled in smoothly. The average test set scores are $0.357$ $L^1(\D)$ error, $0.267$ $L^0(\D)$ distance, and $0.825$ Dice score. The clean test data errors were not much better at $0.348$ $L^1(\D)$ error, $0.258$ $L^0(\D)$ distance, and $0.836$ Dice score. \Cref{tab:shape_scores_all} provides the scores for the four individual reconstructions shown in \cref{fig:tile_shape_30noise20}.

\begin{table}[tb]%
    \captionsetup{width=\textwidth,skip=10pt}
    \caption{Relative errors in $L^1$ norm, ``$L^0$ norm'', or Dice score for the shape detection reconstructions in \cref{fig:tile_shape,fig:tile_shape_30noise20}. Higher Dice scores are better, unlike the norms.}
    \label{tab:shape_scores_all}
    \centering
    \renewcommand{\arraystretch}{1.2}
    \begin{tabular}{llcccc}
        \toprule
        & & \multicolumn{4}{c}{Error of Test Sample} \\
        \cmidrule(l){3-6}
        Loss & Figure & Worst & Median & Best & Random \\
        \midrule
        \multirow{2}{*}{$L^1$} 
            & \cref{fig:tile_shape} &  $0.778$ & $0.244$ & $0.041$ & $0.533$\\
            & \cref{fig:tile_shape_30noise20} &  $1.041$ & $0.329$ & $0.071$ & $0.605$\\
        \midrule
        \multirow{2}{*}{$L^0$} 
            & \cref{fig:tile_shape} &  $0.495$   & $0.187$ & $0.041$ & $0.197$ \\
            & \cref{fig:tile_shape_30noise20} & $0.643$  & $0.252$& $0.054$ & $0.199$ \\
        \midrule
        \multirow{2}{*}{Dice} 
            & \cref{fig:tile_shape} &  $0.522$ & $0.874$ & $0.980$ & $0.953$\\
            & \cref{fig:tile_shape_30noise20} & $0.380$ & $0.834$ & $0.965$ & $0.889$ \\
        \bottomrule
    \end{tabular}
\end{table}

\Cref{fig:tile_three_phase_30noise20} tells a similar story, except this time the regularization due to high training noise actually leads to visually better test set reconstructions than in \cref{fig:tile_three_phase}, at least for the four samples shown in \cref{subfig:tile_three_phase_noisy}. The average relative $L^1(\D)$ error over the whole test dataset is $0.300$. The individual worst, median, best, and random test errors are $0.604$, $0.294$, $0.039$, and $0.348$, respectively. The heart and lung reconstruction in \cref{subfig:heart_three_phase_noisy} is also reasonable given the noise level.

\begin{figure}[ptbh!]
	\centering
    \begin{subfigure}[]{0.99\textwidth}
        \centering
        \includegraphics[width=\textwidth]{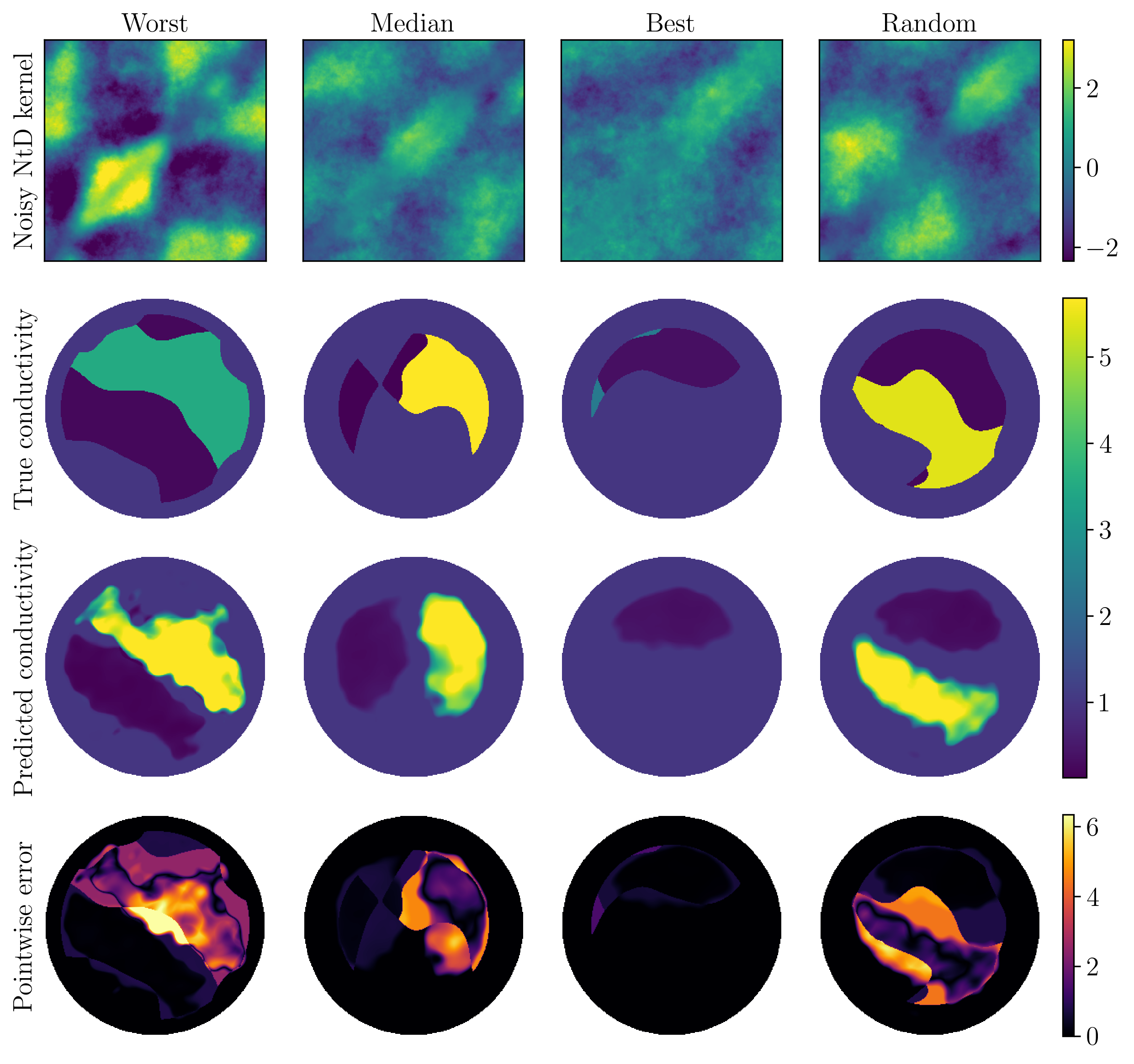}
        \caption{Sample reconstructions from the three phase inclusion dataset subject to high noise levels.}
    	\label{subfig:tile_three_phase_noisy}
    \end{subfigure}
    \begin{subfigure}[]{0.99\textwidth}
        \centering
        \includegraphics[width=\textwidth]{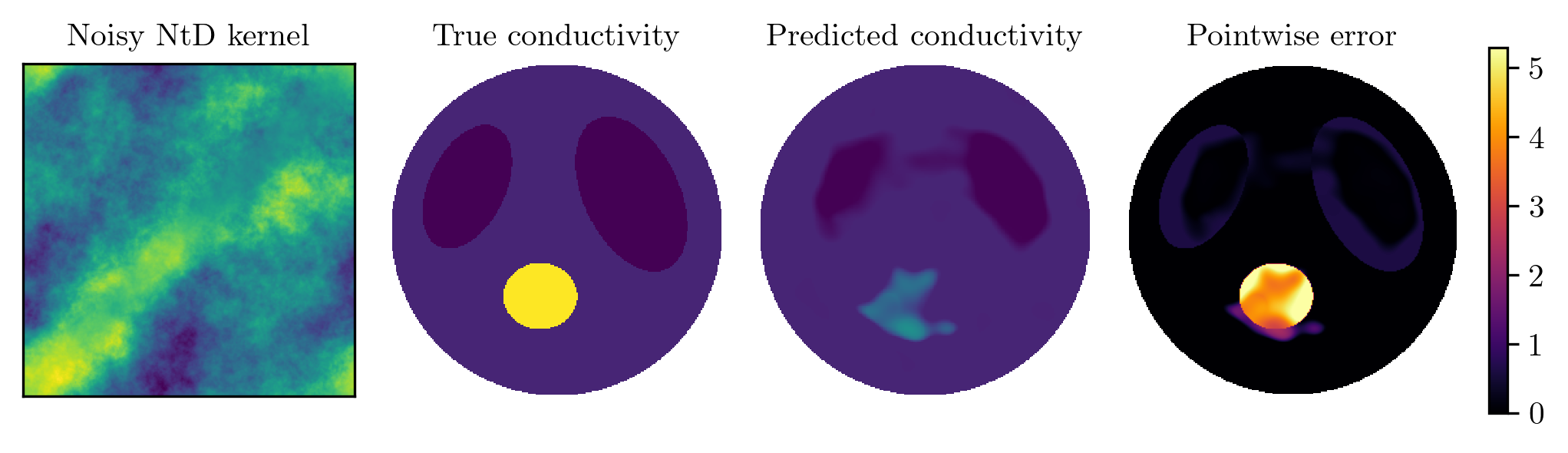}
        \caption{Same as \cref{subfig:tile_three_phase_noisy} for a deterministic heart and lungs phantom.}
        \label{subfig:heart_three_phase_noisy}
    \end{subfigure}
    \caption{Reconstruction of test set samples (top row) and a synthetic heart and lungs phantom (bottom row) subject to $30\%$ training noise and $20\%$ test noise.}
    \label{fig:tile_three_phase_30noise20}
\end{figure}

For the lognormal conductivity dataset, we reduced to $10\%$ training noise and $6\%$ test noise due to its increased sensitivity. \Cref{fig:tile_lognormal_10noise6} contains the results. The top row of noisy NtD kernels does not prevent the FNO from returning accurate reconstructions that capture the large scale features correctly. Pointwise errors are also reasonable and have slightly larger lengthscale features than the analogous row four in \cref{fig:tile_lognormal}. The $L^1$ error over the whole test set is $0.023$. Individually, the worst, median, best, and random test errors are $0.055$, $0.022$, $0.010$, and $0.023$, respectively.

\begin{figure}[tbhp!]
	\centering
    \includegraphics[width=0.99\textwidth]{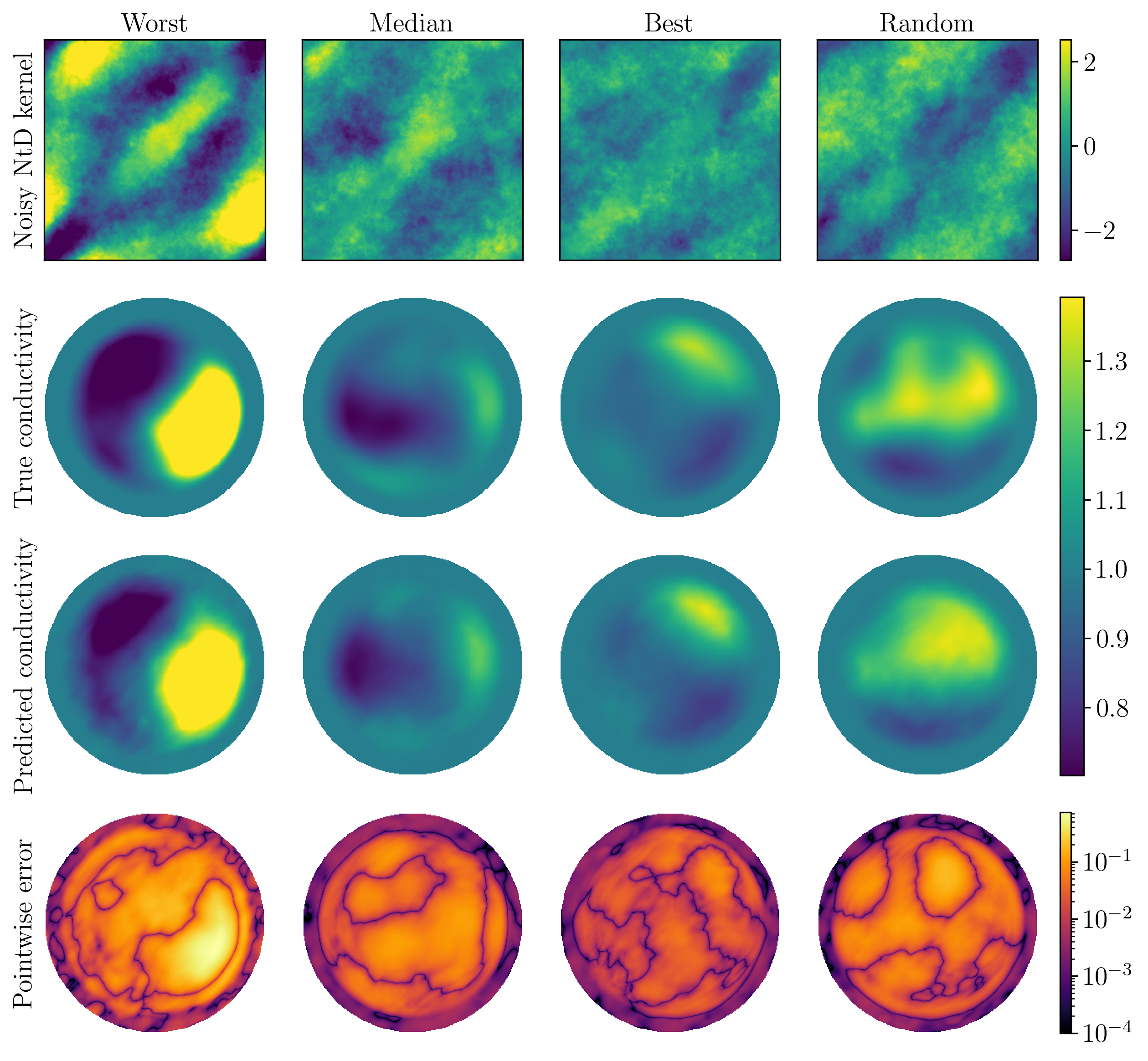}
    \caption{Sample reconstructions from the lognormal conductivity dataset. The training noise is $10\%$ and the test noise is $6\%$.}
	\label{fig:tile_lognormal_10noise6}
\end{figure}

Due to the end-to-end training, the results in \cref{fig:tile_shape_30noise20,fig:tile_three_phase_30noise20,fig:tile_lognormal_10noise6} highlight the ability of the FNO to handle much higher noise levels than traditional inverse solvers.

\subparagraph*{\emph{\textbf{Dataset generation.}}}
Each of the three datasets (hereafter referred to as ``Shape'', ``Three Phase'', and ``Lognormal'', respectively) corresponding to \cref{sec:numerics_shape,sec:numerics_three,sec:numerics_lognormal} is comprised of $10^4$ pairs of conductivities and NtD kernels. The forward map $\gamma\mapsto\ntd$ is implemented with a P1 finite element solver. The conductivities are sampled from the Shape, Three Phase, or Lognormal data distributions (details to follow) on a $256\times 256$ uniform grid representing $(-1,1)^2$, restricted to $\D\subset(-1,1)^2$, and interpolated to the finite element mesh. For each $\gamma$, we approximate $\ntd$ as follows. For every $j\in \sJ_{256}\defeq \{-127,-126,\ldots, 127, 128\}\setminus\{0\}$, solve \eqref{eqn:elliptic} with Neumann data $g=\varphi_j$ to obtain the Dirichlet data $\ntd\varphi_j$. Using these 255 PDE solves, we construct an NtD matrix with entries $\ip{\varphi_j}{\ntd\varphi_k}_{L^2(\T;\C)}$ in Fourier basis coordinates. The kernel $\kappa_\gamma$ is then given by
\begin{align}\label{eqn:app_kernel_from_matrix}
    \kappa_\gamma(\theta,\theta')\defeq \sum_{j\in\sJ_{256}}\sum_{k\in\sJ_{256}}\ip{\varphi_j}{\ntd\varphi_k}_{L^2(\T;\C)}\, \varphi_j(\theta)\overline{\varphi_k(\theta')}
\end{align}
for $(\theta,\theta')\in\T^2$. 

To model noisy boundary data, for an NtD kernel $\kappa_\gamma$ we define the perturbed kernel $\kappa_\gamma^{(\delta)}$ for a noise level $\delta$ by \eqref{eqn:ntd_noise_model}. In particular, the noise is additive, but not independent due to its dependence on $\gamma$ through $\kappa_\gamma$.
The noise factor $\xi\in L^2_\diamond(\T^2;\R)$ in the display \eqref{eqn:ntd_noise_model} is defined in law by the real (or imaginary) part of random series expansions of the form \eqref{eqn:kle_2d_boundary} with the real and imaginary parts of $\{\zeta_{ij}\}$ either i.i.d. $\Unif[-\sqrt{3},\sqrt{3}]$ or $\normal(0,1)$ random variables. Let $\al>1$ and $\tau>0$. The coefficients $\{c_{ij}\}$ in \eqref{eqn:kle_2d_boundary} are defined for $(i,j)\neq (0,0)$ by
\begin{align}
    c_{ij}\defeq \tau^{2\al-2}(4\pi^2(i^2+j^2)+\tau^2)^{-\al}
\end{align}
and $c_{00}=0$. We take $\al=1.5$ and $\tau=10$. This implies that $\supp(\Law(\xi))$ is compact in $L^2(\T^2)$ by the discussion following \eqref{eqn:kle_2d_boundary} in \cref{sec:approx_discuss}. 

We now describe the random field models for the three types of conductivities studied in \cref{sec:numerics}.
Let $\mu_{\tau,\al}=\normal(0,\cC_{\tau,\al})$ be a Gaussian measure on $L^2_\diamond((-1,1)^2)$ with covariance operator
\begin{align}\label{eqn:app_covariance_matern}
    \cC_{\tau,\al}\defeq (\tau/2)^{2\al-2}\bigl(-\Delta + (\tau/2)^2\Id\bigr)^{-\al}\,.
\end{align}
The Laplacian in the preceding display is defined with homogeneous Neumann boundary conditions.
The operator $\cC_{\tau,\al}$ is diagonalized in the Fourier cosine basis. Thus, samples from $\mu$ can be generated efficiently using the FFT; see \cite[Sec.~4]{nelsen2024operator}. Let $\rchi_A$ denote the indicator function of a set $A$ and $A^\comp\defeq (-1,1)^2\setminus A$ be the complement of $A$ with respect to $(-1,1)^2$. The three conductivity distributions are given as follows.

\begin{enumerate}[label=(\roman*),leftmargin=1.75\parindent,topsep=1.67ex,itemsep=0.5ex,partopsep=1ex,parsep=1ex]
    \item \sfit{(shape detection)} Fix $\tau=20$, $\al=4.5$, and $r=0.7$.
    Let $u\sim\mu_{\tau,\al}$. Let $U\defeq {\{x\in(-1,1)^2\colon u(x)\geq 0\}}$ be the induced zero super level set. Define
    \begin{align}\label{eqn:app_gamma_shape}
        \gamma\defeq \Bigl(\bigl(100\,\rchi_{U} + \rchi_{U^\comp}\bigr)\rchi_{B_r(0)} + \rchi_{B_r(0)^\comp}\Bigr)\Big|_{\D}
    \end{align}
    to be a draw from the Shape conductivity data distribution. The indicators involving $\Omega'\defeq B_r(0)$ ensure that $\gamma\in\Gamma'$. It follows that $\gamma$ takes values in the set $\{1, 100\}$ almost surely, and thus the contrast ratio equals $100$.

    \item \sfit{(three phase inclusions)} Independently draw $\tau\sim\Unif[12,18]$, $\al\sim\Unif[4,5]$, $r\sim\Unif[0.65,0.9]$, $c_1\sim\Unif[2,10]$, and $c_2\sim\Unif[0.1,0.5]$. Let $u_1$ and $u_2$ be independent draws from $\mu_{\tau,\al}$. Let $U_i\defeq \{x\in(-1,1)^2\colon u_i(x)\geq 0\}$ for $i\in\{1,2\}$ be the induced zero super level sets. Define
    \begin{align}\label{eqn:app_gamma_three_phase}
        \gamma\defeq \Bigl(\bigl(c_1\rchi_{U_1\cap U_2} + c_2\rchi_{U_1\cap U_2^\comp} + \rchi_{U_1^\comp}\bigr)\rchi_{B_r(0)} + \rchi_{B_r(0)^\comp}\Bigr)\Big|_{\D}
    \end{align}
    to be a draw from the Three Phase conductivity data distribution. The term in inner parentheses creates a vector level set field with three values $c_1>0$, $c_2>0$, and $1$ (the background value). The indicators involving $\Omega'\defeq B_r(0)$ from \eqref{eqn:compact_support_set} ensure that $\gamma\in\Gamma'$ in \eqref{eqn:conductivity_set_prelim}. It follows that $\gamma$ takes values in the (random) set $\{1, c_1, c_2\}$ almost surely. The contrast ratio is bounded above by $100$ almost surely.

    \item \sfit{(lognormal conductivities)} For $w\geq 0$, Define the cutoff function~\cite{miller2006applied} $\rho_{\lambda, r_{\pm}}(\slot;w)\colon [0,1]\to [0,1]$ by
    \begin{align}
        t\mapsto \rho_{\lambda, r_{\pm}}(t;w)\defeq
        \begin{cases}
           1  & \text{if }\, 0\leq t\leq r_{-} \,,\\
                      \frac{1+w}{2} + \frac{1-w}{2}\tanh\Bigl(\frac{(t-r_{-})^{-1} + (t-r_+)^{-1}}{\lambda}\Bigr)  & \text{if }\, r_{-}<t\leq r_+\,,\\
        0\, , & \text{otherwise}\,.
        \end{cases}
    \end{align}
    Independently draw $\tau\sim\Unif[7,9]$, $\al\sim\Unif[3,4]$, $r_-\sim\Unif[0.50,0.55]$, $r_+\sim\Unif[0.85,0.95]$, and $\lambda\sim \Unif[7.5,8.5]$. Draw $u\sim \mu_{\tau,\al}$ independently. Let 
    \begin{align}
        v\defeq \exp(u)\rchi_{B_{r_+}(0)} +\rchi_{B_{r_+}(0)^\comp} 
    \end{align}
    be defined pointwise, which is well-defined by a Sobolev embedding because $\al\geq 3>1=d/2$ almost surely. Let the function $\gamma$ defined for each $x\in\D$ by
    \begin{align}\label{eqn:app_gamma_lognormal}
        \gamma(x)\defeq v(x) \rchi_{(B_{r_+}(0)\setminus B_{r_-}(0))^\comp}(x) + v(x)\rho_{\lambda, r_{\pm}}\bigl(\abs{x};\tfrac{1}{v(x)}\bigr)\rchi_{(B_{r_+}(0)\setminus B_{r_-}(0))}(x)
    \end{align}
    be a draw from the Lognormal conductivity data distribution. The cutoff function ensures that $\gamma\equiv 1$ in a neighborhood of the boundary of the disk.
\end{enumerate}
Conductivities $\gamma$ drawn from the Shape and Three Phase data distributions belong to $\BV(\D)$ almost surely. Indeed, excursion sets of smooth enough Gaussian processes have finite perimeter \cite{adler2007random}. The fields $\gamma$ are linear combinations of excursion set indicator functions. The perimeter of an excursion set equals the total variation of its indicator \cite[Chp.~14]{leoni2017first}. However, such an argument does not guarantee uniform BV norm bounds as in \eqref{eqn:conductivity_set}. This is an interesting technical point to further explore.

\subparagraph*{\emph{\textbf{Implementation details.}}}
We use a standard implementation of FNO \cite{li2020fourier,huang2025operator}. The series in \eqref{eqn:fno_layer} is discretized by truncating to $J$ terms per dimension. The hyperparameter choices include $L=2$ layers, $J=12$ terms for a total of $144$ retained Fourier modes, ReLU activation for layer one and identity activation for layer two, and hidden channel width of $d_\mathrm{c}=48$. The final MLP defining $\cQ$ in \eqref{eqn:fno_torus} has width 256 and ReLU activation. The lifting layer $\cS$ is pointwise affine. These architectural choices are obtained from a small hyperparameter grid search. In particular, the small layer size of $L=2$ helps to reduce overfitting in this ill-posed inverse problem. We use standard input grid concatenation and zero padding to and from the latent torus $\T^2$. The total FNO parameter count is approximately 2.7 million.

We train the FNOs using stochastic gradient descent with batch size 32 and the AdamW optimizer with initial learning rate $8\times 10^{-3}$ and weight decay of $10^{-4}$. The learning rate decays according to a standard cosine annealing scheduler over 250 total epochs. The loss function is the relative $L^1(\D)$ loss defined in \eqref{eqn:erm} with $\epsilon\defeq 10^{-8}$. We enforce that the FNO map into functions defined on domain $\D$ by masking $(-1,1)^2$. All loss calculations use this mask to assign zero error on $(-1,1)^2\setminus\D$.

The test set is fixed with $400$ samples. We also hold out $100$ fixed samples as validation data to select the best model (over $250$ training epochs) for testing. All models are trained on input and output resolution $128\times 128$ and tested on resolution $256\times 256$ to avoid an inverse crime \cite{kaipio2005statistical}. This resolution transfer is trivial with the FNO due to its resolution invariance. For both the training and test data, we use standard pointwise empirical mean and variance data standardization for the input NtD kernel functions and no normalization for the output conductivities. All visualizations show \emph{normalized} NtD kernels to emphasize sample diversity.

Our implementation of the D-bar method used in \cref{fig:compare_dbar,fig:compare_dbar_clean} is based on publicly available MATLAB code \cite[Chp.~15]{mueller2012linear}. We discretize the unit circle with $128$ equally-spaced grid points. We discretize the complex plane on a $128\times 128$ grid and restrict to the disk of radius $10$. The NtD maps are discretized using Fourier modes $j\in\mathcal{J}_{128}$. The D-bar equation and the reconstructions are performed on $256\times 256$ resolution grids. To regularize, we truncate the scattering transform to radius 4.5 (7.7) for the noisy (resp., clean) shape detection test sample, 4.9 (7.9) for the noisy (resp., clean) shape detection heart and lungs phantom, and 5.2 (9.1) for the noisy (resp., clean) three phase heart and lungs phantom.
Each D-bar reconstruction requires one to two hours of wall-clock time on a $256\times 256$ grid. These timings could be improved by parallelizing the GMRES real-linear D-bar solves via batching over points $x\in\Omega$.

\subparagraph*{\emph{\textbf{Computational resources.}}}
The majority of all computations are performed on a machine equipped with an NVIDIA GeForce RTX 4090 GPU (24 GB VRAM) running CUDA version 12.4, Intel i7-10700K CPU (8 cores/16 threads) running at 4.9 GHz, and 64 GB of DDR4-3200 CL16 RAM. 
We generate data in MATLAB (R2019b, Update 6) using the Partial Differential Equation Toolbox and Image Processing Toolbox.
The data generation for shape detection and three phase inclusions each required 168 hours of wall-clock time on a finite element mesh with 262144 P1 elements and 131585 nodes. The lognormal conductivity data generation required 42 hours of wall-clock time on a finite element mesh with 65536 P1 elements and 33025 nodes.
All FNO training experiments are performed in Python using the PyTorch framework in \texttt{float32} single precision.
For each of the three datasets, the total training and evaluation time for the results presented in \cref{sec:numerics} is 48 hours of wall-clock time when accelerated with one NVIDIA GeForce RTX 4090 GPU.

\end{appendices}

\bibliography{references}

\end{document}